\documentclass[10pt,reqno]{amsart}
\usepackage[top=.9in, left=.9in, right=.9in, bottom=.8in]{geometry}
\usepackage[font=small,labelfont=bf]{caption}
\usepackage{graphicx,wrapfig}
\usepackage{inputenc,amsfonts,amsthm,amscd,enumerate,amsmath, txfonts}
\usepackage{appendix}
\usepackage{hyperref}
\usepackage{color}
\numberwithin{equation}{section}
\usepackage{xfrac, nicefrac}
\usepackage{tikz}
\usepackage{comment}
\usepackage{cleveref}
\newtheorem{theorem}{Theorem}[section]

\newtheorem{lemma}[theorem]{Lemma}
\newtheorem{remark}[theorem]{Remark}
\newtheorem{proposition}[theorem]{Proposition}

\newtheorem{definition}[theorem]{Definition}

\newtheorem{assumption}[theorem]{Assumption}
\usepackage{nccmath, amsmath}

\def\g{\gamma}
\def\de{\delta}
\def\e{\varepsilon}

\newcommand{\ddt}{\frac{d}{dt}}

\newcommand{\R}{{\mathbb{R}}}

\newcommand{\A}{{\mathcal{A}}}

\newcommand{\dist}{\mathrm{dist}}

\newcommand{\Loj}{\L ojasiewicz}
\newcommand{\Dom}{[a_1,a_2]}
\newcommand{\B}{\mathcal{B}}
\newcommand{\Aarea}{\mathcal{B}_{\eta}}

\newcommand{\Circ}{\mathrm{Circ}}
\newcommand{\Circeta}{\mathrm{Circ}_\eta}

\newcommand{\Crit}{\mathcal{C}^\ast_\eta}

\newcommand{\Min}{\mathcal{M}^\ast_\eta}

\newcommand{\pd}{\partial_\de}

\newcommand{\kappamaxSi}{{\kappa_{max}(\Sigma)}}
\newcommand{\AreaSi}{A_{\Si}}
\newcommand{\Areahat}{A_{\hat\Si}}
\newcommand{\si}{\sigma}
\newcommand{\na}{\nabla}
\newcommand{\Om}{\Omega}
\newcommand{\Si}{\Sigma}
\newcommand{\ga}{\gamma}
\newcommand{\eps}{\varepsilon}
\newcommand{\al}{\alpha}
\newcommand{\half}{{\frac12}}
\newcommand{\thalf}{{\tfrac12}}
\newcommand{\beq}{\begin{equation}}
\newcommand{\eeq}{\end{equation}}
\newcommand{\beqs}{\begin{equation*}}
\newcommand{\eeqs}{\end{equation*}}
\newcommand{\beqa}{\begin{equation}\begin{aligned}}
\newcommand{\eeqa}{\end{aligned}\end{equation}}
\newcommand{\beqas}{\begin{equation*}\begin{aligned}}
\newcommand{\eeqas}{\end{aligned}\end{equation*}}
\newcommand{\LL}{\mathcal{L}}
\newcommand{\LLeta}{{\mathcal{L}_\eta}}
\newcommand{\Criteta}{\mathcal{C}^\ast_\eta}
\newcommand{\abs}[1]{\vert#1\vert} 
\newcommand{\pihalf}{{\tfrac{\pi}{2}}}
\newcommand{\norm}[1]{\Vert#1\Vert} 
\newcommand{\peps}{\partial_{\eps}}
\newcommand{\ddeps}{\frac{d}{d\eps}}

\newcommand{\areaC}{\bar{\mathfrak{c}}}
\newcommand{\IP}{{I_{\Om}}}
\newcommand{\setmins}{\mathcal{M}_\eta^{\Omega}}

\newcommand{\phiturn}{\phi_{\text{turn}}}

\title[Quantitative estimates for the relative isoperimetric problem and its gradient flow in the plane]{Quantitative estimates for the relative isoperimetric problem and its gradient flow outside convex bodies in the plane}
\author{Elena M\"{a}der-Baumdicker}
\address{Elena M\"ader-Baumdicker: Institute of Mathematics, Freie Universit\"at Berlin, Arnimallee 3, 14195 Berlin, Germany, \textit{elena.maeder-baumdicker@fu-berlin.de}}

\author{Robin Neumayer}
\address{Robin Neumayer: Carnegie Mellon University, 5000 Forbes Avenue, Pittsburgh, PA 15201, USA, \textit{neumayer@cmu.edu}}

\author{Jiewon Park}
\address{Jiewon Park: Korea Advanced Institute of Science and Technology (KAIST), 291 Daehak-ro, Yuseong-gu, 34141 Daejeon, South Korea, \textit{jiewonpark@kaist.ac.kr}}

\author{Melanie Rupflin}
\address{Melanie Rupflin: Mathematical Institute, University of Oxford, Oxford OX2 6GG, United Kingdom,  \textit{rupflin@maths.ox.ac.uk}} 


\begin{document}
\maketitle
\begin{abstract}
    We prove three related quantitative results for the relative isoperimetric problem outside a convex body $\Om$ in the plane: (1) \Loj\ estimates and quantitative rigidity for critical points, (2) rates of convergence for the gradient flow, and (3) quantitative stability for minimizers. These results come with explicit constants and optimal exponents/rates, and hold whenever a simple  two-dimensional auxiliary variational problem for circular arcs outside of $\Om$ is nondegenerate. The proofs are inter-related, and in particular, for the first time in the context of isoperimetric problems, a flow approach is used to prove quantitative stability for minimizers. 
\end{abstract}

\section{Introduction}

Two central themes in the study of geometric variational problems are stability and dynamics: how does the energy grow near a minimizer or critical point, and how rapidly does the associated gradient flow converge to equilibrium? It is well known that both properties are closely linked to nondegeneracy, and when the second variation at any critical point is nondegenerate modulo symmetries, one expects a quadratic stability estimate for minimizers and exponential convergence to equilibrium of the gradient flow.
 Verifying nondegeneracy, however, is in practice often delicate or intractable, especially in geometric settings where the space of competitors is constrained by curvature, topology, or boundary behavior. 

This paper investigates these themes in the context of the exterior isoperimetric problem in the plane. Given a convex body  $\Omega \subset \mathbb{R}^2$, i.e. a compact, convex set with nonempty interior, with $C^2$ boundary $\Si$ and a prescribed area $\eta > 0$, consider the isoperimetric problem
\begin{equation}
\label{eqn: isop problem def}
\IP(\eta) = \inf\left\{P(E; \mathbb{R}^2 \setminus \Omega) \ :\ E \subset \mathbb{R}^2 \setminus \Omega,\ \ |E| = \eta \right\}.
\end{equation}
Here $P(E; \mathbb{R}^2 \setminus \Omega)$ is the relative perimeter, which is equal to $\mathcal{H}^1(\partial E \setminus \Si)$ when $E$ has $C^1$ boundary; see Section~\ref{sec: preliminaries}.

An important aspect of this problem is that we can reduce the question of the nondegeneracy of the second variation of critical points for the infinite-dimensional problem \eqref{eqn: isop problem def} to the nondegeneracy of the Hessian for critical points of an explicit {two-dimensional} variational problem over circular arcs; see \Cref{ass:LojassIntro} below. 
For any convex body $\Omega$ and area constraint $\eta>0$ for which this two-dimensional nondegeneracy condition holds, we establish a suite of sharp results predicted by the second variation theory: (1)  quantitative rigidity and \L ojasiewicz estimates, 
 (2) exponential convergence of the associated gradient flow, 
and (3) quadratic stability for minimizers.

We prove these three main results in an interconnected manner. In particular, to address quantitative stability---that is, the question of whether a set almost achieving the infimum in \eqref{eqn: isop problem def} must be quantitatively close to a minimizer---we evolve the boundary of a set by the free boundary area-preserving curve shortening flow and then integrate out a \L ojasiewicz-type estimate along the trajectory.
Flow-based approaches to proving quantitative stability have been developed in recent years in the contexts of maps from $\mathbb{S}^2$ to $\mathbb{S}^2$ \cite{Topping23, Rupflin23}
and Sobolev-type inequalities \cite{BonforteEtAl}, but to our knowledge this is the first application of the method in the context of isoperimetric problems. This constructive strategy yields sharp estimates with explicit constants, and puts stability, quantitative rigidity, and gradient flow convergence in a unified analytic framework.

For any $\eta>0$, the collection of minimizers $\setmins$ of \eqref{eqn: isop problem def} (among sets of finite perimeter, see Section~\ref{sec: preliminaries}) is nonempty by the direct method.  The first variation shows that
 the relative boundary $\partial E_* \setminus\Sigma$
 of any minimizer $E_* \in \setmins$ is a union of equal-radii circles and circular arcs meeting $\Sigma=\partial \Om$ orthogonally, and a simple competitor argument then ensures that $E_*$ is connected\footnote{In higher dimensions, it is not known whether isoperimetric sets are connected, while for the interior relative isoperimetric problem, they are known to be connected \cite{SternbergZumbrun}.} and intersects $\Si$ nontrivially.
 Thus, the boundary of $E_*$ is the union of a single circular arc $c^*$ meeting $\Si$ orthogonally and a subarc $\si_{c^*}$ of $\Si$.

More generally, we consider oriented immersed curves $\g$ that lie outside of $ \Om^\circ$ and intersect $\Si$ only at their endpoints $x_1(\gamma)$ and $x_2(\gamma)$; we let $\B$ denote the collection of such  curves, see \eqref{def:B}.
We extend the notion of (signed) enclosed area to any $\g\in\B$ by letting $A_\Si(\gamma):=\text{Area}(\gamma+\si_\gamma)$ for the unique subarc $\si_\gamma$ of $\Si$ for which the concatenation $\gamma+\si_\gamma$ is
contractable in $\R^2\setminus \Om^\circ$.
 For positively oriented embedded curves $\gamma \in \B,$ this of course agrees with the area of the set $E_\gamma$ whose relative boundary is given by  $\gamma$.
We let
\beq
  \B_\eta:=\{\gamma\in \B: \abs{A_\Si(\gamma)}=\eta \}\,
\eeq
be the collection of such (oriented) curves with prescribed enclosed area $A_\Si$. In addition to minimizers of the isoperimetric problem \eqref{eqn: isop problem def}, or equivalently, minimizers of the length among curves in $\B_\eta$, we consider critical points of the length functional $L$ in this class.
As area preserving variations are characterized by $\int_{\gamma_t} \partial_t \g_t \cdot \nu_{\g_t}\,ds_{\gamma_t} = 0$, c.f. \eqref{eqn: first var of area}, 
the usual first variation formula for the length can be written as  
\beq
\label{eq:dL-area-preserv}
dL(\gamma)(X)=
-\int (\kappa_\gamma-\bar \kappa_\gamma) \langle \nu_\gamma, X\rangle \, ds_\gamma+\langle X(x_2),\tau_\gamma(x_2)\rangle -\langle X(x_1),\tau_\gamma(x_1)\rangle
\eeq
for vector fields $X = \partial_t \g_t$ induced by such variations. In particular, the set of  critical points of the length functional in $\B_\eta$ is given by 
\beq \label{def:crit-points}
\Crit:=\{\gamma\in \B_\eta: \text{ circular arc intersecting $\Si$ orthogonally}\}\,.
\eeq
Here and in the sequel, we use the convention that $\tau_\g$ is the unit tangent of the oriented curve $\g$, $\nu_\g$ is the normal obtained by rotating $\tau_\g$ counterclockwise by $\pi/2$, $\bar{\kappa}_\g$ is the average of the curvature  $\kappa_\g$, and $s_\g$ is the arclength parameter of $\g$.

The quantitative results established in this paper for the aforementioned minimizers and critical points  are closely connected to our third main focus: the asymptotic behavior of the associated gradient flow, i.e.  the free boundary area-preserving curve shortening flow, or freeAPSCF, introduced by the first author in \cite{EMB2}.  
This natural area-preserving gradient flow of the length evolves families of curves $\g_t$ 
by 
\[
\partial_t \ga_t  = (\kappa_{\g_t} - \bar{\kappa}_{\g_t}) \nu_{\g_t}
\]
subject to the constraint that the curves intersect  $\Si$ orthogonally and from the outside of $\Om$ at their  endpoints $x_{1}(\gamma_t),x_2(\gamma_t)\in \Si$, see \eqref{equ:apcsf} below for the precise definition.  
In general, this flow can be quite poorly behaved: it neither preserves embeddedness nor remains entirely outside of $\Omega$, and singularities may develop in finite time. However, in \cite{EMB2}, the first author established conditions on the initial data guaranteeing that the flow exists and remains outside of $\Om$ for all time.

A fundamental question addressed in this paper is the rate of convergence to equilibrium for such solutions; see Theorem~\ref{mainthm-flow} below. The key ingredient to prove this is the quantitative control on the behavior of almost critical points of the length functional on the set $\B_\eta$; see Theorem~\ref{thm:loj}. 
This flow, in turn, will be the fundamental tool that we use to establish quantitative stability of minimizers of \eqref{eqn: isop problem def} in Theorem~\ref{thm: stability}, as it provides a natural way of deforming a curve in a way that preserves the enclosed area $A_\Si$ while decreasing the length, thereby improving the isoperimetric ratio. 

We prove results related to all three of the above problems with optimal exponents (respectively rates) in particular in the case where $\Si$ is a circle and more generally whenever $\Si=\partial \Om$ satisfies a simple non-degeneracy condition concerning the behavior of the length functional on the set of circular arcs
\beq \label{def: circ eta}
\Circeta:=\{c\in \Aarea: c \text{ subarc of  a circle intersecting } \Si \text{ transversally}\}.
\eeq
By the implicit function theorem $\Circeta$
is a smooth $2$-dimensional manifold, which we may locally parametrize by the centers $z \in \R^2$ of the defining circles of these arcs $c$. The restriction of the length functional to $\Circeta$ 
can hence locally be written as a function $z \mapsto \LL_\eta(z) = L(c)$ of two variables, compare also Section \ref{subsec:circular-arcs}.

In our main results below, we assume either that $\Si$ is a circle, or that the pair $(\Si,\eta)$ satisfies the following nondegeneracy assumption corresponding to the two-dimensional function $\LLeta$.
\begin{assumption}\label{ass:LojassIntro} Given $\eta>0$ and $\Sigma$, we ask that the Hessian $d^2\LLeta(z^*)$ of $\LLeta$ is non-degenerate at any critical point $z^*$ 
of $\LLeta$, i.e. that the eigenvalues of this symmetric $2\times 2$ matrix are non-zero. 
\end{assumption}

We show that, for a given support curve $\Sigma$ and a critical arc $c^\ast$, Assumption~\ref{ass:LojassIntro} can be expressed explicitly in terms of the opening angle of $c^\ast$, its curvature $\kappa_{c^\ast}$, and the curvatures $\kappa_\Sigma(x_{1})$ and $\kappa_\Sigma(x_2)$ at the endpoints $x_{1,2}$ of $c^\ast$; see Lemma~\ref{lemma:Hessian}. In the special case where $\kappa_\Sigma(x_{1})= \kappa_\Sigma(x_{2})$, the nondegeneracy condition admits a geometric interpretation in terms of the osculating circles of $\Sigma$ at $x_1$ and $x_2$. A direct computation shows that this condition is satisfied by a broad class of support curves $\Sigma$, for example any ellipse.

Our first main result is the following quantitative rigidity and \Loj\ estimate for critical points.

\begin{theorem}
    \label{thm:loj}
    Let  $\Om\subset\R^2$  be  a convex body with $C^2$ boundary $\Sigma =\partial \Om $, fix $\eta>0$ and assume that either $\Sigma$ is a circle or the pair $(\Sigma,\eta)$ satisfies Assumption~\ref{ass:LojassIntro}. Then for each $\bar{L}>0$ and  $\bar\phi>0$, there exist constants $C_{0,1} =C_{0,1}(\eta, \Sigma, \bar{L}, \bar\phi)$ such that the following holds. 
For any 
curve $\gamma \in \Aarea$ with length $L(\gamma) \leq \bar{L}$ and turning angle $\abs{\int_\gamma \kappa_\gamma ds_\gamma}\leq \bar\phi$, we may find $c^* \in \Crit$ such that
    \begin{equation}\label{est:Loj-distance}
            \| \gamma - c^*\|_{C^1(ds_\gamma)} \leq C_0 \big[ \|\kappa_{\gamma} - \bar{\kappa}_{\gamma}\|_{L^2(ds_\gamma)}+\abs{\al_1(\gamma)-\pihalf}+\abs{\al_2(\gamma)-\pihalf}\big]
    \end{equation}
    and  
    \beq\label{est:Loj-length}
    \abs{L(\gamma)-L(c^*)}\leq C_1 \big[ \|\kappa_{\gamma} - \bar{\kappa}_{\gamma}\|_{L^2(ds_\gamma)}+\abs{\al_1(\gamma)-\pihalf}+\abs{\al_2(\gamma)-\pihalf}\big]^2
    \eeq
    where $\al_{1,2}\in [0,\pi]$ denote the intersection angles at which $\gamma$ intersects $\Si$.
\end{theorem}
We carry out the proof of this theorem using explicit geometric constructions, which in particular avoid any use of compactness arguments. As such, the constants $C_{0,1}$ appearing in the above theorem are computable in terms of  basic geometric and analytic quantities associated to problem.\footnote{An inspection of the proof shows that the constants can be bounded explicitly in terms of $\eta$, the $C^2$ norm of the arclength parametrization $\si$ of $\Si$, the modulus of continuity of $\si''$, and (for $\Si$ not a circle) the spectral gap around $0$ of the $2\times 2$ matrix $d^2\LLeta(z^*)$ at critical points $z^*$.} 
Here and in the following we use the convention that 
the $C^1(ds_\gamma)$ distance $\norm{\gamma-\tilde \gamma}_{C^1(ds_\gamma)}$ of a curve $\tilde \gamma$  to a given curve $\gamma$ is computed using the reparametrisation of $\tilde \gamma$ over the interval $[0,L(\gamma)]$ with constant speed.
The expression 
\beq \label{def:eps-gamma}
\eps(\gamma):= \|\kappa_{\gamma} - \bar{\kappa}_{\gamma}\|_{L^2(ds_\gamma)}+\abs{\al_1(\gamma)-\pihalf}+\abs{\al_2(\gamma)-\pihalf}
\eeq
appearing in the above result is equivalent to the norm of the gradient of $L$ on $\Aarea$, as the first variation of the length along area preserving vector fields is given by  \eqref{eq:dL-area-preserv} above 
and as $\abs{\langle X(x_i),\tau_\Si(x_i)\rangle}=\abs{X(x_i)}|\cos(\al_i)|$ is bounded from above and below by a multiple of $ \abs{X(x_i)}\abs{\al_i-\pihalf}$. Thus \eqref{est:Loj-distance} is a quantitative form of the classification of critical points in \eqref{def:crit-points}.

Quantitative rigidity (or ``quantitative Alexandrov'') estimates for isoperimetric problems along the lines of \eqref{est:Loj-distance} have been investigated intensively over the past decade---see, e.g, \cite{CiraoloMaggi, Julin2D, JulinNiin, JuMoOrSp24, JiaZhang, Poggesi24}---motivated in part by applications to flows \cite{Julin2D, JulinNiin, JuMoOrSp24}.
For higher dimensional isoperimetric problems, quantitative ridigity theorems must be formulated to account for the possibility of bubbling; this behavior is precluded in the present context, essentially because controlling $\e(\g)$ amounts to controlling the second fundamental form in a super-critical norm.
 The one-dimensional nature of the problem allows for a hands-on proof of \eqref{est:Loj-distance}: we associate to $\g$ an explicit, quantitatively close circular arc $c\in\Circeta$, and then prove \eqref{est:Loj-distance} for curves in $\Circeta$ by using Assumption~\ref{ass:LojassIntro} or, in the case that $\Si$ is a circle, direct planar geometric arguments.

The second \Loj\ estimate \eqref{est:Loj-length} will be  obtained from \eqref{est:Loj-distance} using a first variation argument and will be
the key tool in the proof of our second main result, which establishes exponential convergence to equilibrium of global solutions to the free boundary area-preserving curve shortening flow. 
We recall that, 
starting from the seminal work of Simon \cite{SimonLoj}, \Loj\ estimates have been used as a powerful tool in the analysis of both asymptotics and  singularities for  gradient flows in myriad settings.

\begin{theorem} \label{mainthm-flow}
Let $\Omega$ be  convex body with $C^{2}$ boundary $ \Si=\partial \Omega $, fix $\eta, \bar{L}, \bar{\phi}>0$ and assume that either $\Sigma$ is a circle or the pair $(\Sigma,\eta)$ satisfies Assumption~\ref{ass:LojassIntro}.
Let $\gamma: \Dom\times [0,\infty)\to\R^2$
be a global-in-time solution to the flow  
\begin{align}
\begin{cases}
 \partial_t\gamma_t =(\kappa_{\gamma_t}-\bar{\kappa}_{\gamma_t})\nu_{\gamma_t} & \text{ on } \Dom\times [0,\infty),\\
 \gamma_t(a_1),\  \gamma_t(a_2) \in \Sigma  & \text{ on } [0,\infty),\\
 \tau_{\gamma_t}(a_1) = - \nu_\Sigma (\gamma_t(a_1)),  \ \ \tau_{\gamma_t}(a_2)  = \nu_\Sigma (\gamma_t(a_2)) & \text{ on } [0,\infty).
\end{cases}
\label{equ:apcsf}
\end{align}
Assume that the turning angle remains bounded by $|\int\kappa_{\g_t} ds_{\gamma_t}| \leq \bar\phi$ and that $\gamma_t$ intersects $\Omega$ only at the endpoints $x_{1,2}(\gamma_t)\in \Si$  for each $t$.
Then there is a unique arc $c^*\in \Crit$ such that $\g_t$ converges smoothly exponentially to $c^*.$ 
More precisely, for every $k \in \mathbb{N},$ there exists $C_k=C_k(\gamma_0,\Si, \eta, \bar{\phi}) $ and $c_k=c_k(\gamma_0,\Si, \eta, \bar{\phi})>0 $ 
so that 
\begin{align}\label{eqn: exp conv}
\|\tilde\gamma_t - \tilde c^*\|_{C^k([0,1])} \leq C_k \exp(- c_k t) \ \ \ \text{ for all } t\in [1,\infty)
\end{align}
for the constant speed parametrizations $\tilde {\g}_t$ and $\tilde{c}^*$ of $\g_t$ and $c^*$ on $[0,1]$. 
Furthermore, there is a constant
 $C= C(\Si,\eta,\bar{L},\bar{\phi})$ such that if $L(\g_0) \leq \bar{L},$ then
 the total $L^2$-distance traveled by both the original flow and the 
 reparametrized flow  is bounded by
\begin{equation}\label{est:dist-total}
   \int_{0}^{\infty} \|\partial_t \gamma_t\|_{L^2(ds_{\gamma_t})} dt  + 
   \int_{0}^{\infty} \|\partial_t \tilde\gamma_t\|_{L^2([0,1])} dt 
   \leq C(L(\gamma_{0})- L(c^*))^{1/2}.
        \end{equation}
\end{theorem}
We always consider $\Si$ with a  positive orientation, so in particular, $\nu_\Si$ is the
inner unit normal to $\Om$ in \eqref{equ:apcsf} above. The results of \cite{EMB2} give sufficient conditions on the initial data to ensure the assumptions of Theorem~\ref{mainthm-flow} hold, which will be essential in its application to the next result.

Our final main result is a sharp quantitative stability estimate for minimizers of the relative isoperimetric problem \eqref{eqn: isop problem def}. In the statement, $\kappamaxSi:=\max_{\Si}|\kappa_\Si|$ and $\areaC \approx .04$ is an explicit universal constant whose precise value is given in \eqref{def: c area bound}. 
\begin{theorem}\label{thm: stability} 
Let $\Omega$ be  convex body with $C^{2}$ boundary $ \Si=\partial \Omega $.
Fix $\eta\in (0, \areaC \kappamaxSi^{-2})$, and assume that either $\Sigma$ is a circle or the pair $(\Sigma,\eta)$ satisfies Assumption~\ref{ass:LojassIntro}.
There is an explicitly computable constant $c=c(\Sigma, \eta)$ such that
	\begin{align}\label{eqn: stability L1 theorem statement}
	P(E; \R^2 \setminus \Omega )  - \IP(\eta) 
    &\geq c \inf_{E_* \in \setmins}  |E\Delta E_*|^2
\intertext{  for any set of finite perimeter $E$ in $\R^2 \setminus \Omega$ with $|E|=\eta$.   If $\partial E\setminus \Omega$ is a rectifiable  curve, then additionally}
\label{eqn: stability thm statement hausdorff}
    P(E; \R^2 \setminus \Omega )^2  - \IP(\eta)^2 
    &\geq c \inf_{E_* \in \setmins}  d_H(\partial E, \partial E_*)^2\,.
	\end{align}
\end{theorem}

The second statement \eqref{eqn: stability thm statement hausdorff} is a free-boundary counterpart of the classical theorem of Bonnesen \cite{Bonnesen}. Simple examples given by removing or adding a small ball from a minimizer show that \eqref{eqn: stability thm statement hausdorff} is false without the additional assumption on $E$. By \cite{AmbrosioJEMS}, any {simple} set of finite perimeter $E$ in $\R^2$, i.e. one such that $E$ and $\R^2 \setminus E$ are indecomposable,  is bounded by a rectifiable Jordan curve up to modification of $E$ on a Lebesgue-null set. 
The assumption $\eta\in (0, \areaC \kappamaxSi^{-2})$ is an artifact of the proof; as described below, this is used to ensure the global existence of a well-behaved solution to the gradient flow.

Let us sketch how the we use the gradient flow to establish \eqref{eqn: stability L1 theorem statement}. (The proof of \eqref{eqn: stability thm statement hausdorff} is similar.) Let $E$ be a set as in \Cref{thm: stability}, and assume without loss of generality that $P(E; \R^2 \setminus \Omega )$ lies below the energy level of any non-minimizing critical point in \eqref{def:crit-points}. (Assumption 
~\ref{ass:LojassIntro} guarantees that the critical values of $L$ on $\Aarea$ are discrete, while if $\Si$ is a circle, minimizers are the only critical points.)
In a by-hand reduction procedure in \Cref{prop: summary reduction}, we associate to $E$ a set $F$ with $|F\setminus \Omega| = \eta$ whose relative boundary $\partial F \setminus \Omega$ is a convex $C^{2,\alpha}$ curve $\gamma_F$ meeting $\Sigma$ orthogonally such that
$$
    \delta_\eta(F) + |E\Delta F|^2 \leq C\,  \delta_\eta(E)\,.
$$
Here we let $\delta_\eta(E) =P(E; \R^2 \setminus \Omega ) - \IP(\eta)$ denote the isoperimetric deficit appearing on the right-hand side of \eqref{eqn: stability L1 theorem statement}.
It thus suffices to show that there is an isoperimetric set $E_*\in \setmins$ such that
\begin{equation}
    \label{eqn: intro stability for improved set}
L(\gamma_F) - \IP(\eta) =\de_\eta(F) \geq c |F\Delta E_*|^2.
\end{equation}
To this end, we evolve $\g_F$ by the gradient flow above. The convexity of $F$ and assumed bound $\eta < \areaC \kappamaxSi^{-2}$ on the enclosed area allow us to apply results of the first-named author \cite{EMB2}, which guarantee that the flow exists, remains embedded, and satisfies the assumptions of Theorems~\ref{mainthm-flow} for all $t \in [0,\infty)$.
By \Cref{mainthm-flow} and the monotonicity of length under the gradient flow, $\g_t$ converges exponentially to an arc $c^*\in \Crit$ that is the relative boundary $\partial E_* \setminus\Om$ of a minimizer $E_* \in \setmins.$ 
In particular, $L(c^*) = \IP(\eta)$. 
Finally, the fundamental theorem of calculus together with some geometric estimates show that the left-hand side of the displacement estimate \eqref{est:dist-total} bounds $|F\Delta E_*|$ above. Thus \eqref{eqn: intro stability for improved set} follows from \eqref{est:dist-total}.

The past two decades have seen tremendous advances in the theory of quantitative stability for isoperimetric inequalities, see e.g. \cite{FMP, FigalliMP, CicaleseLeonardi1, FuscoJulin, Neumayer16, BBJGaussian}, including for the relative (and capillary) isoperimetric problems on half planes and cones \cite{FigalliIndrei, MaggiMihaila,  Kreutz_Schmidt_2024, PPCap, carazzato2025quantitativeisoperimetricinequalitiescapillarity}. 
The story for the classical planar isoperimetric problem starts nearly a century earlier with Bonneson's theorem \cite{Bonnesen} (and the earlier \cite{Bernstein} for convex curves); see \cite{Osserman} for a survey of these early developments. 

Various proofs of Bonneson's inequality are known, using tools such as Steiner formulas for convex sets \cite{Osserman}, 
Fourier analysis \cite{Hurwitz, Fuglede}, an improved Wirtinger inequality \cite{Benson}, or integral geometry \cite{Santalo}. None of these proofs admit direct generalization to  yield alternative proofs of \Cref{thm: stability}.
An alternative approach to prove \Cref{thm: stability} would be to use a selection principle argument \cite{CicaleseLeonardi1}, where the spectral analysis is carried out in carefully chosen coordinates. While this 
approach would likely allow for the removal of the assumption $\eta < \areaC \kappamaxSi^{-2}$, its use of a compactness argument would prevent one from obtaining explicit constants. In contrast, the constants in all of our main results come from elementary geometric arguments, making them explicitly computable, and our approach highlights the intertwined nature of these three core problems of the quantitative analysis of PDEs.

\noindent{\it Acknowledgments:}
Parts of this work were carried out at BIRS Workshop 23w5062 in Banff and during a scientific visit to the University of Darmstadt; the authors are grateful for this support. The authors warmly thank Theodora Bourni and Raquel Perales for useful conversations. RN is grateful for the support of NSF grants DMS-2340195, DMS-2155054, and DMS-2342349. JP was supported by the National Research Foundation of Korea (NRF) grant funded by the Korea government (MSIT) RS-2024-00346651. EMB was partially funded by the Deutsche Forschungsgemeinschaft (German Research Foundation, DFG) under the grant number MA 7559/1-2 and she appreciates the support.

\section{Preliminaries}\label{sec: preliminaries}

\subsection{The isoperimetric profile}
Throughout the paper, $\Omega\subset \R^2$ will denote a  convex body, i.e. a compact convex set with nonempty interior, whose boundary $\Si = \partial \Omega$ is of class $C^2$. We let $\kappamaxSi >0$ be the maximum curvature of $\Si.$

For a set of finite perimeter $E$ in $\R^2$, we denote by $P(E)$  its perimeter, $P(E;A)$ its relative perimeter in an open set $A$, and $\partial^* E$ its reduced boundary, so that $P(E; A) = \mathcal{H}^1(\partial^* E; A)$; see \cite[Chapter 12]{MaggiBOOK} for basics on sets of finite perimeter. We will tacitly choose a representative of a set of finite perimeter $E$ with $\overline{\partial^*E}= \partial E$; see \cite[Prop. 12.19]{MaggiBOOK}. In two dimensions, sets of finite perimeter have a simple structure thanks to \cite{AmbrosioJEMS}; this structure will be recalled and utilized in Section~\ref{ssec: rect curve}.

The isoperimetric profile is continuous and non-decreasing and satisfies the upper and lower bounds 
\begin{equation}\label{eqn: bounds on profile}
    (2\pi)^\half\, \eta^\half \leq \IP(\eta) \leq 2\pi^\half\,  \eta^{\half}.
\end{equation} 
The left-hand side is precisely the isoperimetric profile $I_{{H}}$ of the half-plane, and the first inequality follows directly by using the convexity of $\Omega$ and the fact that $\eta\mapsto I_{{H}}(\eta)$ is increasing.\footnote{The analogous comparison theorem holds in higher dimensions but is no longer trivial to prove, see \cite{Choeetal}.} The upper bound in \eqref{eqn: bounds on profile} comes from choosing as a competitor in \eqref{eqn: isop problem def} a ball that is disjoint from $\Om$.

We recall that the set of circular arcs that intersect $\Si$ orthogonally and enclose a set of area $\eta$ is denoted by $\Crit$. 
It follows from  convexity of $\Sigma
$ that the center $z_c$ of 
 any $c\in \Crit$ is in $\R^2 \setminus \Omega^\circ$.  So, for $c\in \Crit$ with positive orientation,
 \beq
 \label{est:rad-circ-lower}
\phiturn(c) \in [\pi,2\pi), \qquad \pi r_c^2 \geq \eta \geq \tfrac12\pi r_c^2, \qquad \text{ and }\qquad 
\kappa_c \in [(\pi/2\eta)^{1/2} ,(\pi/\eta)^{1/2}],
\eeq
where $r_c$ and $\kappa_c$ respectively are the radius and curvature of the circle containing $c$, and $\phiturn(c)=\int_c \kappa_c ds_c$ is the turning angle of $c$.  
Given $r>0,$ we let ${\Theta}_\Sigma(r)$ be the maximal turning angle of a circular arc of radius $r$ in $\R^2 \setminus \Omega$ which meets $\Sigma$ orthogonally at the endpoints and note that a trigonometric exercise shows that
\begin{equation}\label{eqn: theta bound turning angle}
\Theta_\Sigma(r) \leq 2\pi - 2 \arctan(1/(r \kappamaxSi))\,.
\end{equation}
We can hence get a slightly improved upper bound  for $\IP(\eta)$ as follows. Let $c^*\in \Crit$ be a
circular arc which bounds a 
minimizer $E_*\in \Min$. Then $E_*$ contains 
 a circular sector of $B_{r_{c^*}}(z_{c^*})$ of angle $\phi=\phiturn(c^*) $.
In particular, $\half r_{c^*}^2\phi \leq \eta,$ i.e. $r_{c^*} \leq (2\eta/\phi)^{1/2}$. Thus  $\IP(\eta) \leq (2\theta \eta)^{1/2},$ and so by \eqref{eqn: theta bound turning angle},
\begin{equation}\label{eqn: improved upper bound}
\IP(\eta) \leq 2 \pi^\half \eta^{\half} \left(1  -\pi^{-1} \arctan \Big( \big( \tfrac{\pi}{2\eta}\big)^\half {\kappamaxSi}^{-1}\Big)\right )^{\half} \,.
\end{equation}

We fix the universal constant 
\begin{equation}\label{def: c area bound}
    \bar{\mathfrak{c}} = \mfrac{{4}}{25 {\pi}} \arcsin^2\left(\mfrac{1}{4\pi} \right) \approx .0415.
\end{equation}
 In \Cref{thm: stability}, we assume $\eta < {\bar{\mathfrak{c}}}\kappamaxSi^{-2}$, which guarantees the following upper bound for the isoperimetric profile.
\begin{lemma}\label{lem: sat flow hp}
    Suppose $\eta < {\bar{\mathfrak{c}}}\kappamaxSi^{-2}$. 
There is an explicitly computable constant $\delta_0 = \delta_0(\eta,\kappamaxSi) >0$ such that  for any $\delta \in [0,\delta_0],$
         \begin{equation}\label{eqn: 2 flow assumption holds}
       \IP(\eta) +\delta < \mfrac{4}{5 \kappamaxSi} \arcsin\left(\mfrac{\eta}{(\IP(\eta) +\delta)^2}\right)\,.
     \end{equation} 
\end{lemma}
\begin{proof}
The estimate \eqref{eqn: 2 flow assumption holds} with $\delta=0$ follows directly follows from the upper bound in \eqref{eqn: bounds on profile} and the definition of $\bar{\mathfrak{c}}$. Since the inequality \eqref{eqn: 2 flow assumption holds} with $\delta=0$ is strict and both sides are continuous functions of $\delta$, it is clear the estimate holds up to some $\delta_0$. This $\delta_0$ can be made explicit with the claimed dependence by using \eqref{eqn: improved upper bound}.
\end{proof}
\begin{remark}
    {\rm 
    The proof of \Cref{thm: stability} actually goes through for any $\eta $ small enough such that \eqref{eqn: 2 flow assumption holds} holds with $\delta=0.$ 
    }
\end{remark}
\begin{remark}\label{rmk: d sigma bound}
    {\rm 
A basic geometric argument shows that $\kappamaxSi^{-1} \leq d_\Si/2$, where
\begin{align}\label{def: width}
 d_\Sigma := \min\{|x-x'|:x, x'\in\Sigma,\tau_\Sigma(x) = -\tau_\Sigma(x')\}
\end{align}
is the \emph{width} of $\Sigma$, that is, the minimal distance of two parallel lines touching $\Sigma$. Hence the right-hand side of \eqref{eqn: 2 flow assumption holds} is bounded by $\frac25 \arcsin(1/(2\pi)) d_\Si$, and we can in particular use that $\IP(\eta) + \delta \leq d_\Si/2$ whenever 
 \eqref{eqn: 2 flow assumption holds} holds for a given $(\eta, \delta)$.
    }
\end{remark}

 The isoperimetric profile is differentiable outside a countable set and is left- and right-differentiable everywhere. A classical argument shows that the left-and right-derivatives of $\IP$ at $\eta$ are bounded above by the (constant) curvature $\kappa$ of any minimizer of $\IP(\eta)$, and at differentiability points, $\IP'(\eta) = \kappa$. 
 So, $\IP$ is absolutely continuous on compact subsets of $(0,\infty)$ and thus by the fundamental theorem of calculus, for any $0 <\eta_1 < \eta_2,$
 \begin{equation}\label{eqn: ftc profile}
   \Big(\mfrac{\pi}{2\eta_2}\Big)^\half (\eta_2 -\eta_1) \leq   \IP(\eta_2) -\IP(\eta_1 ) \leq \Big(\mfrac{\pi}{\eta_1}\Big)^\half(\eta_2 -\eta_1) .
 \end{equation}
 
\subsection{Notation and basic properties of oriented curves}
For a regular oriented $C^2$ curve $\g$ 
from $x_1(\ga)$ to $x_2(\ga)$
we let $\tau_\g =\frac{\g'}{\abs{\ga'}}$ 
be the unit tangent to the curve, and let $\nu_\ga = J \tau_\ga$, where 
$J$ denotes a counterclockwise rotation by $+\pi/2$. We let $L(\g)$ be the length of $\g$ and $s_\g$ the arclength parameter.
We denote by $ \kappa_\g= \frac{\langle \gamma'',\nu_\gamma \rangle}{|\gamma'|^2}$ the signed curvature with respect to $\nu_\gamma = J\tau_\gamma$. 
We also set  
 $\bar \kappa_\gamma : = \frac{1}{L(\ga)} \int_{\gamma}\kappa_\gamma \, ds_\gamma$ and denote by  
$\phiturn(\gamma):= \int_{\gamma}\kappa_\gamma \, ds_\gamma
$
 the total turning angle of $\gamma$. 
We recall that $\kappa_\ga, \bar\kappa_\gamma, \phiturn(\gamma), \tau_\ga,$ and $\nu_\g$ are of course independent of the choice of parametrization, but that they flip a sign when the orientation is reversed.

In the following we use the convention that $\Si$ is parametrized with positive orientation (i.e. counterclockwise). With this convention $\nu_\Si(p)$ is   the inner unit normal and $\kappa_\Si$ is nonnegative.
Throughout the paper we consider curves $\gamma\in \B$ for 
\begin{equation}
    \label{def:B}
\begin{split}
\B:=&\Big\{\gamma: \gamma \text{ is an oriented $H^2$ curve for which } \g \cap \Omega =\{ x_1(\gamma),  x_2(\gamma)\} \Big\}\,.
\end{split}
\end{equation}

\begin{definition} \label{def: sigma_gamma}
Given a curve $\ga \in \B$, we concatenate $\g$
with the (unique)  oriented sub-arc  $\si_\ga$ of $\Sigma$ for which the oriented immersed closed curve $\g + \sigma_\g$ is contractible in $\R^2 \setminus {\rm int}(\Om)$ and define the relative area $A_\Sigma(\g)$ enclosed by $\gamma$ as 
\begin{equation}
    \label{eqn: def area}
A_\Sigma(\gamma) := A(\gamma + \sigma_\gamma) := -\frac{1}{2}\bigg(\int \gamma \cdot \nu_{\gamma}\, ds_{\gamma} + \int\sigma_\gamma \cdot \nu_{\sigma_\gamma}\, ds_{\sigma_\gamma}\bigg).
\end{equation}
\end{definition}

An explicit way to construct $\sigma_\gamma$ is by taking the projection  $\tilde{\sigma}_\gamma=\pi_\Sigma \circ \gamma$ of $\gamma$ onto $\Sigma$, and defining $\sigma_{\gamma}$ as the (unique) oriented arc of $\Sigma$ that is homotopic to $-\tilde{\sigma}_\gamma$ (with fixed endpoints).
We note that $\sigma_\gamma$ can traverse $\Sigma$ multiple times, may have the same or opposite orientation as the full curve $\Si$, and can even just be a point. 
If $\gamma$ is embedded and oriented so that 
$\g +\si_\g$ has positive orientation, then $A_\Sigma(\gamma)$ coincides with the area of the region bounded by $\gamma$ and $\Sigma$.

\begin{remark}
    \label{rmk:Asi-continuous}
\rm{ As $\si_\gamma$ is determined uniquely by the above topological condition, this construction in particular ensures that $\si_{\gamma_t}$ and $\A_\Si(\gamma_t)$ vary continuously along any continuous family of curves $\gamma_t\in \B$, and that 
 $t\mapsto A_\Si(\gamma_t)$ is differentiable whenever  $t\mapsto \gamma_t\in \B$ is differentiable with 
\begin{equation}\label{eqn: first var of area}
\mfrac{d}{d t} A_{\Sigma}\left(\gamma_{t}\right) = 
-\int_\gamma X_t\cdot \nu_{\gamma_t} ds_{\gamma_t}= -
\int_{\gamma_t} X_t \cdot J \tau_{\gamma_t} \, ds_{\gamma_t}\, \qquad 
\text{ where }X_t:=
\mfrac{d}{d t}\gamma_t \,.
\end{equation}
We note that the variation of $\si_{\gamma_t}$ does not appear in the above formula as $\partial_t \si_{\gamma_t}$ is tangential.
}
\end{remark}

Since 
a change in orientation of $\g$ results in a change of sign of
$A_\Si(\g)$,
we define for each given $\eta>0$ the set of admissible curves for the relative isoperimetric problem as 
\[
\mathcal{B}_\eta := \{ \gamma \in \mathcal B :| A_\Sigma(\gamma)| = \eta\}.
\]
 It is immediate from the definition of $A_\Si$ that
  $L(\g) \geq\IP(\eta)$ whenever $\g\in \B_\eta$ is embedded and a short argument, which we include in Appendix~\ref{app: proof of remark},  ensures that 
  we can always bound 
  \beq
  \label{est:length lower bound}
  L(\gamma) \ge I_\Omega(\eta) \qquad \text{ for every } \gamma \in \mathcal{B}_\eta \text{ and every } \eta>0
  \eeq
  and hence in particular $L(\gamma)\geq (2\pi)^\half \eta^\half$ by \eqref{eqn: bounds on profile}.

We denote by $\al_i(\gamma)\in [0,\pi]$  the angles at which a given curve $\gamma\in \B$  intersect $\Si$, 
characterized  by 
\beq
\label{rel:angle-tangent}
\tau_\Si(x_1)=R_{\al_1}\tau_\gamma(x_1) \text{ and } \tau_\Si(x_2)=R_{-\al_2}\tau_\gamma(x_2)\,.
\eeq
Here and in the following 
$R_\psi$ denotes the rotation by angle $\psi$ in positive, i.e. counter-clockwise direction.

In what follows will often consider curves $\gamma\in \B_\eta$ whose length $L(\gamma)$ and turning angle $\abs{\phiturn(\gamma)}$ are bounded by given numbers $\bar L$ and $\bar \phi$ and whose intersection angles $\al_i$ with $\Si$ defined in \eqref{rel:angle-tangent} satisfy $\abs{\al_i(\gamma)-\pihalf}\leq \bar \beta$ for some $\bar\beta<\pihalf$. 
As $\Si$ is convex, this ensures that $\abs{\phiturn(\gamma)}\geq \pi-2\bar\beta$
and hence, recalling \eqref{eqn: bounds on profile}, that  
\beq
\label{est:prel-kappa-bar}
\frac{\pi-2\bar\beta}{\bar L}\leq \abs{\bar\kappa} \leq \frac{\bar\phi}{ (2\pi)^\half\, \eta^\half }.
\eeq 
In the following we also use that the parametrisation by arclength of any curve $\gamma:[0,L(\gamma)]\to \R^2$ can be expressed as 
\beq\label{eq:arclength-param}
{\gamma}(p) = \gamma(0) +  \int_0^p (\cos {\theta}_\gamma(q), \sin {\theta}_\gamma(q) ) \,dq \qquad \text{ for }\qquad \theta_\gamma(q)=\theta_0+ \int_0^q\kappa_\gamma(s) \,ds_\gamma \,,
\eeq
where $\theta_0(\gamma)$ 
denotes the angle formed by the tangent vector $\tau_\gamma(0)$ at the starting point and the $e_1$ axis.

\section{Proof of Theorem~\ref{thm:loj}}
In this section, we prove Theorem \ref{thm:loj}. Throughout the section, $\Omega$ denotes a convex body with boundary $\Si$ of class $C^2.$

For curves $\gamma$ as in Theorem~\ref{thm:loj} whose deficit $\e(\g)$ defined in  \eqref{def:eps-gamma} is bounded from below by a fixed constant $\eps_0>0$,
the theorem  holds trivially  
by choosing $\gamma^*\in \Crit$ 
as a global minimizer of $L$ in $\B_\eta$. Thus,
 we need only to consider curves $\g \in\B_\eta$ for which $\e(\g)$ lies below an explicit threshold. The proof in this case has three main steps: 
We first show that for 
any $\eta>0$ and $\gamma\in \Aarea$, there exists $c$ in the collection $\Circeta$ of circular arcs belonging to $\mathcal B_\eta$ (recall \eqref{def: circ eta})
whose distance to $\g$ in $C^1$ is controlled by $C\eps_\kappa(\gamma)$, where we set
\beq
\label{def:eps-kappa}
\eps_\kappa(\gamma):=  \|\kappa_{\gamma} - \bar{\kappa}_{\gamma}\|_{L^2(ds_\gamma)},
\eeq
compare Lemma \ref{lemma:reduction} and \Cref{cor:reduction}.
  This allows us to reduce the proof of the claimed distance \Loj\ estimate \eqref{est:Loj-distance} to the analysis of circular arcs, which is carried out in Section \ref{subsec:circular-arcs} and which, unlike the other parts of the proof, exploits the non-degeneracy condition on the curve $\Si$ if $\Si$ is not a circle.
 Combined, this will allow us to prove the claimed bound \eqref{est:Loj-distance} on the distance of $\gamma$ to the nearest critical point $\gamma^*$ with a short argument that is carried out at the beginning of Section \ref{subsec:proof-thm-loj} before we show how this distance \Loj\ estimate yields the claimed estimate \eqref{est:Loj-length} on $\abs{L(\gamma)-L(\gamma^*)}$ .

\subsection{Reduction to circular arcs}\label{ssec: reduction}
We first prove that curves for which $\eps(\gamma)$ is small are $C^1$-close to a circular arc $c_\gamma$ whose enclosed area will be close, but not necessarily equal, to $\eta$. We note that an argument in a similar spirit was used also in \cite{Julin2D} for closed curves. In a second step, we modify this initial circular arc in a way that the area constraint is satisfied. 
 In what follows, we will always take $\e(\g)$ small enough such that, in particular,
\beq\label{ass:angle}
\abs{\al_i(\gamma)-\pihalf}\leq \tfrac{\pi}{12}, \quad i=1,2,
\eeq for the intersection angles $\al_i$ defined in \eqref{rel:angle-tangent}.
For the first step we begin with the following lemma.

\begin{lemma}
   \label{lemma:reduction}
  For any $\eta,\bar{L},\bar{\phi}>0$, 
  there are explicit constants $\e_1 = \e_1(\Si, \eta, \bar{L}, \bar{\phi})>0 $ and $C_2=C_2(\Si,\eta,\bar{L},\bar{\phi})>0$  such that the following holds.
Let $\gamma \in \Aarea$ be any curve with $L(\gamma) \leq \bar{L}$, $\abs{\phiturn(\g)}\leq \bar\phi$, and $\e(\g) \leq \e_1$. 
Then $\gamma$ is embedded. Moreover,
letting $c_\gamma \in\mathcal{B}$ be the circular arc with radius $\bar r:= |\bar{\kappa}_\gamma|^{-1}$ emanating from $\gamma(0)$ which has the same intersection angle $\al_1(c_\gamma)=\al_1(\gamma)$ at this initial point and turns  counterclockwise if $\bar{\kappa}_\ga >0$ and clockwise if $\bar{\kappa}_\ga< 0$, we have
    \begin{equation}
        \label{eqn: reduction est}
    \| \gamma - c_\gamma\|_{C^1(ds_\gamma)} \leq C_2 
    \|\kappa_{\gamma} - \bar{\kappa}_{\gamma}\|_{L^2(ds_\gamma)}.
      \end{equation}      
\end{lemma}

\begin{remark}\rm{
 As mentioned above, we will take $\e_1$ small enough so that any for any curve $\g$ as in Lemma~\ref{lemma:reduction}, the intersection angles of $\g$ with $\Si$ satisfy \eqref{ass:angle}. Then by \eqref{est:prel-kappa-bar}, $|\bar \kappa_\ga| \geq \frac{5\pi}{6\bar{L}}$ and in particular, $\bar \kappa_\ga\neq0$ so $c_\ga \in \B$ is well defined.
 }
\end{remark}

For the proof of Lemma \ref{lemma:reduction} we will use the following elementary geometric fact.

\begin{lemma}
\label{lemma:circ-arcs}
Let $\hat c$ be a circle with radius $R$  and let 
$c_1$ be an oriented circle 
with radius $r$ through a point $x_1$ in the exterior of $\hat c$.
Suppose that there exists a point $x_0\in \hat c$ so that 
$\abs{x_1-x_0}\leq 
\tfrac{\pi}{24} \min(r,R)$ and  $\abs{\angle(\tau_{{c_1}}(x_1),\nu_{\hat c}(x_0))}\leq \frac{\pi}{12}
$, where $\nu_{\hat c}$ is the inward unit normal of $\hat c$. Then 
$c_1$ intersects $\hat c$ and the length of the circular arc $c$ from $x_1$ to the first point $x_2$ 
where $c_1$ (with the given orientation) intersects $\hat c$  is bounded by 
$L(c)\leq 2 \abs{x_1-x_0}$. 
\end{lemma}
We note that the analogous statement also holds if $\hat c$ is replaced by a straight line $\mathcal{T}$ and that the proof of this variation of the lemma can be either obtained in the limit $R\to \infty$ or by a simplified version of the proof below. 

\begin{proof}[Proof of Lemma \ref{lemma:circ-arcs}]
As the claim is invariant under rescaling, translation and rotation we can assume without loss of generality that $\abs{x_1-x_0}=1$, that $\hat c=\partial B_R(0)$ and that $x_0=(-R,0)$ and hence $\nu_{\hat c}(x_0)=e_1$. 
We parametrize $c_1:\R\to \R^2$ by arclength so that $c_1(0)=x_1$ and denote by 
 $\theta_{c_1}(p)$ the angle between the tangent of $\tau_{c_1}$  and $e_1$, compare \eqref{eq:arclength-param}. If $c_1$ intersects $\hat c$ we set  $p^*:=\inf\{p>0: c_1(p)\in \hat c\}$, so that $p^*$ is the parameter of the first intersection point, while in the case where $c_1\cap \hat c=\emptyset$ we simply let $p^*=\infty$. We note that this choice of $p^*$ ensures that  $\abs{c(p)}\geq R$ for all $p\in [0,p^*]$ and that  the claim of the lemma follows provided we show that $p^*<2$.

To prove this we first note that the assumptions of the lemma (and the above normalisation) ensure  that $\abs{\theta_{c_1}(0)}\leq \frac{\pi}{12}$ as well as that  
$r\geq \frac{24}{\pi}$ and hence that $\abs{\kappa_{c_1}}\leq \frac{\pi}{24}$. This ensures that 
$\abs{\theta_{c_1}(p)}\leq 2\abs{\kappa_{c_1}}+\abs{\theta_{c_1}(0)}=2 r^{-1}+\frac{\pi}{12}\leq \frac{\pi}{6}$ for all $p\in [0,2]$. 

To obtain a similar estimate for the angle 
  $\beta(p)$ between $-\frac{c_1(p)}{\abs{c_1(p)}}$ and $e_1$
   we recall that the initial point $x_1$ has distance at least $R$ from the origin and $e_2$ coordinate no more than $\abs{x_1-x_0}\leq 1$, hence allowing us to bound  
   $\abs{\beta(0)}\leq \abs{\tan(\beta(0))}\leq R^{-1}\leq\frac{\pi}{24}$. 
As $\abs{c_1(p)}\geq R$ for all $p\in [0,p^*]$ we can bound 
  the change in $\beta$
 by  
 $\abs{\beta'(p)}=\frac{1}{\abs{c_1(p)}}\abs{\Pi_{c_1(p)}(c_1'(p))}\leq \frac{1}{\abs{c_1(p)}}\leq R^{-1}\leq \frac{\pi}{24}$
 for all such $p$, where we let $\Pi_{c_1(p)}$ denote the projection onto ${T_{c_1(p)}\partial B_{\abs{c_1(p)}}}$. We hence conclude that  $\abs{\beta(p)}\leq \abs{\beta(0)}+\frac{\pi}{24} p\leq  \frac{\pi}{24}+\frac{\pi}{12}<\frac{\pi}{6}$ for all $p\in [0,\min(p^*,2)]$.

Combined this ensures that the angle between $-\frac{c_1(p)}{\abs{c_1(p)}}$ and $\tau_{c_1}(p)$ remains bounded by $\abs{\theta_{c_1}(p)}+\abs{\beta(p)}< \frac{\pi}{6}+\frac{\pi}{6}=\frac\pi3$  on $[0,\min(p^*,2)]$,
which ensures that on this interval
$$-\abs{c_1(p)}'=\langle \tau_{c_1}(p),-\frac{c_1(p)}{\abs{c_1(p)}}\rangle> \cos(\pi/3)=\tfrac12.$$
As $-\abs{c_1(p)}\leq -R=-\abs{x_0}$ on $[0,p^*]$ we can integrate this bound to deduce that 
$$\tfrac{1}{2}\min(p^*,2) < -\int_0^{\min\{p^*, 2\}} |c_{{1}}(p)|'dp \leq  \abs{x_1}-R=\abs{x_1}-\abs{x_0}\le\abs{x_1-x_0}= 1,
$$and hence indeed that $p^*<2$ as required.
\end{proof}

Based on this lemma we can now complete the proof of Lemma~\ref{lemma:reduction}.
\begin{proof} [Proof of Lemma~\ref{lemma:reduction}]
Let $d_0=d_0(\Si,\eta,\bar\phi)>0$ be chosen so that any (full) circle with radius $r\geq (2\pi)^\half \eta^\half \bar\phi^{-1}$ which intersects $\Si$ at an angle $\al_1$ with $\abs{\al_1-\pihalf}\leq \frac{\pi}{12}$ contains a point in $\Om$ whose distance to $\Si$ is greater than ${2}d_0$. Then set
\beq\label{ass:de_0}
\e_1':= 2^{1/2} \bar L^{-3/2} \min\Big(\mfrac{\pi}{24} \kappamaxSi^{-1}, \mfrac{\pi}{24}\mfrac{\eta^\half}{(2\pi)^{1/2}}, {\mfrac{d_0}{4}}, {\mfrac{\bar{L}}{4}}\Big)\,.
\eeq
and $\e_1= \min\{\e_1', \pi/12\}.$
As the claim is invariant under change of orientation, it suffices to consider the case where $\g$ has $\bar{\kappa}_\g> 0$ and hence $\phiturn(\ga)>0$. 
We parametrize $\ga$ by arclength on the interval $[0,L]$, $L:= L(\gamma)$, and recall that this parametrization can be expressed in terms of $\theta_\gamma(p)=\theta_0+\int_0^p \kappa_\gamma(s)ds_\gamma$ as described in \eqref{eq:arclength-param}.
Setting ${\theta}_1(p) = \theta_0 + \bar{\kappa}_\gamma p$ for $p \in [0,\infty)$, we get an analoguous parametrization $\gamma_1:\R\to \R^2$ by arclength of the circle that contains the circular arc $c_\gamma$ of radius $\bar{r}:= \bar{\kappa}_\g^{-1}$.

The fundamental theorem of calculus and the fact that $\theta_\gamma(L)=\theta_1(L)$ ensure that $\abs{\theta_\gamma(p)-\theta_1(p)}
\leq (L/2)^{1/2} \eps_\kappa(\gamma)
$ for all $p\in[0,L]$, which, when inserted into 
 \eqref{eq:arclength-param},
immediately imply that 
\begin{equation}\label{eqn: c1 ests}
|\gamma'(p) - {\gamma}_1'(p)| \leq 2^{-1/2}  L^{1/2}\eps_\kappa(\gamma) \qquad \text{ and }\qquad|\gamma(p) - {\gamma}_1(p)| \leq 2^{-1/2} L^{3/2}\eps_\kappa(\gamma).
\end{equation}
As the angle between two unit vectors $w_1, w_2$ is given by  $\angle(w_1,w_2)=2\arcsin(\frac{1}{2}\abs{w_2-w_1})$ and as \eqref{eqn: c1 ests} ensures that
$\abs{\tau_\gamma(p)-\tau_{\gamma_1}(p)}\leq \frac{1}{4}\leq 2\sin(\pi/10)$ we can in particular bound 
\beq
\label{est:needed-to-get-embedded}
\angle(\tau_\gamma(p),\tau_{\gamma_1}(p))\leq \mfrac{\pi}{5}\quad  \text{ and }\quad  \abs{\gamma(p)-\gamma_1(p)}\leq \min(\mfrac{d_0}4,\mfrac{\pi}{24}
\mfrac{\eta^\half}{(2\pi)^{1/2}}) \quad
\text{ for all } p\in [0,L].
\eeq
We now want to show that this ensures that $\g$ is embedded.  For this we first note that 
$\phiturn(\gamma)=\phiturn(\gamma_1\vert_{[0,L]})<2\pi$, i.e. $L<2\pi\bar r$, 
 since the full circle $\gamma_{1}\vert_{[0,2\pi\bar r]}$ contains a point $\gamma_1(p_0)$ with $\dist(\gamma_1(p_0),\gamma([0,L])\geq \dist(\gamma_1(p_0),\R^2\setminus \Om)\geq 2d_0$.  
Combined with  \eqref{eqn: bounds on profile} this also gives an improved upper bound of $\bar{\kappa}_\g <(2\pi/\eta)^\half$, and hence an improved lower bound of $\bar{r} > (\eta/2\pi)^\half$. The second estimate of 
\eqref{est:needed-to-get-embedded} hence in particular ensures that $\abs{\gamma(p)-\gamma_1(p)}<\half \min(d_0,\bar r)$ for every $p$. It is now useful to observe that whenever $0\leq p_1<p_2\leq L$ are so that 
$\abs{\gamma_1(p_1)-\gamma_1(p_2)}\leq \min(d_0,\bar r)$ we have 
$$\angle(\tau_{\gamma_1}(q_1),\tau_{\gamma_1}(q_2))\leq\angle(\tau_{\gamma_1}(p_1),\tau_{\gamma_1}(p_2))  \quad \text{ for all } q_{1},q_2\in [p_1,p_2] .$$
This implication is immediate if $\phiturn(\gamma_1\vert_{[0,L]}) \leq \pi$. Conversely, if $\phiturn(\gamma_1\vert_{[0,L]}) \geq \pi$, the implication holds since in this case we can bound  $|\g_1(0) -\g_1(L)|\geq 
|\g_1(0) -\g_1(p_0)| \geq |\g(0) -\g_1(p_0)|-\abs{\g(0)-\g_1(0)}>d_0$ which ensures that $\phiturn(\gamma_1\vert_{[p_1,p_2]})<\pi$.  

If there were any  $0 \leq p_1< p_2 \leq L$ for which 
 $\g(p_1)=\g(p_2)$, then this estimate would be applicable since $\abs{\g_1(p_1)-\g_1(p_2)}\leq  \abs{\gamma(p_1)-\gamma_1(p_1)}+\abs{\gamma(p_2)-\gamma_1(p_2)}
 <\min(d_0,\bar r)$ by \eqref{est:needed-to-get-embedded}. At the same time, 
 as $\gamma_1$ parametrises a circle with radius $\bar r$, this estimate would ensure that  $ \abs{\tau_{\gamma_1}(p_1)-\tau_{\gamma_1}(p_2)}=\abs{\nu_{\gamma_1}(p_1)-\nu_{\gamma_1}(p_2)}=\bar r^{-1} \abs{\gamma_1(p_1)-\gamma_1(p_2)}\leq 1$ and hence that 
  $\angle( \tau_{\gamma_1}(p_2),\tau_{\gamma_1}(p_1))<\pi/2$.
  We could hence conclude that 
  $$\angle(\tau_{\gamma}(q_1),\tau_{\gamma}(q_2))<\frac{2\pi}{5}+\frac{\pi}{2}<\pi \text{ for all } q_1,q_2\in [p_1,p_2]$$
  which contradicts the assumption that the curve $\gamma$ intersects itself at $\gamma(p_1)=\gamma(p_2)$. Hence our choice of $\eps_1$ indeed ensures that $\gamma$ is embedded.

 While the parametrization of $c_\gamma$ by arclength is $\gamma_1\vert_{[0,L_1]}$, where $L_1:=L(c_\gamma)$, the  constant-speed parametrization of $c_\gamma$ that is used to compute the $C^1$ distance in the lemma is given by 
 $\tilde \gamma_1(p):=\gamma_{1}(\frac{L_1}{L}p) $, $p\in [0,L]$. 
 To derive the desired estimate \eqref{eqn: reduction est} from  \eqref{eqn: c1 ests} 
 it suffices to 
 to prove that 
\begin{equation}\label{claim:diff-L}
    |L-L_1| \leq 2 \left| \gamma(L) - {\gamma}_1(L)\right|,
\end{equation}
which will allow us to deduce that 
$\abs{L-L_1}\leq 2^{1/2}L^{3/2} \eps_\kappa(\gamma).$

To prove \eqref{claim:diff-L} we
first note that the tangent to $\gamma_1$ at $x_1:= \gamma_1(L)$ agrees with $\gamma'(L)$ since $\theta_\gamma(L) =\theta_1(L)$. 
Setting $x_0:=\gamma(L)$ we can hence use the assumed bound  \eqref{ass:angle} on the intersection angles to conclude that  the angle  $\beta$ between $\gamma_1'(L)$ and the inner normal $\nu_\Si(x_0)$ to $\Si$ is so that  $\abs{\beta}=\abs{\al_2(\gamma)-\pi/2}<\frac\pi{12}$.
We furthermore note that $\bar r\geq \eta^\half/(2\pi)^\half$ using \eqref{est:prel-kappa-bar} and $\phiturn(\gamma)=\phiturn(\gamma_1)<2\pi$. We also recall that we have already shown that  $\abs{x_0-x_1}\leq  2^{-1/2}L^{3/2} \e_1$ in \eqref{eqn: c1 ests}. 
Our choice of $\e_1$ hence ensures that  $\bar r$ and $\hat r_{\Si}:=(\max_\Si \kappa_\Si)^{-1}$ are so that
$\abs{x_0-x_1}\leq \frac{\pi}{24}\min(\bar r, \hat r_\Si).$  
 To prove \eqref{claim:diff-L} 
we can hence apply Lemma 
 \ref{lemma:circ-arcs} for this choice of $x_0=\gamma(L)$ and $x_1:= \gamma_1(L)$ as follows:

If $L_1\leq L$, we choose $\hat c$ as the   tangent $\mathcal{T}^\Si_{x_0}:=\{x_0+q \tau_\Si(x_0), q\in \R\}$ through $x_0$  to $\Si$ and choose $c_1$ as the circle which contains $\gamma_1$ but is parametrised with opposite orientation. Lemma 
 \ref{lemma:circ-arcs} then ensures that $c_1$ intersects $\mathcal{T}^\Si_{x_0}$ and that $p^*:=\inf\{p>0: \gamma_1(L-p)\in \mathcal{T}^\Si_{x_0}\}$ is bounded by 
 $p^*\leq 2\abs{x_0-x_1}=2\abs{\gamma(L)-\gamma_1(L)}$.
We then observe that $\gamma_1\vert_{(L_1,L]}$ cannot intersect $\mathcal{T}^\Si_{x_0}$ as this would force the total turning angle of $\gamma_1$ on $[0,L]$ to be at least $2\pi$, which is excluded by our choice of $\eps_0$.  Thus we must have $L-p^*\leq L_1$, allowing us to deduce the desired bound of $\abs{L_1-L}\leq p^*\leq2\abs{\gamma(L)-\gamma_1(L)} $
in this first case where $L>L_1$.

In the case where 
$L_1>L$ we argue analogously, but now choose $\hat c$ to be the circle of radius $\hat r_\Si =(\max \kappa_\Si)^{-1}$ which touches $\Si$ in $x_0=\gamma(L)$  from the inside, and which is hence fully contained in $\Om$, and let $c_1$ be the circle that contains $\gamma_1$ and that has the same orientation as $\gamma_1$. 
  Lemma \ref{lemma:circ-arcs} then yields that $\gamma_1$ intersects $\hat c$ and that 
$p^*:=\inf\{p>0: \gamma_1(L+p)\in \hat c\}$ is bounded by $p^*\leq 2 \abs{x_0-x_1}$.
As $L_1$ is the first (positive) time at which $\gamma_1$ intersects $\Si$, we know that $\gamma_1\vert_{[L,L_1)}$ is contained in the exterior of $\Om$ and hence disjoint from  $\hat c$. Thus we must have  $L+p^*\geq L_1$ which implies the claimed inequality 
$\abs{L-L_1}\leq p^*\leq 2 \abs{x_0-x_1}$ if this second case where $L_1>L$.  

The reparametrised curve  
 $\tilde \gamma_1(p):=\gamma_{1}(\frac{L_1}{L}p) $ hence  satisfies 
\begin{equation}
    \label{eqn: c c tilde}
|\tilde \gamma_1(p) - {\gamma}_1(p)| \leq |L-L_1|, \qquad |\tilde \gamma_1'(p) -{\gamma}_1'(p)| \leq \left(\tfrac{1}{L} +\bar\kappa \right)|L-L_1| \leq \tfrac{(1+ 2\pi)}{L} |L-L_1| \leq \tfrac{\eta^{1/2}(1+ 2\pi)}{(2\pi)^{1/2}} |L-L_1|,
\end{equation}
where we use in the last step that $\phiturn(\gamma)<2\pi$ and \eqref{eqn: bounds on profile}. The proof now follows by combining \eqref{eqn: c1 ests}, \eqref{eqn: c c tilde}, and \eqref{claim:diff-L}.
\end{proof}

This allows us to prove the following key ingredient for the proof of Theorem \ref{thm:loj}. 

\begin{proposition}
\label{cor:reduction} 
Let $\Si$ and $\gamma$ be as in Lemma \ref{lemma:reduction}. 
Then there exists a circular arc $\tilde c_\gamma\in \Circeta$ so that 
 \begin{equation}
        \label{est:gamma-circle}
    \| \gamma - \tilde c_\gamma\|_{C^1(ds_\gamma)} \leq C_3  \eps_\kappa(\gamma) \qquad \text{ for } \qquad 
    \eps_\kappa(\gamma) :=\|\kappa_{\gamma} - \bar{\kappa}_{\gamma}\|_{L^2(ds_\gamma)}  \end{equation}
         and a constant $C_3=C_3(\Si,\eta,\bar L,\bar \phi)$.  In particular, we have
\beq     \abs{\al_i(\gamma)-\al_i(\tilde c_\gamma)}\leq       C_4\eps_\kappa(\gamma) 
, \qquad i=1,2
\eeq
 for a constant $C_4=C_4(\Si,\eta,\bar L,\bar \phi)$.    
\end{proposition}
\begin{proof}[Proof of Proposition \ref{cor:reduction}]
It suffices to consider curves for which $\eps_\kappa(\gamma)$ is less than a fixed constant $\eps_2=\eps_2(\Si,\eta,\bar L, \bar\phi)>0$,
as the claims trivially hold for $\e_\kappa(\g) \geq \e_2$ and any $\tilde c_\gamma \in \Circ_\eta$ by choosing the constants depending on $\e_2.$ 
As before, we write for short $L=L(\ga)$  and let $\bar{r} =\bar{\kappa}_\g^{-1}$ be the radius of the circular arc $c_\g$ of \Cref{lemma:reduction}. We also recall that by \Cref{lemma:reduction}, $\g$ is embedded and $\abs{\phiturn(\gamma)}<2\pi$, and that combining this latter fact with \eqref{est:prel-kappa-bar} guarantees that $(\frac{\eta} {2\pi})^{1/2}\leq \bar r\leq \frac{6\bar L}{5\pi}$.
Furthermore, since $\Si$ is embedded, there exists a constant $C=C(\Si)$ so that 
$\angle(\tau_\Si(x),\tau_\Si(\tilde x))\leq C \abs{x-\tilde x} $ for all $x,\tilde x\in \Si$.

Lemma \ref{lemma:reduction} hence in particular ensures that the angle between the tangent to $\Si$ at the endpoints $x_2(\gamma)$ of $\gamma$ and $x_2(c_\gamma)$ is bounded by $C\norm{\gamma-c_\gamma}_{C^0}\leq C\eps_\kappa(\gamma)$. Our construction furthermore ensures that 
 $\tau_\gamma(x_1(\gamma))=\gamma_1'(L)$ and that the angle between this vector and $\tau_{c_\gamma}(x_2(c_\gamma))=\gamma_1'(L_1)$ is given by $\frac{1}{\bar r} \abs{L(\gamma)-L(c_\gamma)}\leq(\frac{\eta} {2\pi})^{-1/2} 2 \left| \gamma(L) - {\gamma}_1(L)\right| \leq C \eps_\kappa(\gamma)$. 
 Combined we hence obtain that 
$\abs{\al_2(\gamma)-\al_2(c_\gamma)}\leq C \norm{\gamma-c_\gamma}_{C^1}\leq C \eps_\kappa(\gamma)$, $C=C(\eta,\Si)$. For $\eps_2$ small enough this in particular ensures that the second intersection angle 
$\al_2(c_\gamma)$ is bounded away uniformly from $0$ and $\pi$, say so that $\al_i(c_\gamma)\in [\pi/4,3\pi/4]$ for $i=2$, while this estimate is trivially true for $i=1$.

Since $\gamma$ is embedded, we can bound
\beq \label{est:diff-area-new}
\abs{A_\Si(c_\gamma)-\eta}=\abs{A_\Si(c_\gamma)-A_\Si(\gamma)}\leq 2\pi (\bar r+ \norm{c_\gamma-\gamma}_{C^0(ds_\gamma)}) \norm{c_\gamma-\gamma}_{C^0(ds_\gamma)}\leq C\eps_\kappa(\gamma), 
\eeq
for an explicitly computable constant $C=C(\Si, \eta, \bar{L},\bar{\phi})$, where in the final inequality we use $\bar{r}\leq \frac{6\bar{L}}{5\pi}$. 
 After reducing $\eps_2$ if necessary, we assume $C\e_2 \leq \eta/4$ for this constant, and thus \eqref{est:diff-area-new} guarantees that 
$\abs{A_\Sigma(c_\gamma) -\eta }\leq \eta/4$.

We now let $z\in\R^2$ be the center of the circle that contains $c_\gamma$ and consider the (continuous) family of circular arcs  $c_r(z)$ in $\mathcal{B}$ parametrized on $[0,L]$ with center $z$ for which $c_{\bar r}(z)=c_\gamma$. The uniform a priori bounds on the angles $\al_i(c_\gamma)$, the radius $\bar r$, and the area of $c_\gamma$ ensure that the map $r\mapsto c_r(z)\in \mathcal{B}$ is a well-defined $C^2$ map into $C^1([0,L], \R^2)$ at least on an interval of the form $(\bar r-c_0,\bar r+c_0)$ for a number $c_0=c_0(\eta, \Si)>0$ (compare Remark \ref{rmk: C2 map} below). 
The first variation of the area along families of circular arcs $c_{r}(z)\in \B$ with fixed center is given by $\partial_r A_\Si(c_r(z))=L(c_r(z))$; see \eqref{var:dArea-by-r} below. 
In particular,  
$\partial_r A_\Si(c_{r}(z))\geq I_\Om(\half\eta)\geq \pi^\half\eta^\half$
for all $r$ for which $A_\Si(c_r(z))\geq \half \eta$.

So, after reducing $\eps_2$ if necessary, we deduce that there is a (unique) $\hat r \in (\bar r-c_0,\bar r+c_0)$ for which  
 $A_\Si(c_r(z))=\eta$  and 
 \begin{equation}
     \label{eqn: r bound}
      \abs{\hat r-\bar r}\leq \frac{\abs{A_\Si(c_\gamma)-\eta}}{\pi^\half\eta^\half} \leq C \eps_\kappa(\gamma)
 \end{equation}
Thus $\tilde{c}_\g = c_{\hat r}(z) \in \Circeta$. It is simple to check that 
$
\norm{\tilde{c}_\g - c_{\gamma}}_{C^1([0,L])}\leq C\abs{r-\hat r}
$ for an explicit $C=C(\bar{L},\eta)$. Combining this with \eqref{eqn: r bound} and the bound \eqref{eqn: reduction est} on $\norm{c_\gamma-\gamma}_{C^1(ds_\gamma)}$ obtained in Lemma \ref{lemma:reduction}, we obtain the first claim \eqref{est:gamma-circle}.
 
The second claim follows from the first. As noted above, $\abs{\angle(\tau_\Si(x),\tau_\Si(\tilde x))}\leq C \abs{x-\tilde x} $ for all $x,\tilde x\in \Si$. So, again using that  $\angle(w_1,w_2)=2\arcsin(\frac{1}{2}\abs{w_2-w_1})$
for unit vectors $w_1,w_2$, we have 
 $\abs{\al_i(\tilde c_{\ga})-\al_i(\gamma)}\leq C \abs{x_i(\tilde c_{\ga})-x_i(\gamma)}+2\arcsin(\frac{1}{2}\abs{\tau_{\tilde c_{\ga}}(x_i(\tilde c_{\ga}))-\tau_\gamma(x_i(\gamma))})\leq C \norm{\tilde c_{\ga}-\gamma}_{C^1(ds_\gamma)}$, which combined with \eqref{est:gamma-circle} completes the proof. 
\end{proof}

\subsection{Analysis of circular arcs}
\label{subsec:circular-arcs}
We let $\Circ$ be the subset of $\B$ that is made up of circular arcs which intersect $\Si$ transversally and let $\Circeta:=\Circ\cap \B_\eta$, compare \eqref{def: circ eta}. Given $c\in\Circ$, we denote by $z_c$ and $r_c$ the center and radius of the circle that contains $c$. 

We first note that for any $c\in \Circ$ there exist neighborhoods $U_{c}$ of $z_c$, $U_{1,2}$ of the endpoints $x_1(c)$, $x_2(c)$ of $c$ and $I$ of $r_c$ so that for every $z\in U_c$ and any $r\in I$ there is a unique circular arc $c_r(z)\in \B$ with radius $r$ and center $z$  whose endpoints $x_i(c_r(z))$ are in $U_i$. 
As $\Si$ is assumed to be $C^2$, a simple argument using the implicit function theorem, which is applicable when the intersection angles of these circular arcs remain bounded away from $0$ and $\pi$, furthermore gives the following.

\begin{remark}\label{rmk: C2 map}
\rm{
For any $r_0>0$ and $\beta_0<\pi/2$ there exist numbers $d_{0,1}=d_{0,1}(r_0,\beta_0,\Si)>0$ so that the following holds true. If the radius and intersection angles of $c\in\Circ$ are so that $r_c\geq r_0$ and $\abs{\al_i(c)-\pi/2}\leq \beta_0$, $i=1,2$, then the above family $(r,z)\mapsto c_{r}(z)$ is well defined on $(r_c-d_0,r_c+d_0)\times B_{d_1}(z_c)$, and the maps $(r,z)\mapsto x_i(c_r(z))$ and  $(r,z)\mapsto \theta_i(c_r(z))$ which assign to each such pair the endpoints and the angles between $x_i-z$ and $e_1$ are given by $C^2$ maps whose norms are bounded by a constant that only depends on $r_0$, $\beta_0$ and the $C^2$ norm of $\Si$.  
}
\end{remark}
    
In the following it suffices to consider circular arcs with positive orientation and it will be convenient to parametrize these arcs $c_r(z)$ with constant speed over $[0,1]$, i.e. as  
\beq
\label{para:circ-arcs}
 p\mapsto z+r(\cos,\sin)(\theta(p)), \quad \theta(p):=\theta_1+p(\theta_2-\theta_1), p\in [0,1],
 \eeq 
which then provides a way of viewing $(r,z)\mapsto c_r(z)$ as a $C^2$ function from 
$(r_c-d_0,r_c+d_0)\times B_{d_1}(z_c)$ to $C^2([0,1],\R^2)$ with uniformly bounded norms. We can furthermore use that the dependence of the intersection angles $\al_i(c_r(x))$ on $r$ and $x$ is controlled in $C^1$ since $\Si$ is $C^2$. 

For families of such circular arcs we now show the following useful lemma.

\begin{lemma}
\label{lemma:variations-circ-arcs}
Let $c_\eps=c_{r_\eps}(z_\eps)$ be a differentiable family of circular arcs in $\Circ$ which have positive orientation. 
Then the variation of the enclosed area is given by 
 \beq\label{var:Area-circles}
\ddeps A_\Si(c_\eps)=\peps r\cdot L +\langle \peps z,- J (x_2-x_1))\rangle.
\eeq
Furthermore, the variation of the endpoints along general variations $c_\eps$ can be expressed as 
\beq
\label{eq:var-xi-with-mu}
\peps x_{1}= \mu_1 \tau_{\Si}(x_1)\quad \text{ and } \quad \peps x_{2}= -\mu_2 \tau_{\Si}(x_2)\quad \text{ for }\quad \mu_i=\frac{-\peps r+\langle \nu_c(x_i),\peps z\rangle}{\sin(\al_i)}.
\eeq
For area preserving variations \eqref{eq:var-xi-with-mu}  this formula reduces to 
\beq\label{eq:Yi}
\mu_i=\frac{1}{\sin(\al_i)} \langle \peps z,Y_i(c)\rangle
\quad\text{ for } \quad Y_i(c):= \nu_{c}(x_i)-L^{-1}J(x_2-x_1), \quad i=1,2
\eeq
where $\nu_c$ is the inward unit normal of $c$, which in turn ensures that the variation of the length along area preserving variations is given by 
\beq
\label{var:lambda}
\peps L(c_\eps)=-\sum_i\cot(\al_i) 
\langle \peps z,Y_i(c)\rangle.
\eeq

\end{lemma}
In the above lemma and its proof we use the convention that all geometric quantities, such as the length $L$, the endpoints $x_i$, the intersection angles $\al_i$ are evaluated for the corresponding circular arc $c=c_\eps$. 
\begin{remark}
  \rm{
  From \eqref{var:Area-circles}, we see that 
along families of circular arcs with fixed center, the first variation of the area is given by 
\beq
\label{var:dArea-by-r} 
\partial_r A_{\Si}(c_r(z))=L(c_r(z)),
\eeq
and is in particular bounded below by  $(2\pi \eta)^{1/2}>0$ whenever $c_{r}(z)\in \Circeta$.
  }  
\end{remark}

\begin{proof}[Proof of Lemma \ref{lemma:variations-circ-arcs}]
We parametrize the circular arcs $c_\eps$ as in \eqref{para:circ-arcs} and use that the orientation of $c_\eps$ ensures that $\nu_{c_\eps}$ is the inner normal to $c_\eps$ to write $\nu_\eps(p)=-(\cos,\sin)(\theta_\eps(p))$. As $J\nu_\eps(p)=-\tau_\eps(p)=(\sin,-\cos)(\theta_\eps(p))$, this allows us to express the  generating vector field 
$X=\peps c_{\eps}$ as  
\begin{equation}\label{eqn: vector field}
    X=\peps c_{\eps}=\peps z_\eps-\peps r_\eps \nu_c-r \peps \theta_\eps J\nu_c.
\end{equation}

 As $\abs{c'}=L$ and as $\frac{L}{\theta_2-\theta_1}=r$, the formula 
 \eqref{eqn: first var of area}  for the variation of the area hence yields that 
\beqa\label{eqn: first var area step 1}
\ddeps A_\Si(c_{r_\eps}(z_\eps))=-\int_c X\cdot \nu_c \, ds_c
& = \peps r \cdot L+ L \cdot 
\Big\langle \peps z,\int_0^1(\cos,\sin)(\theta_1+p(\theta_2-\theta_1)) dp\Big\rangle = I + II.
\eeqa
For term $II$, we integrate to find 
\beqa
II &= \tfrac{L}{\theta_2-\theta_1}
\langle \peps z,(\sin,-\cos)(\theta_2)-(\sin,-\cos)(\theta_1)\rangle
=
\langle \peps z,r J\nu_c(x_2)-rJ \nu_c(x_1) \rangle
=
\langle \peps z_\eps, -J  (x_2-z) +J (x_1-z)\rangle\\
&=
\langle \peps z_\eps, -J  (x_2-x_1)\rangle
\eeqa
which together with \eqref{eqn: first var area step 1} establishes the first claim \eqref{var:Area-circles} of the lemma.

Next, we note that since the endpoints $x_{1}, x_2$ are contained in $\Si$, their variation can be written as $\peps x_{i}=\pm \mu_i \tau_{\Si}(x_i)$ for some $\mu_i=\mu_i(\eps)\in \R$, where here and in the following $\pm$ is to be understood as $+$ for $i=1$ and as $-$ for $i=2$. 
To determine $\mu_i$ we can use that 
$$\peps r_\eps=\thalf r^{-1} \peps\abs{x_i-z}^2=\langle r^{-1}(x_i-z), \pm \mu_i\tau_\Si(x_i)-\peps z\rangle=-
\mu_i
\langle \nu_c(x_i), \pm \tau_\Si(x_i)\rangle
+\langle \nu_c(x_i),\peps z\rangle$$
since $\nu_c(x_i)=-\frac{x_i-z}{r}$ is the inner unit normal. We can now use that $\tau_\Si(x_i)=R_{\pm\al_i}\tau_c(x_i)$, compare
\eqref{rel:angle-tangent}, to write
$$\langle \nu_c(x_i), \pm \tau_\Si(x_i)\rangle=\langle J\tau_c(x_i),\pm R_{\pm\al_i}\tau_c(x_i)\rangle=\cos(\pi/2-\al_i)=\sin(\al_i).$$
Thus 
$\peps r_\eps=-
\mu_i\sin(\al_i) 
+\langle \nu_c(x_i),\peps z\rangle.$
This establishes the formula \eqref{eq:var-xi-with-mu} for the variation of the endpoints. 

Now, by \eqref{var:Area-circles}, 
 area preserving variations are characterized by $\peps r=L^{-1} \langle \peps z,J (x_2-x_1)\rangle.$ Making this substitution in \eqref{eq:var-xi-with-mu} directly yields the expression \eqref{eq:Yi} for $\mu_i$ in the case of area preserving variations.
 Inserting the resulting expression for $\peps x_i$ into the formula 
\eqref{eq:dL-area-preserv} for the first variation of the length yields the claimed expression in (\ref{var:lambda}) of 
\beqa
\peps L(c_\eps)&=\langle \peps x_2,\tau_c(x_2)\rangle-\langle \peps x_1,\tau_c(x_1)\rangle=-\sum_i \mu_i \langle \tau_\Si(x_i),\tau_c(x_i)\rangle=-\sum_i\cot(\al_i) 
\langle \peps z,Y_i(c)\rangle
.
\eeqa
\end{proof}

We note that \eqref{var:dArea-by-r} in particular ensures that $\Circeta$ is a $C^2$ manifold that can be parametrized locally using the center. 
To be more precise, 
given a local $C^2$ parametrization $(r,z)\mapsto c_r(z)$ of a neighborhood of a given $c\in \Circeta$ in $\Circ$ as considered above, the implicit function theorem ensures that for $z$ in a sufficiently small neighborhood there exists a unique $r(z)$ so that the corresponding arc $c_{r(z)}(z)$ encloses the required area $\eta$ and that the function $z\mapsto r(z)$ is $C^2$. 

This allows us to view the restriction of the length functional to this $2$-dimensional manifold locally as a $C^2$ function $\LL_\eta(z):=L(c_{r(z)}(z))$ of $z$ whose first variation is described by \eqref{var:lambda}. 
As $\Circeta$ is a submanifold of $\Aarea$, and since the critical points of the length functional (with prescribed enclosed area) are always circular arcs which intersect $\Si$ transversally (and indeed orthogonally), it is trivially true that any critical point $\gamma\in C_\eta^*$ of our original problem is also a critical point of this restricted functional. Conversely, as the vectors $Y_{i}=\nu_{c}(x_i)-L^{-1}J(x_2-x_1)$, $i=1,2$, 
appearing in the formula \eqref{var:lambda} of $d\LLeta(z)(\peps z)$ 
are trivially linearly independent,  we can immediately deduce that $\na \LL_\eta$ vanishes if  and only if 
$\cot(\al_i)=0$ for both $i=1,2$, i.e. if and only if 
both intersection angles are $\al_i=\pihalf$. We also note that $\abs{Y_i(c)}\leq 2$ and hence that $\abs{d\LLeta(\peps z)}\leq 2 (\abs{\cot(\al_1)}+\abs{\cot(\al_2)})\abs{\peps z}$.

We can hence use that $\LL_\eta$ and $L$ are related by the following.

\begin{lemma}
\label{lemma:relation-Leta-L}
For any $\eta>0$
we have $\Crit=\{c\in \Circeta: d \LL_\eta(c)=0\}$ and,  parametrizing $\Circeta$ locally by the center of the circular arcs as described above, we can estimate 
\beq
\label{est:na-L-above}
\abs{\nabla \LL_\eta(z)}
\leq 2(\sin\de)^{-1}
[\abs{\al_1(\gamma)-\pihalf}+\abs{\al_2(\gamma)-\pihalf}]
\eeq
for $\de>0$ chosen so that $\al_i\in [\de,\pi-\de]$, $i=1,2$.  
\end{lemma}

 Given any distinct points $x_1,x_2\in \Si$, we of course know that there  
exists a circular arc $c^*$ that meets $\Si$ orthogonally at $x_1,x_2\in \Si$ if and only if the tangents $\mathcal{T}_{x_1}^\Si$ and 
$\mathcal{T}_{x_2}^\Si$ to $\Sigma$ intersect at a point $p^*$ with $\abs{p^*-x_1}=\abs{p^*-x_2}$. If these tangents are not parallel this is equivalent to the existence of a circle $\Si^*$ which is tangent to $\Si$ at both $x_1$ and $x_2$, and also equivalent to these two tangents having the same intersection angle with the line through $x_1$ and $x_2$.
In particular, given $c^* \in \Circ^*$, we can use that the unit vectors 
\[
v_{\parallel} =v_{\parallel} (c^*)= \frac{x_1(c^*) - x_2(c^*)}{|x_1(c^*) - x_2(c^*) |}, \qquad
v_{\perp} =v_{\perp} (c^*)
=\frac{ \nu_{\Si}(x_1)  + \nu_{\Si}(x_2)}{| \nu_{\Si}(x_1)  + \nu_{\Si}(x_2)|},
\]
are orthogonal. Since  we always parametrize $c^*$ with positive orientation,  the orthonormal basis $\{v_{\perp}, v_{\parallel} \}$ is positively oriented and we can furthermore write 
\beq 
\label{eq:normal-in-new-basis}
v_{c^*}(x_i)=-R_{\pm \beta}v_{\perp}=-(\cos\beta \, v_\perp\pm \sin\beta \,v_{\parallel})
\eeq
where $\beta=\beta(c^*)\in (0,\pi/2]$ is so that the opening angle of $c^*$ is given by $\theta_2(c^*)-\theta_1(c^*)=2\pi-2\beta$, see Figure~\ref{fig: config for hess}. 
Here and in the following we continue to use the convention that  
$\pm$ is to be understood as $+$ for $i=1$ and $-$ for $i=2$.

As $L(c^*)=r_{c^*} (2\pi-2\beta)$ and 
$\abs{x_2-x_1}=2r_{c^*}\sin\beta$, we can write the vectors  $Y_i(c^*)=\nu_{c^*}(x_i)-L^{-1}J(x_2-x_1)$ obtained in Lemma \ref{lemma:variations-circ-arcs} as
\beq
\label{eq:writing-Y-basis}
Y_i(c^*)=-(\cos\beta \,v_\perp\pm \sin\beta \,v_{\parallel})+\mfrac{\sin\beta}{\pi-\beta} Jv_{\parallel}=-\big(\cos\beta +\mfrac{\sin\beta}{\pi-\beta}\big)\,v_\perp \mp \sin\beta \,v_{\parallel}.
\eeq

Furthermore, in this canonical basis, the Hessian of $\LLeta$ at $c^*$ is given by the following explicit geometric expressions.

\begin{lemma}\label{lemma:Hessian}
    Let $c^* \in \Criteta$ be parametrized with positive orientation, let $z^*:=z_{c^*}$ be its center and let
    $ v_{\perp}=v_\perp(c^*) , v_{\parallel}=v_\parallel(c^*) $ and $\beta=\beta(c^*)$ be as above, see also Figure~\ref{fig: config for hess}. 
    Then the Hessian of $\LLeta$ at $z^*$ is given by 
    \beqa 
    \label{eq:Hessian-in-v}
    (d^2 \LL_\eta)(z^*)(v_{\perp}, v_{\perp})&=(\kappa_\Si(x_1)+\kappa_\Si(x_2))\,\big(\cos\beta+\mfrac{\sin\beta}{\pi-\beta}\big)^2
    +2\kappa_{c^*}\sin\beta\,\big(\cos\beta+\mfrac{\sin\beta}{\pi-\beta}\big),\\
    (d^2 \LL_\eta)(z^*)(v_{\perp}, v_{\parallel})&=(\kappa_\Si(x_1)-\kappa_\Si(x_2))\,\sin\beta\,\big(\cos\beta+\mfrac{\sin\beta}{\pi-\beta}\big), \\
     (d^2 \LL_\eta)(z^*)(v_{\parallel}, v_{\parallel})& =(\kappa_{\Si}(x_1)+\kappa_{\Si}(x_2))\,\sin^2\beta-2\kappa_{c^*}\sin\beta\cos\beta.
    \eeqa
    \end{lemma}

 \begin{figure}
        \centering    \includegraphics[width=0.5\textwidth]{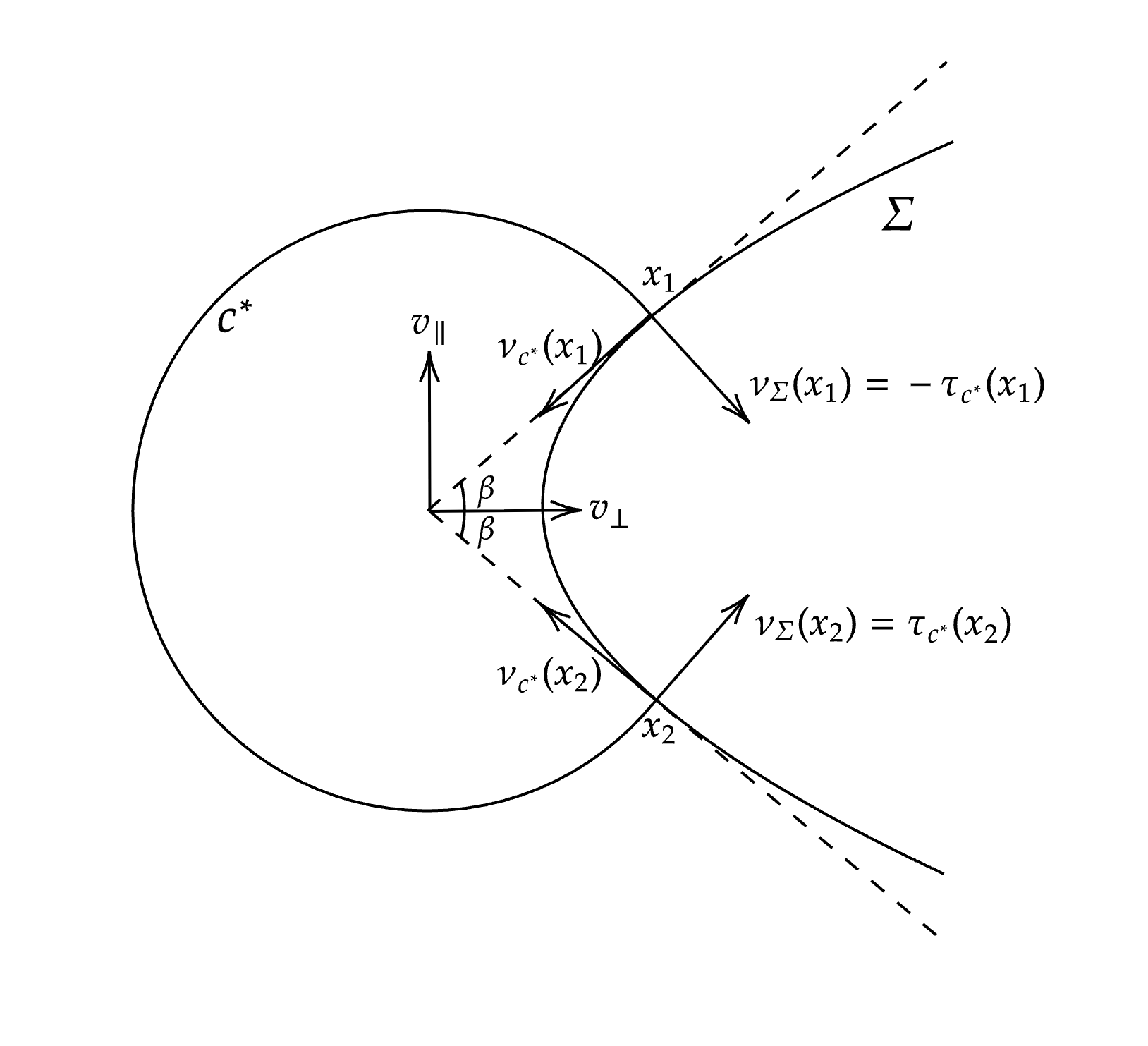}
        \caption{A critical arc $c^*$ with relevant points and directions}
        \label{fig: config for hess}
    \end{figure}

We recall that changes of the support curve $\Si$ that leave $\Si$ invariant up to first order at the endpoints $x_1(c^*), x_2(c^*)$ of a critical arc $c^*$ of $\Si$ will preserve the criticality of $c^*.$ Analogously, Lemma~\ref{lemma:Hessian} shows that perturbations of $\Si$ that preserve the second order behavior of $\Si$ at the endpoints leave the Hessian invariant.

We also note that the above lemma in particular ensures that for every support curve $\Si$ 
\beq 
\label{est:d2L-positive}
(d^2 \LL_\eta)(z_{c^*})(v_{\perp}(c^*), v_{\perp}(c^*))>0 \text{ for every } c^*\in \Circeta^*
\eeq
and hence that $v_\perp$ can never be a Jacobi field.  
In the special case when $\Si$ is a circle, $v_\parallel$ is the generator of the family of critical points we obtain by rotating $c^*$ around the center of $\Si$ and thus trivially a Jacobi field. In this case the set of Jacobi fields is hence given by the span of $v_\parallel$.

In the case where $\kappa_{\Si}(x_1) = \kappa_{\Si}(x_2)$  the Hessian diagonalizes with respect to this basis and   
the non-degeneracy condition has the following simple geometric interpretation.

\begin{remark}
For critical points $c^*$  of $\LLeta$ for which $\kappa_{\Si}(x_1(c^*)) = \kappa_{\Si}(x_2(c^*)) \ge 0$, 
the non-degeneracy condition  holds if and only if 
the osculating circles for $\Si$ at the endpoints $x_1(c^*)$ and $x_2(c^*)$ of $c^*$  do not coincide.\footnote{There are many settings in which this symmetry condition on $\kappa_{\Si}(x_i(c^*))$ is a priori known to hold, for example due to symmetry properties and/or monotonicity properties of the curvature of $\Si$. For instance, if $\Si$ is an ellipse, then critical points are always symmetric across one of the axes, and the osculating circles at $x_1$ and $x_2$ cannot coincide unless $\Sigma$ is a circle.} Here, the osculating circle for $\Sigma$ at a point with $\kappa_\Si=0$ is to be understood as the tangent line to $\Sigma$; in fact, if $\kappa_{\Si}(x_1(c^*)) = \kappa_{\Si}(x_2(c^*))=0$ then the non-degeneracy fails if and only if one arc of $\Si$ from $x_1$ to $x_2$ is a straight line segment.
\end{remark}

To see this, we note that since 
the Hessian is  diagonal in this case and $(d^2 \LL_\eta)(z^*)(v_{\perp}, v_{\perp})>0$, the existence of a non-trivial Jacobi field is  equivalent to $(d^2 \LL_\eta)(z^*)(v_{\parallel},v_{\parallel}) = 0$. When $\kappa_\Sigma(x_1) = 0$, this quantity vanishes if and only if $\beta =\pi/2$ by \eqref{eq:Hessian-in-v}, in which case the tangent lines to $\Sigma$ at $x_1,x_2$ must agree as well as $\tau_\Sigma(x_1)=\tau_\Sigma(x_2)$. Since $\Sigma$ is convex and embedded, this implies that an arc of $\Sigma$ between $x_1$ and $x_2$ is contained in the tangent line at $x_1$. When $\kappa_\Sigma(x_1) > 0$, by \eqref{eq:Hessian-in-v}, $(d^2 \LL_\eta)(z^*)(v_{\parallel},v_{\parallel}) = 0$ if and only if $\cos \beta>0$ and $r^*\tan \beta = {1}/\kappa_{\Si}(x_1)$  where $r^*$ is the radius of $c^*.$ In this case we observe that since $c^*$ is orthogonal to $\Si$ at $x_1, x_2$, a circle through $x_1, x_2$ is tangential to $\Si$ if and only if its center is on the tangent to $c^*$ at $x_i$. Thus let $p$ denote the unique point where the tangent to $c^*$ at $x_1$ intersects $z^*+\text{span}\{v_\perp\}$, so that geometrically, $\rho^* := r^*\tan \beta$ is the distance between $x_1$ and $p$. By symmetry of $x_1,x_2$ across $z^*+\text{span}\{v_\perp\}$, we see that both $x_1$ and $x_2$ lie in on the circle of radius $\rho^*$ centered at $p$. Thus, the  circles of radius $1/\kappa_{\Sigma}(x_1)$ tangent to $\Si$ at $x_1$ and $x_2$ have the same center if and only if $\rho^* = 1/ \kappa_{\Sigma}(x_1)$.

\begin{proof}[Proof of Lemma \ref{lemma:Hessian}]
Let $\{c_{\epsilon,\delta}\}$ be a two-parameter family of circular arcs with $c_{\epsilon,\delta}\in\Circeta$ and $c_{0,0}=c^*$.
As 
\eqref{var:lambda} is applicable for $\eps\mapsto c_{\eps,\de}$ for any $\de$
 we can differentiate this formula with respect to $\de$ to obtain that  
\beq 
\label{eq:sec-var-step1} 
\partial_\epsilon \partial_\delta L(c_{\epsilon,\delta})=-\partial_\de\bigg(\sum_{i} \cot(\alpha_i(c_\de))\cdot\langle \partial_\epsilon z_{\eps,\de},Y_i(c_\de)\rangle\bigg)
=
\sum_{i} \partial_\delta \alpha_i(c_\de)\cdot\langle \partial_\epsilon z,Y_i(c^*)\rangle
\eeq
where we continue to use the convention that derivatives with respect to $\eps$ and $\de$ are evaluated at $\eps=0$ and $\de=0$ and where we use that   $\alpha_i(c^*)=\pi/2$ and hence $\cot(\alpha_i(c^*))=0$ and $\partial_\de \cot(\al_i)=-1$.

To calculate $\partial_\delta \alpha_i$ 
we differentiate the relation $\cos(\alpha_i) = \langle \tau_{\Si}(x_i), \tau_c(x_i)\rangle$ and use that
$\pd x_i=\pm \mu_i \tau_\Si(x_i)$ for $\mu_i=\langle \pd z,Y_i(c^*)\rangle$ 
as well as $\nu_{\Si}(x_i)=\mp  \tau_{c^*}(x_i)$. Since $\partial_s \tau_\Si=\kappa_\Si \nu_\Si $ we hence 
get
\beqa 
\label{eq:angle-derivative}
-\partial_\de \al_i&=-\sin(\alpha_i(c^*))\cdot \partial_\delta\alpha_i = \langle \partial_\delta\tau_{\Si}(x_i), \tau_c(x_i)\rangle+\langle \tau_{\Si}(x_i), \partial_\delta\tau_c(x_i)\rangle\\
&=\pm\mu_i \kappa_\Si(x_i)\langle\nu_{\Si}(x_i), \tau_c(x_i)\rangle+\langle \nu_{\Si}(x_i), \partial_\delta\nu_c(x_i)\rangle\\
&= -\kappa_\Si(x_i) \langle \partial_\delta z,Y_i(c^*)\rangle+\langle \nu_{\Si}(x_i), \partial_\delta\nu_c(x_i)\rangle.
\eeqa
Since $\al_i(c^*)=\pi/2$ and since $\pd x_i$ is tangential to $\Si$ we can now use that the first and third  term of the right hand side of 
$$\partial_\delta\nu_{c}(x_i)=\pd  [r_{c_\de}^{-1}( z_\de-x_i(c_\de))]  
=\pd(r_{c_\de}^{-1}) (z^*-x_i(c^*)) +\kappa_{c^*} \pd z-\kappa_{c^*} \pd x_i
$$
are normal to $\nu_\Si(x_i)$ to conclude that 
\[\partial_\delta \alpha_i = \kappa_\Si(x_i) \langle \partial_\delta z,Y_i(c^*)\rangle-\kappa_{c^*}\langle \nu_{\Si}(x_i), \partial_\delta z\rangle.\]
Hence \eqref{eq:sec-var-step1} allows us to write
\begin{equation} \label{eq: 2nd var LLeta}
d^2\LLeta(z^*)(\peps z,\pd z)=\partial_\epsilon \partial_\delta L(c_{\epsilon,\delta})=\sum_{i} \kappa_\Sigma(x_i)\langle \partial_\delta z,Y_i(c^*)\rangle \langle \partial_\epsilon z, Y_i(c^*)\rangle-\kappa_{c^*}\sum_i \langle \partial_\delta z,\nu_{\Si}(x_i)\rangle\langle \partial_\epsilon z, Y_i(c^*)\rangle.    
\end{equation}
As $Y_i(c^*)$ is described by 
\eqref{eq:writing-Y-basis} in the basis $v_{\perp}, v_\parallel$ and as 
$$\nu_{\Si}(x_i)=R_{\pm \pi/2} \nu_c(x_i)=R_{\mp(\pi/2-\beta)}v_\perp =\sin\beta\, v_\perp \mp \cos\beta \, v_\parallel$$
this immediately yields the claimed formula \eqref{eq:Hessian-in-v} for the Hessian. 
\end{proof}

We can now use these basic properties of $\LLeta$ to show 
\begin{lemma}\label{lemma:2d} 
Let $\eta>0$ and assume that either $\Si$ is a circle or that $(\Si,\eta)$ satisfies the non-degeneracy assumption \eqref{ass:LojassIntro}.
Then 
there exists a constant $C_5=C_5(\eta, \Si)$ so that for every $c\in\Circeta$ there exists 
$c^*\in \mathcal{C}_\eta^*$ with 
\begin{equation}\label{est:2D-dist-loj}
 \|c - c^*\|_{C^1([0,1])} \leq C_5 \eps_\al(c) \quad\text{ where } \quad\eps_\al(c):= \abs{\al_1(c)-\pihalf}+\abs{\al_2(c)-\pihalf}.
\end{equation}
\end{lemma}

\begin{proof}
We note that the lemma is trivially true (with $c^*$ chosen as a global minimizer of $A_\Si$) for curves with $\max\abs{\al_i(c)-\pi/2}>\pi/4$. Hence we only need to consider circular arcs with $\abs{\al_i(c)-\pi/2}\leq \pi/4$ for $i=1,2$ which we can furthermore assume to be positively oriented. As this subset of $\Circeta$ is compact and as our assumption on $(\Si,\eta)$ ensures that the set of critical points is finite (up to symmetries in the case of the circle), the claim of the lemma hence follows provided we show that 
for any $c^*\in \Crit$ there exist $C>0$ and a neighborhood $\hat U_{c^*}$ of $c^*$ in $\Circeta$ so that \eqref{est:2D-dist-loj} holds true. 
As $\Circeta$ can locally be represented by $z\mapsto c_{r(z)}(z)$ as described above and as $\LLeta$ and $L$ are related as described in Lemma \ref{lemma:relation-Leta-L}, this follows provided we prove that for every positively oriented  
$c^*\in \Crit$ there exist $\eps>0$ and $C>0$ so that 
 for any $ z\in B_{\eps}(z^*)$, where $z^*:= z_{c^*}$ is the centre of $c^*$, there exists $\hat z^*$ with $\na \LLeta(\hat z^*)=0$ so that 
\beq \label{claim:LLeta-loj}
\abs{z-\hat z^*}\leq C\abs{\na \LLeta(z)}.
\eeq

We first prove that this holds for $\hat z^*=z^*$ in the case  where
 $(\Si,\eta)$ satisfies the non-degeneracy assumption \eqref{ass:LojassIntro}, i.e. where the 
 eigenvalues $\lambda_{1,2}$ of the Hessian $d^2\LLeta(z^*)$ are non-zero. In this case we set  $\lambda_0:=\min(\abs{\lambda_1},\abs{\lambda_2})>0$ and use that $z\mapsto \LLeta(z)$ is $C^2$ and that 
 $z\mapsto \al_i(z):=\al_i(c_{r(z)}(z))$ is $C^1$ with $\al_i(z^*)=\pi/2$
 to choose 
$\eps>0$ (depending on the modulus of continuity of $d^2\LLeta$ and thus on $\eta$ and the $C^2$ norm and modulus of continuity of the second derivative of the arc length parametrization of $\Sigma$) so that 
\beq 
\norm{d^2\LLeta(z^*)-d^2\LLeta(z)}\leq \tfrac12 \lambda_0 \quad \text{ and } \quad \abs{\al_i(c_{r(z)}(z))-\tfrac\pi2}\leq \tfrac \pi4 \quad \text{ for all } z\in B_\eps(z^*).
\eeq
We then let $E_{1,2}$ be orthonormal eigenvectors of $d^2\LLeta(z^*)$ to the eigenvalues $\lambda_{1,2}$, and given any unit vector $w\in \R^2$ consider the unit vector  $\tilde w:= \sum_i\text{sign}(\lambda_i)\langle w,E_i\rangle E_i$ which is chosen so that   
$$d^2\LLeta(\tilde z)(w,\tilde w)\geq d^2\LLeta(z^*)(w,\tilde w)-\thalf \lambda_0=\abs{\lambda_1}\langle w,E_1\rangle^2+\abs{\lambda_2}\langle w,E_2\rangle^2-\thalf \lambda_0\geq \thalf \lambda_0$$
for all $\tilde z\in B_\eps(z^*)$.
Given $z\in  B_\eps(z^*)$ we apply this for $\tilde z_t:=z^*+t(z-z^*)$, $t\in [0,1]$ and $w=\frac{z-z^*}{\abs{z-z^*}}=\frac{\partial_t \tilde z_t}{\abs{z-z^*}} $. As  $\nabla\LLeta(z^*)=0$ this yields  
\beqa
\abs{\na \LLeta(z)}&\geq \nabla\LLeta(z)\cdot \tilde w=(\nabla\LLeta(z)-\nabla\LLeta(z^*))\cdot \tilde w=\int_0^1 \ddt \nabla\LLeta(\tilde z_t)\cdot \tilde w dt \\
&=\abs{z-z^*}\int_0^1 d^2\LLeta(\tilde z_t)( w,\tilde w) dt\geq \thalf \lambda_0 \abs{z-z^*}
\eeqa
hence establishing that 
$$\abs{z-z^*}\leq 2\lambda_0^{-1}\abs{\na \LLeta(z)}   \qquad\text{ for all } z\in B_\eps(z^*) .$$

It hence remains to prove the analogous claim in the case where $\Si$ is circle, without loss of generality given by $\Si=\partial B_{\rho_0}(0)$ for some $\rho_0>0$. 
As the symmetries of this setting ensure that $\LLeta(z)=\LLeta(\abs{z})$, we can assume without loss of generality that $z^*=z_1^* e_1$ where $z_1^*<0$ is given by a critical point of $x\mapsto \LLeta(x e_1)$ and it suffices to prove that there exists $\eps>0$ and $C>0$ so that 
\beq
\label{claim:dist-in-circle-case} 
\abs{z_1-z_1^*}\leq C\abs{\partial_{z_1}\LLeta(z_1 e_1)} \qquad \text{ for all } \abs{z_1-z_1^*}<\eps.\eeq

As $v_\perp(z^*)=e_1$ we can use that 
\eqref{est:d2L-positive} implies that  
$\partial^2_{z_1} \LLeta(z^*)=d^2 \LLeta(z^*)(v_{\perp},v_\perp)>0$. We can hence choose $\eps>0$ small enough so that $\partial_{z_1}^2\LLeta(z_1e_1)\geq \half \partial_{z_1}^2\LLeta(z^*) =:C^{-1}>0$ on $(z_1^*-\eps,z_1^*+\eps)$ to deduce that $\abs{\partial_{z_1}\LLeta(z_1e_1)}= \abs{\partial_{z_1}\LLeta(z_1e_1)-\partial_{z_1}\LLeta(z_1^*e_1)}\geq 
C^{-1} \abs{z_1-z_1^*}$. This  immediately implies \eqref{claim:dist-in-circle-case} for this choice of $C$ and hence completes the proof of 
the lemma in this second case where $\Si$ is a circle. 
\end{proof}

\subsection{Proof of Theorem \ref{thm:loj}}
\label{subsec:proof-thm-loj}
In this section we explain how the results from the previous two subsections allow us to complete the proof of the \Loj\ estimates \eqref{est:Loj-distance} and \eqref{est:Loj-length} claimed in Theorem \ref{thm:loj}.
So let $\bar L>0$ and $\bar \phi$ be any given numbers and let $\gamma\in \Aarea$ be so that $L(\gamma)\leq \bar L$ and $\abs{\phiturn(\g)}\leq \bar \phi. $ As discussed at the beginning of the section,
it suffices to consider curves $\gamma\in \B_\eta$ with $\eps(\gamma)\leq \eps_0$ for a fixed constant $\eps_0=\eps_0(\Si,\eta,\bar L,\bar \phi)>0 $, which we can in particular choose so that $\eps_0\leq \eps_1$ for the constant $\eps_1$ obtained in Lemma \ref{lemma:reduction}.

For such curves, Proposition \ref{cor:reduction} applies and ensures that there exists a circular arc $c\in \Circeta$ so that 
\beq \label{est:proof-thm-1-1}
\norm{\gamma-c}_{C^1(ds_\gamma)}\leq C_3 \eps_\kappa(\gamma) \quad \text{ and } \quad \abs{\al_i(c)-\al_i(\gamma)}\leq C_4 \eps_\kappa(\gamma), \quad i=1,2.
\eeq
Lemma \ref{lemma:2d} then yields the existence of $c^*\in \Crit$ for which we can bound  
\beq\label{est:proof-thm-1-2}
\norm{c-c^*}_{C^1(ds_\gamma)}\leq \max(1,L(\gamma)^{-1}) \norm{c-c^*}_{C^1([0,1])}
\leq C_5 \max(1,\IP(\eta)^{-1})\eps_\al(c)\leq C_5 \max(1,\IP(\eta)^{-1})\big[ \eps_\al(\gamma)+2C_4\eps_\kappa(\gamma)\big].\eeq
Combined, \eqref{est:proof-thm-1-1} and \eqref{est:proof-thm-1-2} provide the claimed distance \L ojasiewicz estimate \eqref{est:Loj-distance}.

It remains to show that this estimate on the $C^1$ distance of $\gamma$ to $c^*$ ensures that the difference of their lengths is controlled by \eqref{est:Loj-length}.
To this end we first note that 
the length
of the segments $\si_i$ of $\Si$ between $x_i(c^*)$ and $x_i(\gamma)$ is bounded by  
\beq
\label{est:l-si-i}L(\si_i)\leq C\norm{\gamma-c^*}_{C^0}\leq 
C\eps(\gamma)\leq \min\Big( \mfrac{\pi}{4\kappamaxSi},\big( \mfrac{\eta}{2\pi}\big)^\half\Big),\eeq where the last estimate holds after reducing $\eps_0$ if necessary.

We now note that since $c^*$ meets $\Si$ orthogonally, the turning angles of $c^*$ (when parametrized with positive orientation) and of the subarc $\si_{c^*}$ by which we close up $c^*$ are related by $\phiturn(c^*)=\pi+\phiturn(-\si_{c^*})$. This immediately implies that $\phiturn(-\si_{c^*})\in [0,\pi)$ and hence that $L(\Si\setminus \si_{c^*})\geq 
 \frac{\pi}{\kappamaxSi} $. 
If $\phiturn(-\si_{c^*})\in [\pi/2,\pi]$ we can furthermore bound $L(\si_{c^*})\geq \frac{\pi}{2\kappamaxSi}$ while otherwise we can use that $\phiturn(c^*)\in [\pi,\frac{3\pi}2]$ and hence 
 $L(\si_{c^*})\geq \abs{x_1(c^*)-x_2(c^*)}=(2-2\cos(\phiturn(c^*))^{1/2} r_{c^*}
\geq 2^\half  r_{c^*}\geq 2^\half (\eta \pi^{-1})^\half$. 
 The above estimate \eqref{est:l-si-i} hence in particular implies that  $\min(L(\si_{c^*}), L(\Si\setminus\si_{c^*}))\ge L(\si_1)+L(\si_2)$ which ensures that 
 $\si_1$ and $\si_2$ are disjoint.

We now consider the modified support curve  $\hat\Si$ which we obtain from $\Si$ by replacing these short segments $\sigma_i$ of $\Si$  with the 
line segments $\hat \si_i$ between 
 $x_i(\gamma)$ and $x_i(c^*)$ and denote by $\hat \B$ and $\Areahat:\hat \B\to \R$ the set of admissible curves and enclosed area for this modified support curve. 
 
The convexity of $\Sigma$ ensures that the sets $E_i$, $i=1,2$, which are enclosed by
$\si_i$ and $\hat \si_i$ are contained in the triangle formed by the line segment $\hat \si_i$, whose length is $\abs{x_1(c^*)-x_1(\gamma)}\leq C\eps(\gamma)$, and the two tangents $T_{x_i(c^*)}\Si$ and $T_{x_i(\gamma)}\Si$, whose intersection angles with $\hat \si_i$ can be no larger than $L(\si_i) \kappamaxSi\leq C\eps(\gamma)\leq \pi/4$. This implies that
\beq\label{est:hat-Area-1}
\abs{\hat E_i}
\leq \half \abs{x_i(c^*)-x_i(\gamma)}^2 \tan(\kappamaxSi L(\si_i)) \leq C \eps(\gamma)^3, \eeq
from which we deduce that  $\abs{\AreaSi(\tilde \gamma)-\Areahat(\tilde \gamma)}\leq C\eps(\gamma)^3 $
for all $\tilde \gamma\in \hat\B\cap \B $.
Applied for $c^*$ and $\gamma$ and combined with the fact that $\AreaSi(\gamma)=\eta=\AreaSi(c^*)$ this in particular ensures that 
\beq\label{est:Area-hat-Area}
\abs{\Areahat(c^*)-\Areahat(\gamma)}\leq C\eps(\gamma)^3.\eeq
We now want to argue that this implies that the first variation of both the modified area functional and of the length functional at $c^*$ in direction of the vector field $X(p)=(\gamma(p)-c^*(p))$,  are of order $O(\eps(\gamma))$. To this end, we note that since the  modified support curve is flat between the respective endpoints of $c^*$ and $\gamma$,  the curves 
 $\gamma_t\in C^1([0,L(\gamma)],\R^2)$, $t\in [0,1]$, 
 obtained by interpolating linearly  
$$\gamma_t(p)=c^*(p)+t(\gamma(p)-c^*(p)), \quad p\in [0,L], \quad L:=L(\gamma)$$
between the constant speed parameterizations of $c^*$ and $\gamma$,
are all in $\hat \B$ and so that $\partial_t \gamma_t=X$ for every $t$.
While  these parameterizations of $\gamma_t$, $t\in (0,1)$ will in general not be by constant speed, we can use that
$\norm{\abs{\gamma'_t}-1}_{C^0}\leq \norm{\gamma_t-\gamma}_{C^1([0,L])}\leq \norm{c^*-\gamma}_{C^1([0,L])}\leq 
C\eps(\gamma)\leq \half$, where the last step holds after reducing $\eps_0$ if necessary, and hence that  $\norm{\abs{\gamma_t'}^{-1}-1}_{C^0}
\leq 2 \norm{c^*-\gamma}_{C^1([0,L])}\leq C\eps(\gamma)$.

As the first variation of the length along these curves can be computed as  $dL({\g_t})(X) 
=\int_{0}^{L}|\gamma_t^{\prime}(p)|^{-1}
\gamma_t^{\prime}(p) \cdot X^{\prime}(p) d p $ 
   while $d A_{\hat \Si}(\gamma_t)(X)= -\int_{0}^{L} X(p) \cdot J\gamma_t'(p) dp,$ we can bound  
\beq
\label{est:firstvar-hat-nearly-same}
\abs{dL(\gamma_t)(X)-dL(c^*)(X)}+
\abs{d\Areahat(\gamma_t)(X)-d\Areahat(c^*)(X)}
\leq C\norm{\gamma_t-c^*}_{C^1}\cdot \norm{X}_{C^1}\leq  C\eps(\gamma)^2.\eeq
Writing $\Areahat(\gamma)-\Areahat(c^*)=\int_0^1 \ddt \Areahat(\gamma_t) dt =\int_0^1 d \Areahat(\gamma_t)(X)dt$ and using  \eqref{est:Area-hat-Area} and \eqref{est:firstvar-hat-nearly-same}
we can thus conclude that 
\beqa
\label{est:dhatArea}
\Big|\int_{c^*} X\cdot \nu_{c^*} ds_{c^*}\Big| &=\abs{d A_{\hat \Si}(c^*)(X)}
\leq \abs{A_{\hat \Si}(\gamma)-A_{\hat \Si}(c^*)}+ 
C\eps(\gamma)^2
\leq C\eps(\gamma)^3+C\eps(\gamma)^2\leq C \eps(\gamma)^2.
\eeqa
Since the curvature $\kappa_{c^*}$ of the circular arc $c^*$ is constant, the  
first variation of the length along the (in general not area preserving) vector field $X$ can be written as 
$$dL(c^*)(X)=-\kappa_{c^*} \int  X\cdot \nu_{c^*} ds_{c^*}+\langle X(L),\tau_{c^*}(L)\rangle -\langle X(0),\tau_{c^*}(0)\rangle. $$
Since $X$ is parallel to the line segment $\hat \si_i$ at the endpoints and since the 
intersection angles $\hat \al_i(c^*)$ between $c^*$ and $\hat \si_i$ are so that $\abs{\hat \al_i(c^*)-\pi/2}=
\abs{\hat \al_i(c^*)-\al_i(c^*)}=\angle(\hat \si_i,T_{x_i(\gamma)}\Si)\leq C\eps(\gamma)$,
we can hence bound
\beqa
\label{est:dlength-hat}
\abs{dL(c^*)(X)}
&\leq C \eps(\gamma)^2+C\eps(\gamma)\left(\cos(\hat\al_1(c^*))+\cos(\hat\al_2(c^*))\right)
\leq C\eps(\gamma)^2.
\eeqa
Combined with \eqref{est:firstvar-hat-nearly-same} we thus conclude that 
\beq
\label{est:dist-L-proof}
\abs{L(c^*)-L(\gamma)}=\Big|{\int_0^1 \ddt L(\gamma_t)dt}\Big|=\Big|{\int_0^1 dL(\gamma_t)(X)dt}\Big|
\leq \abs{dL(c^*)(X)}+C\eps(\gamma)^2\leq C\eps(\gamma)^2
\eeq
which establishes the second claim \eqref{est:Loj-length} of Theorem \ref{thm:loj}.

\begin{remark}
{\rm 
If $\Si$ is analytic, then the expected analogues of the \Loj\ estimates \eqref{est:Loj-distance} and \eqref{est:Loj-length} hold (without imposing any form of non-degeneracy). Namely,
there exist $\beta_{1,2} = \beta_{1,2}(\eta, \Si)\in (0,1]$ such that \eqref{est:Loj-distance} holds with exponent $\beta_1$ replacing $1$ on the right-hand side, and \eqref{est:Loj-length} holds with exponent $1+\beta_2$ replacing $2$ on the right-hand side. 
Indeed, as the map $z\mapsto \LLeta(z)$ is analytic whenever $\Si$ is analytic, the classical \L ojasiewicz inequality for analytic functions on finite dimensional spaces guarantees that 
$\abs{z-\hat z^*}\leq C\abs{\na \LLeta(z)}^{\beta_1}$
for some $\beta_1 \in (0, 1]$.
The first estimate can be shown in the same way as \eqref{est:Loj-distance} but with this \L ojasiewicz inequality in place of the estimate \eqref{claim:LLeta-loj}.
\\
To prove the second estimate, observe that, up to minor modifications, one could replace the arc $c^*\in \Crit$ with the arc $\tilde{c}_\g \in \Circeta$ obtained in \Cref{cor:reduction} in the estimates starting from \eqref{est:l-si-i} and ending with \eqref{est:dist-L-proof}, to obtain $|L(\tilde c_\ga) -L(\g)| \leq C \e(\g)^2$ in place of \eqref{est:dist-L-proof}. Next, the classical finite dimensional (gradient) \Loj\ estimate for analytic functions guarantees that $|\LLeta(c_\g) - \LLeta(c^*)| \leq C |\nabla \LLeta(c_\g)|^{1+\beta_2}$ for some $\beta_2 \in (0,1]$. Pairing  \Cref{lemma:relation-Leta-L},  with the estimate \eqref{est:proof-thm-1-1}, we know $|\nabla \LLeta(\tilde{c}_\g)| \leq C\e(\g)$, thus allowing us to conclude the second estimate.

Examples of analytic support curves $\Si$ for which there exist $\eta>0$ and $c^*\in \Crit$ for which the non-degeneracy condition is violated but for which the resulting non-trivial Jacobi-field $v_\parallel$ is not generated by a family of critical points can e.g. be obtained by taking carefully chosen perturbations $\Si$ of a circle $\Si_0$: Fixing a circular arc $c^*$ that meets $\Si_0$ normally and choosing 
$\Si$ so that it agrees with $\Si_0$ up to order $2$ ensures 
that $c^*$ is also a critical point with non-trivial Jacobi-field $v_\parallel$ for the perturbed support curve $\Si$ and the corresponding $\eta=A_\Si(c^*)$. Further constraining the perturbation to be symmetric with respect to the line $\ell_\perp =\frac{1}{2}(x_1+x_2)+\langle v_\perp\rangle$  and so that 
$\kappa_{\Si}$ is strictly monotone near $x_{1,2}$ then excludes the possibility 
that $v_\parallel$ is generated by a family of critical points as these conditions ensure that any circular arc $\tilde c$ that meets $ \Si$ orthogonally in points close to $x_{1,2}$ must be 
symmetric with respect to $\ell_\perp$ (note that such $\tilde c\neq c^*$ will not enclose the prescribed area $\eta=A_{\Si}(c^*)$ and that $c^*$ is indeed an isolated critical point of the length on $\mathcal{B}_\eta$).\footnote{E.g. for  $\Si$  parametrized by $[0,2\pi] \ni\varphi\mapsto (1+ a\cos^3(2\varphi))e^{i\varphi} $, $\abs{a}>0$ small, the circular arc $c^*$ which intersects $\Si$ orthogonally at $e^{\pm i\pi/4}$ is such an isolated degenerate critical point for which the above quantitative estimates are not valid for $\beta_1=\beta_2=1$.}}
\end{remark}

\section{Exponential convergence of the gradient flow}
In this section, we recall some background on the area preserving curve shortening flow with Neumann free boundary conditions, including results of the first author that will be crucial to apply the flow to prove \Cref{thm: stability} in the next section. Then, we prove \Cref{mainthm-flow}.

\subsection{Background on the flow}
Parabolic regularity theory implies that a solution of (\ref{equ:apcsf}) with initial regularity $C^{2+\alpha}$ satisfies 
\[
\gamma \in C^{2+ \alpha, 1+ \frac{\alpha}{2}}\left(\Dom \times [0,T_{max}),\R^2\right) \cap  C^{\infty} \left(\Dom  \times(0,T_{max}), \R^2\right) \ \ \ \text{ for } 0<\alpha<1
\]
where $ C^{2+ \alpha, 1+ \frac{\alpha}{2}}$ denotes the usual parabolic H\"older spaces and $T_{max}>0$ is the maximal time of existence.
We will use the notation $\gamma_t:=\gamma(\cdot,t)$. This flow was constructed as (formal) $L^2$-gradient flow of the length with the constraint that the enclosed area is constant. It satisfies in particular  
\begin{equation}
    \label{eqn: dt length}
\frac{d}{dt}L(\gamma_t) = -\|\kappa_{\g_t}-\bar\kappa_{\g_t}\|_{L^2(ds_{\gamma_t})}^2\,.
\end{equation}
Due to the flow's non-local nature and the free boundary condition, preserved properties are rare to find. In~\cite{EMB2} it was shown that convexity of $\gamma_t$ is preserved under the flow. However, embeddedness and the property of being contained in $\R^2 \setminus \Omega$ is not preserved in general. Thus, for general initial data, understanding how to close up the curve to define the enclosed area is \emph{subtle}, as doing this naively might produces a family of closed curves whose enclosed area jumps by $\pm|\Omega|$. An example to illustrate this behavior can be found in Figure~\ref{Fig:self-int}. 

In this paper, we only deal with flows that stay outside of $\Omega$, and therefore we avoid this subtlety. Instead, Definition~\ref{def: sigma_gamma} gives a simple canonical way to close up the curve and hence to define the enclosed area, which by Remark~\ref{rmk:Asi-continuous} is continuous along the flow.

  \begin{figure}
 \begin{center}
 \begin{tikzpicture}[scale=.75]
\draw (0,0) circle (1.5cm);
\draw[out=90,in=0, thick] (0,1.5) to (-0.25,2.3);
\draw[->,out=180,in=90, thick] (-0.25,2.3) to (-2.3,0);
\draw[out=-90,in=180, thick] (-2.3,0) to (0,-2.5);
\draw[out=0,in=-90, thick] (0,-2.5) to (2.5,0);
\draw[out=90,in=0, thick] (2.5,0) to (0.4,2.3);
\draw[out=180,in=90, thick] (0.4,2.3) to (0.1,1.52);
\coordinate[label = left: ${\gamma_t}$] (A) at (-2.4,0);
\coordinate[label = -90: ${\Sigma}$] (B) at (-0.5,1.2);
\fill (-0.015,1.5) circle (1.2pt);
\fill (0.1,1.5) circle (1.2pt);
 \end{tikzpicture}
	  \end{center}
	   \caption{A solution of the flow  $\gamma_t$ can cross at the endpoints $\gamma_t(a_1)$, $\gamma_t(a_2)$ at some time $t>0$. If we naively close the curves $\gamma_t$ for each $t$ by connecting $\gamma_t(a_2)$ with $\gamma_t(a_2)$ via the positive orientation of $\Sigma$, then the  enclosed area jumps by $|\Omega|$ or $-|\Omega|$. Our choice of boundary curves closing up $\gamma_t$, Definition~\ref{def: sigma_gamma}, ensures that the algebraic enclosed area is continuous with respect to $t$ along the flow.} \label{Fig:self-int} 
	  \end{figure}

The following result due to the first author in \cite{EMB2} shows that, if the initial curve $\gamma_0$ is short enough in relation to the maximal curvature of $\Sigma$ and if its shape is not too bad regarding its isoperimetric quotient, then the above mentioned pathological behavior does not appear. In particular, the curves stay embedded and stay outside of $\Omega$. This theorem will be essential in the proof of \Cref{thm: stability} together with \Cref{mainthm-flow}.

\begin{theorem}[\cite{EMB2}] \label{thm: convex data flow}
Let $\Sigma$ be a convex $C^2$-Jordan curve and suppose $\eta \in (0, \areaC \kappamaxSi^{-2})$. 
Let $\g_0 \in \B_\eta$ be an embedded, convex curve of class $C^{2,\alpha}$ satisfying
\begin{align} \label{cond:converg}
 L(\gamma_0) < \tfrac{4}{5 \kappamaxSi} \arcsin(\eta/ L(\gamma_0)^2).
\end{align}
Then the solution $\{\ga_t\}_{t>0}$ of the free boundary area preserving curve shortening flow emanating from $\ga_0$ exists for all $t>0$. Moreover,  $|\int\kappa ds_{\gamma_t}|\leq 2\pi$ and for each $t\geq0$, $\gamma_t$ is embedded and intersects $ \Omega$ only at the endpoints. 
In particular, $\gamma_t\in \B_\eta$ for all $t\in[0,\infty)$. 
\end{theorem}
\begin{proof}
This theorem is a summary of the several main results of \cite{EMB2}. We note that condition (\ref{cond:converg}) implies $L(\gamma_0)<\frac{1}{2\kappa_{max}(\Sigma)}<d_\Sigma$,
 where  $d_\Si$ is the width defined in \eqref{def: width}. By \cite[Proposition~2.7]{EMB3}, if $\{\g_t\}_{t \in [0,T]}$ is a solution to the flow \eqref{equ:apcsf} with initial curve $\g_0$ satisfying $L(\g_0) <d_\Si$, then for every $t \in [0,T_{max})$, the  closeup curve $\sigma_{\g_t}$ from \Cref{def: sigma_gamma} has turning angle at most $\pi$, and moreover there exists some $l \in \mathbb{Z}$ such that 
$
(2l-2)\pi < \int \kappa_{\g_t} \, ds_{\g_t} < 2 l\pi$ for all $t \in [0,T_{max}).$
The assumption that $\g_0$ is convex and embedded guarantees that $l=1$ {if $\ga_0$ is positively oriented and $l=0$ if $\ga_0$ is negatively oriented}.\\
Furthermore, by \cite[Theorem~5.6]{EMB2}, we know that $\gamma_t$ stays outside of $\bar\Omega$ and, together with the line segment from $x_2(\gamma_t)$ to $x_1(\gamma_t)$, $\gamma_t$ traces out a convex domain for $t\in[0,\infty)$. In particular, 
$\g_t \in \B_\eta$.
\end{proof}

\subsection{Proof of \Cref{mainthm-flow}}
In this section, we prove \Cref{mainthm-flow}.
In the proof, we will need to compare the $L^2$ norm of the velocities of $\g_t$ and of its constant speed reparametrization on a fixed interval:

\begin{lemma}\label{lem:L2equivalence-general}
  Let $T>0$ and $\gamma: \Dom\times[0,T)\to\R^2$ be a $C^{2,1}$-family of curves moving in normal direction, i.e.\ $(\partial_t \gamma)^T=0$. We denote by $\tilde\gamma$ its orientation preserving reparametrization by constant speed on $[0,1]$. Then we have
  \begin{align}\label{est:L2-equiv-general}
\left| L(\gamma_t) \cdot \|\partial_t\tilde\gamma_t\|_{L^2([0,1])}^2 -  \|\partial_t \gamma_t\|_{L^2(ds_{\gamma_t})}^2\right| \leq\ 4 L(\gamma_t) \,\bigg(\int |\kappa||\partial_t \gamma|\,  ds_{\gamma_t}\bigg)^2.
  \end{align}
\end{lemma}

\begin{proof} 
We closely follow \cite{RuppSpener}.
We write $\g_t(p) = \g(p,t)$ and use the notation $\ga'(p,t)$ to mean $\partial_p \g(p,t)$.
For $t \in [0,T),$ consider the strictly increasing function $p\mapsto
  \varphi(p,t): = {L(\gamma_t)}^{-1} \int_{a_1}^p |\gamma'(q,t)|\, dq$ and let $\psi(\cdot,t):[0,1]\to\Dom$ denote its inverse. This way, the unit speed reparametrzation of $\g_t$ on $[0,1]$ is given by $\tilde\gamma(p,t):= \gamma(\psi(p,t),t)$. By the chain rule,
  \[
  \partial_t \tilde{\g}(p,t) = \partial_t\g (\psi, t) + \partial_t \psi(p,t) \, \ga'(\psi, t)  \qquad \text{ where }\psi =\psi(p,t).
  \]
 For the normal speed, the change of variable $p =\varphi(q,t) $ shows $\int_0^1 |\partial_t \g(\psi(p,t), t)|^2\,dp  =L(\g_t)^{-1} \| \partial_t\ga_t \|_{L^2(ds_{\g_t})}^2$. So, to prove the lemma, it remains to estimate the squared $L^2$ norm of the tangential speed, which by the same change of variable is
  \begin{equation}
      \label{eqn: reparam claim}
   I:=  \int_0^1 |\partial_t \psi(p,t)|^2\, |\g'(\psi(p,t)|^2 \,dp = L(\g_t)^{-1}\int_{a_1}^{a_2} \big|\partial_t\psi\big|_{(\varphi(q,t),t)} \big|^2 |\ga'(q,t)|^3 \, dq\,.
  \end{equation}
Toward this aim, we differentiate the identity $q  =\psi(\varphi(q,t),t)$ with respect to $t$ and $q$ to find 
\begin{align*}
    \partial_t\psi\big|_{(\varphi(p,t),t)} = - \partial_p\psi \big|_{(\varphi(p,t),t)}  \partial_t\varphi(p,t) = \mfrac{L(\gamma_t)}{|\gamma'(p,t)|} \partial_t\varphi(p,t).
\end{align*}
Now, from the Frenet equation $\partial_p \tau = \kappa \nu |\partial_p\gamma|$ and the fact that $\partial_t\g$ is orthogonal to $\tau$, we have 
 \begin{align*}
 \partial_t \int_{a_1}^p |\gamma'| &
 = \int_{a_1}^p \langle { \partial_t\gamma'}, \tau\rangle = \int_{a_1}^p \partial_p (\langle \partial_t\gamma ,\tau \rangle) - \langle \partial_t\gamma, \partial_p\tau\rangle = - \int_{a_1}^p  \langle \partial_t\gamma, \kappa\nu\rangle |\gamma'|.
 \end{align*}
Therefore,
 $  \partial_t \varphi(p,t) = -L(\g_t)^{-1}({\partial_t L(\gamma_t)} \,\varphi(p,t) + \int_{a_1}^p\langle \partial_t\gamma, \kappa \nu\rangle |\gamma'_t|).
$
Substituting this into the expression for $I$ in \eqref{eqn: reparam claim} and keeping in mind that $|\varphi|\leq 1$, we have  
\begin{align*}
 I  = L(\gamma_t)^{-1}\int_{a_1}^{a_2} \bigg({\partial_t L(\gamma_t)}\, \varphi(p,t)+ \int_{a_1}^q \langle \partial_t \gamma,\kappa\nu\rangle |\gamma'|\bigg)^2 ds_{\gamma_t}(q)\leq 2 \, (\partial_t L(\gamma_t))^2 + 2 \bigg( \int |\partial_t\gamma| |\kappa| ds_{\gamma_t} \bigg)^2.
\end{align*}
Finally, noting that $\frac{d}{dt} L(\gamma_t) = - \int \langle \partial_t\gamma,\kappa\nu\rangle\,  ds_{\gamma_t} \leq \int |\kappa||\partial_t \gamma|\,  ds_{\gamma_t}$ completes the proof.
\end{proof}

\begin{proof}[{Proof of \Cref{mainthm-flow}}]
{We proceed in several steps.} Let $\g_t$ be a global-in-time solution to the flow as in the statement of the theorem. We recall that $L(t):= L(\g_t)$ is monotone along the flow and bounded below by $L(t)\geq \IP(\eta)$. In particular,  and thus $L(\infty):=\lim_{t\to\infty} L(t)$ exists and is contained in $  [\IP(\eta), L(0)]$. Let $\bar L$ be an upper bound on $L(\gamma)$.

We recall that by the \Loj\ estimate \eqref{est:Loj-length} of Theorem~\ref{thm:loj}, there is a  constant $C_1=C_1(\eta, \Sigma, \bar{L}, \bar\phi)>0$ such that
\begin{align}\label{eqn: application of Loj}
| L(t) - {\ell_*(t)}|\leq C_1 \|\kappa_{\gamma} - \bar\kappa_{\gamma_t}\|_{L^2(ds_\gamma)}^{2}
\end{align}
for $\ell_*(t)$ chosen so that $|L(t) - {\ell_*(t)}| = \min \{|L(t)- L(c^*)| : c^* \in \Crit \}.$ 
Note that this minimum is achieved:
in the case of the circle, $\ell_0(\eta)=\IP(\eta)$ is the only critical value, while otherwise 
 our assumption on $(\Sigma,\eta)$ implies that the set of critical values of the length $L$ is discrete for the given $\eta$.
In the following, we denote by
\beq \label{eqn: energy levels}
\IP(\eta)=:\ell_0(\eta) < \ell_1(\eta) < ... < \ell_m(\eta)< ...
\eeq
the critical values of the length, and let $N$ be so that $\ell_{N-1}(\eta) \leq \bar{L}<\ell_N(\eta)$.

{\it Step 1:}
As a first step, we show there is a constant $C=C(\eta, \Sigma, \bar{L},\bar\phi)$ such that for any $0\leq t_1< t_2\leq \infty$,
        \begin{equation}\label{est:dist-geometric}
   \int_{t_1}^{t_2} \|\partial_t \gamma_t\|_{L^2(ds_{\gamma_t})} dt \leq C(L(t_1)- L(t_2))^{1/2} \,.
        \end{equation}
We divide the time interval $[0,T)$ into subintervals defined by $T_0=0<T_1 < ... <T_K,$
where $T_i$ are the times where $L(\gamma_t)$ either reaches a critical value or one of the midpoints of the intervals $[\ell_m(\eta),\ell_{m+1}(\eta)]$, i.e. $L(\gamma_t) \in\{\ell_m(\eta): m\leq N\} \cup \{\frac{1}{2}(\ell_{m+1}(\eta) - \ell_m(\eta)): m\leq N-1\}$. Notice that $K \leq 2N$ is bounded by a constant depending only on $L(\g_0)$ since the length is monotone along the flow.

On intervals $[T_i,T_{i+1}]$ on which   $L(\gamma_t)\in[\ell_m(\eta), \frac{1}{2}(\ell_{m+1}(\eta) - \ell_m(\eta))]$ for some $m$, we can apply \eqref{eqn: application of Loj} with $\ell_*(t)\equiv \ell_m(\eta)$, and use that length decays according to  \eqref{eqn: dt length}  to bound 
\begin{align*}
  -\mfrac{d}{dt} \left(L(\gamma_t) - \ell_m(\eta)\right)^{\frac{1}{2}}
  & = \mfrac{1}{2} \left(L(\gamma_t) - \ell_m(\eta)\right)^{-\frac{1}{2}}\left(-\mfrac{d}{dt}L(\gamma_t)\right)\\
  & = \mfrac{1}{2} \left(L(\gamma_t) - \ell_m(\eta)\right)^{-\frac{1}{2}} \int(\kappa_{{\ga_t}}-\bar\kappa_{{\ga_t}})^2 ds_{\gamma_t}\geq \mfrac{1}{2 C_1} \|\partial_t\gamma\|_{L^2(ds_{\gamma_t})}\,.
 \end{align*}
Integrating over subintervals $[t',t'']\subset [T_i,T_{i+1}]$ and using the fact that $(a-b)^2 \leq a^2-b^2$ for $a\geq b\geq 0$ yields
 \beq\label{eqn: intermed loj int}
 \int_{t'}^{t''} \|\partial_t\gamma_t\|_{L^2(ds_{\gamma_t})} dt \leq 2 C_1\left( (L(\gamma_{t'}) - \ell_m(\eta))^{\frac{1}{2}} - (L(\gamma_{t''}) - \ell_m(\eta))^{\frac{1}{2}} \right) \leq 2 C_1 \left(L(\gamma_{t'}) - L(\gamma_{t''})\right)^{\frac{1}{2}}\,.
 \eeq
 The same argument is applicable also 
 on intervals $[T_i, T_{i+1}]$ on which instead $L(\gamma_t)\in[\frac{1}{2} (\ell_{m-1}(\eta) - \ell_m(\eta)),\ell_m(\eta)]$,
 except that we have to 
 consider the evolution of the square root of $-(L(\gamma_t) - \ell_m(\eta))$ instead of $(L(\gamma_t) - \ell_m(\eta))$ 
 since $L(\gamma_t) \leq \ell_m(\eta)$ on such intervals.  Thus \eqref{eqn: intermed loj int} holds in this case as well.

Given arbitrary $0\leq t_1\leq t_2\leq \infty$, we can divide the interval $[t_1,t_2]$ into finitely many subintervals $I_k$ such that each $I_k$ is fully contained in one of the intervals $[T_i,T_{i+1}]$ 
and add the inequalities from the two cases above. Applying Cauchy-Schwarz inequality to bound the resulting sum  on the right-hand side and using that the resulting series is telescoping, this yields \eqref{est:dist-geometric} with $C = 2C_1K^{1/2}$.

\medskip 

{\it Step 2:}
Next, we prove analogue for the  constant speed reparametrization $\tilde{\g}_t$  of $\g_t$ on $[0,1]$, i.e. show that for any $0 \leq t_1<t_2\leq \infty$,
       \begin{equation}
           \label{eqn: dist reparam proof}
       \int_{t_1}^{t_2} \|\partial_t \tilde{\gamma}_t\|_{L^2([0,1])} \, dt \leq C(L({t_1})- L({t_2}))^{1/2} \,.
           \end{equation}
To this end, we first recall from Lemma  \Cref{lem:L2equivalence-general} that 
\begin{align*}
    \| \partial_t\tilde{\g}_t\|_{L^2([0,1])} \leq \mfrac{1}{L(\g_t)^{1/2}} \| \partial_t \g_t\|_{L^2(ds_{\g_t})} + 2 \| \kappa_{\g_t}\|_{L^2(ds_{\g_t})}\, \| \partial_t \g_t\|_{L^2(ds_{\g_t})}\,.
\end{align*}
Since $\vert\partial_t \g_t\vert 
= \vert \kappa_{\g_t} - \bar\kappa_{\g_t}\vert
$, we have 
$\| \kappa_{\g_t}\|_{L^2(ds_{\g_t}) }= ( \| \partial_t \g_t\|_{L^2(ds_{\g_t})}^2 + \bar{\kappa}_{\g_t}^2 L(\g_t))^{1/2}\leq  \| \partial_t \g_t\|_{L^2(ds_{\g_t})}+\bar{\kappa}_{\g_t}L(\g_t)^{1/2}
$. 
Moreover, $\g_t \in \B_\eta$ and thus $L(\g_t)\geq \IP(\eta)$. So, keeping in mind that $|\bar{\kappa}_{\g_{t}}| L(\g_t)^{1/2} = |\phiturn(\g_t)|\,L(\g_t)^{-1/2},$ we find 
\begin{align*}
      \| \partial_t\tilde{\g}_t\|_{L^2([0,1])} \leq \mfrac{2(1+ \bar\phi)}{\IP(\eta)^{1/2}} \, \| \partial_t \g_t\|_{L^2(ds_{\g_t})} + 2\| \partial_t \g_t\|_{L^2(ds_{\g_t})}^2\,.
\end{align*}
Using the fact that $\frac{d}{dt}L(\g_t)=  -\| \partial_t \g_t\|_{L^2(ds_{\g_t})}^2$ and integrating, we find 
\begin{align}\label{eqn: intermed}
    \int_{t_1}^{t_2} \| \partial_t \tilde{\g}_t\|_{L^2([0,1])}\,dt  \leq \mfrac{2(1+ \bar\phi)}{\IP(\eta)^{1/2}}   \int_{t_1}^{t_2}  \| \partial_t {\g}_t\|_{L^2(ds_{\g_t})} \, dt + 2( L({t_1}) - L({t_2}))\,. 
\end{align}
Combining \eqref{eqn: intermed} with \eqref{est:dist-geometric} and using that 
$L(t_1) - L(t_2)\leq \bar L^{1/2}(L(t_1) - L(t_2))^{1/2}$ 
completes the proof of \eqref{eqn: dist reparam proof}.

As \eqref{est:dist-geometric} and  \eqref{eqn: dist reparam proof}  are in particular 
applicable for $t_1 =0$ and $t_2 =\infty$, and in this case yield the claimed  estimate \eqref{est:dist-total}, it remains to prove the asymptotic convergence \eqref{eqn: exp conv}.
\medskip

{\it Step 3.} 
We recall that it is one of the main results of \cite{EMB2} that for global-in-time solutions with bounded  turning angle 
$|\bar\kappa(t)|\leq c_0$ and  $L(\gamma_t)\geq c_1>0$ for all $t\in[0,\infty)$, all derivatives of the curves are bounded $|\partial_t^k \partial_p^l \tilde\gamma| \leq C(k,l,\Sigma, c_0, c_1, L(\gamma_0))$ on $[0,1]\times[1,\infty)$, see also Theorem~2.11 in \cite{EMB3}.

At the same time, from the Sobolev inequality and Gagliardo-Nirenberg interpolation inequalities (see, e.g. \cite[p.125]{Nirenberg}), for any $k \in \mathbb{N}$, there is a constant $C=C(k)$ such that
\beqa\label{eqn: embed}
\| u\|_{C^{k}[0,1]} \leq C\norm{u}_{H^{k+1}([0,1])} &\leq 
C( \norm{u}_{H^{2k+2}([0,1])}^{1/2} \, \norm{u}_{L^2([0,1])}^{1/2} + \norm{u}_{L^2([0,1])})
\eeqa
 for any $u \in H^{2k+2}([0,1])$. 
As the estimate \eqref{eqn: dist reparam proof} obtained in step 2 allows us to bound 
\beq
\|\tilde \gamma_{t_1}- \tilde \gamma_{t_2}\|_{L^2([0,1])}  \leq\int_{t_1}^{t_2} \|\partial_t \tilde\gamma_t\|_{L^2([0,1])} \, dt \leq C (L(t_1)-L(t_2))^{1/2}
\eeq 
we can thus apply \eqref{eqn: embed} and the aforementioned derivative estimates to $\g_t$ for $t\geq1$ to find that 
\beqa\label{eqn: control by length}
\norm{\tilde \gamma_{t_1}-\tilde \gamma_{t_2}}_{C^k([0,1])} &\leq C_{k}
(L(t_1)-L(t_2))^{1/4} \text{ for any } k \text{ and for all } 1\leq t_1<t_2<\infty\,.
\eeqa
In particular,  the curves $\gamma_t$ are Cauchy with respect to any $C^k$ norm, and so converge smoothly to a unique limit $\gamma_\infty$, which must be an element of $\Crit$ since 
$\eps(\gamma_t)\to 0$, compare \eqref{est:Loj-distance}. 
Thanks to \eqref{eqn: control by length}, it remains to show that $L(t)-L(\infty)$ decays exponentially as $t \to \infty.$ Since this trivially holds for $t$ less than a fixed constant $t_1$ and suitably chosen $C$, it suffices to consider times $t\geq t_1$ with $t_1$ chosen below.

If $L(t) = L(\infty)$  for some $t<\infty$, then this trivially holds by monotonicity \eqref{eqn: dt length} of the length. Otherwise, we let $m$ be so that $L(\infty)=\ell_m(\eta)$ (recall \eqref{eqn: energy levels}) and
choose $t_1>1$ sufficiently large so that 
$L(t_1)\leq \half (\ell_m(\eta)+\ell_{m+1}(\eta)). $
This allows us to apply the  the \Loj\ inequality 
\eqref{eqn: application of Loj} with $\ell_*(t)=\ell_m(\eta)=L(\infty)$ on all of $[t_1,\infty)$. 

This ensures that for any $t\geq t_1$
\begin{align*}
-\frac{d}{dt} \log\left( L(t) - L(\infty)\right) = \frac{-\frac{d}{dt}L(t)}{ L(t) - L(\infty)} = \frac{\int (\kappa_{\g_t}-\bar\kappa_{\g_t})^2 ds_{\gamma_t}}{ L(t) - L(\infty)} \geq C_1^{-1}\,.
\end{align*}
Integrating this over $[t_1,t]$ yields 
$
-\log\left(L(t) - L(\infty) \right) \geq C_1^{-1}(t-t_1)  -\log\left(L(t_1) - L(\infty)\right)=C_1^{-1}t -C$ which yields the desired exponential decay 
$\left(L(t) - L(\infty) \right)\leq Ce^{C_1^{-1} t}$ that is required to complete the proof of the theorem.  
\end{proof}

\section{Quantitative stability}
This section is dedicated to the proof of \Cref{thm: stability}. 
We fix, for the entirety of the section, a convex body $\Omega$ whose boundary $\Sigma$ is a  $C^{2}$ curve with positively oriented parametrization $\si: \mathbb{S}^1 \to \R^2$.
Recall the universal constant $\areaC$ defined in \eqref{def: c area bound}. Given a set of finite perimeter $E$ in $\R^2\setminus\Omega$ with $|E|=\eta,$ define the isoperimetric deficit
\beq\label{def:peri-defect}
\de_\eta(E):=P(E; \R^2 \setminus{\Omega})- \IP(\eta).
\eeq
\Cref{thm: stability} is shown in two steps.
First, in \Cref{lem:distances}, we prove \Cref{thm: stability} for relatively convex sets with acute contact angle and small isoperimetric deficit.
More precisely, we call an open set $F\subset \R^2 \setminus{\Omega}$ {\it relatively convex} (with respect to $\R^2 \setminus {\Omega})$ if $F$ is connected
and $F$ is the intersection of $\R^2 \setminus {\Omega}$ and an open convex set in $\R^2$. 
If $F$ is a relatively convex set whose boundary coincides with $\Si$ on a (connected) set of positive length, then $\partial F \setminus \Si$ is a (geometrically convex) rectifiable curve in $\R^2 \setminus\Omega$ with distinct endpoints $x_0, x_1 \in \Si$.
Let $H_*$ be the half-plane containing $F\setminus\Omega$ with $x_0$ and $x_1$ in $\partial H_*$. The interior contact angle $\alpha_0 \in (0,\pi)$ of $\partial F\setminus \Si$ at $x_0$ is defined as the smallest interior angle of a wedge containing $F\setminus\Omega$ formed by $H_*$ and a half plane $H$ with $x_0 \in \partial H$. The interior contact angle $\alpha_1$ at $x_1$ is defined analogously.
\begin{proposition}
    \label{lem:distances}
    Fix $\eta \in (0,\areaC\, \kappamaxSi^{-2}]$ and assume that either $\Sigma$ is a circle or that the pair $(\Sigma, \eta)$ satisfies  Assumption~\ref{ass:LojassIntro}.
    There are explicit constants $\delta_0 = \delta_0(\Sigma,\eta)>0$ and $C_0 = C_0(\Sigma, \eta)$ such that the following holds.
    Let $F \subset \R^2 \setminus {\Omega}$ be an open, relatively convex set with $|F|=\eta$  whose boundary coincides with $\Si$ on a set of positive length and whose interior contact angles are at most $\pi/2$. If $\delta_\eta(F) \leq \delta_0$, then
\begin{equation}
    \label{eqn: stability for good sets}
 \inf_{E_* \in \setmins} d_H(\partial F,\partial E_*)^2 +     \inf_{E_* \in \setmins} |F\Delta E_*|^2  \leq C_0\, \delta_\eta(F). 
\end{equation}
\end{proposition}
Recall that $\setmins$ denotes the collection minimizers of \eqref{eqn: isop problem def}.
\Cref{lem:distances} is shown in Section~\ref{ssec: stability for convex}. Next, through a by-hand reduction procedure, we show that it is always possible to reduce to the setting of \Cref{lem:distances}.
\begin{proposition}\label{prop: summary reduction}  Fix $\eta \in (0,\areaC\, \kappamaxSi^{-2}]$.
There are explicit constants $\delta_1 = \delta_1(\Sigma, \eta)>0$ and $C_1= C_1(\Sigma, \eta)>0$ such that, for any set of finite perimeter $E$ in  $\R^2 \setminus {\Omega}$  with $|E| =\eta$ and $\delta_\eta(E) \leq \delta_1$, there is an open, relatively convex set $F\subset \R^2 \setminus {\Omega}$ with $|F|=\eta$
 whose boundary coincides with $\Si$ on a connected subset of $\Si$ of positive length and whose interior contact angles are at most $\pi/2$
  such that
\begin{equation}\label{eqn: reduction bound}
    \delta_\eta(F) + |E\Delta F|^2 \leq C_1\,  \delta_\eta(E)\,.
\end{equation}
Moreover, if $\partial E \setminus \Omega$ is a rectifiable curve, then 
\begin{equation}\label{eqn: reduction bound 2}
    d_H( \partial E , \partial F)^2 \leq\,  C_1\, \delta_\eta(E)\,.
\end{equation}
\end{proposition}
\Cref{prop: summary reduction} is shown in \Cref{ssec: proof of summary reduction}. Once we have \Cref{lem:distances} and  \Cref{prop: summary reduction},  \Cref{thm: stability} follows in a straightforward manner:
\begin{proof}[Proof of Theorem~\ref{thm: stability}] 
Let $\bar\delta = \min \{\delta_0/C_1, \delta_1\}$ where $\delta_0$ is from \Cref{lem:distances} and  $\delta_1$ and $C_1$ are from \Cref{prop: summary reduction}. 
Since $|E\Delta E_*| \leq 2\eta$ for any $E_* \in \setmins$, the first statement \eqref{eqn: stability L1 theorem statement} in \Cref{thm: stability} holds trivially when  $\delta_\eta(E) > \bar\delta$ by choosing $c \leq \bar{\delta}/(4\eta^2).$ 

The second statement \eqref{eqn: stability thm statement hausdorff} also holds trivially when  $\delta_\eta(E) > \bar\delta$, after the following small argument. Up to a translation, assume $0 \in \Omega^\circ$. Choose $\rho_0$ such that $\Omega \subset B_{\rho_0}$ and let $\rho = \max\{\rho_0 + 2(2\eta/\pi)^{1/2}, 2\IP(\eta)\}$. From \eqref{est:rad-circ-lower}, we know any minimizer $E_* \in \setmins$ is contained in $B_{\rho}$.
Let $E$ be a set as in the second part of \Cref{thm: stability}, and first consider the case that $E \setminus  B_{2\rho}$ is nonempty. Let $R_E>2\rho$ be the smallest radius such that $E \subset B_{R_E}.$ Then $d_H(\partial E,\partial E_*)^2 \leq 4R_E^2$ for any $E_* \in \setmins$, while by the connectedness of $\partial E$, we have 
$P(E;\R^2 \setminus \Omega)^2 - \IP(\eta)^2 \geq 4(R_E-\rho)^2 - (\rho/2)^2 \geq \frac{15}{16}R_E^2\geq \half R_E^2$. Thus \eqref{eqn: stability thm statement hausdorff} holds for such a set with $c=1/8.$ Next consider the case when $E  \subset B_{2\rho}.$ In this case $d_H(\partial E, \partial E_*) \leq 4\rho$, and thus if $\delta_\eta(E) > \bar\delta$, the estimate holds by taking $c = \bar{\delta}/(16\rho^2).
$ 

We henceforth assume $\delta_\eta(E) \leq \bar{\delta}.$
Let $F$ be the set obtained by applying \Cref{prop: summary reduction} to $E$. Observe that $\delta_\eta(F) \leq C_1 \delta_\eta(E) \leq \delta_0$ thanks to \eqref{eqn: reduction bound} and the choice of $\bar\delta$, and thus $F$ satisfies the assumptions of \Cref{lem:distances}. So, letting $F_* \in \setmins$ achieve the infimum in $\inf_{E_* \in \setmins} |F\Delta E_*|^2$ 
and combining \eqref{eqn: stability for good sets} and \eqref{eqn: reduction bound} yields
\[
\inf_{E_* \in \setmins } |E\Delta E_*|^2  \leq |E\Delta F_*|^2 
\le 2 |F\Delta F_*|^2 + 2|E\Delta F|^2 \le 2C_0\delta_\eta(F) + 2C_1 \delta_\eta(E) \le (2C_0 C_1 + 2 C_1)\delta_\eta(E).
\]
Thus the first statement \eqref{eqn: stability L1 theorem statement} of \Cref{thm: stability} holds with $c = \min\{ \bar{\delta}/(4\eta^2), 1/( 2C_0 C_1 + 2 C_1)\}.$ When $\partial E$ is a rectifiable Jordan curve, the second statement \eqref{eqn: stability thm statement hausdorff} of \Cref{thm: stability} follows analogously using \eqref{eqn: reduction bound 2}.
\end{proof}

\subsection{Proof of \Cref{lem:distances}}\label{ssec: stability for convex}
To prove \Cref{lem:distances}, we need three preparatory  lemmas. The first lets us approximate a set $F$ as in the statement of \Cref{lem:distances} by a set bounded by $\Si$ and a  convex $C^{2,\alpha}$ curve meeting $\Si$ orthogonally.
\begin{lemma}\label{lem: mollify}   Fix $\eta \in (0,\areaC\, \kappamaxSi^{-2}]$ and $\alpha \in (0,1)$ and let $\delta_0 = \delta_0(\eta, \kappamaxSi)$ be chosen according to Lemma~\ref{lem: sat flow hp}.
 Let $F \subset \R^2 \setminus \Omega$ be an open, relatively convex set with $|F|=\eta$ and $\delta_\eta(F)\leq \delta_0$ such that $\partial F \setminus\Si$ meets $\Sigma$ with contact angles at most $\pi/2$.
Then for any $\e>0$, there is an open, relatively convex set $F_\e \subset \R^2 \setminus \Omega$ with $|F_\e|=\eta$ for which  $\partial F_\e \setminus \Si$ is given by a  convex curve of class  $C^{2,\alpha}$ which meets $\Sigma$ orthogonally at the endpoints and which is so that
\[
d_H(\partial F , \partial F_\e) + |F\Delta F_\e| + |\delta_\eta(F) - \delta_\eta({F_\e}) | \leq \e.
\]
\end{lemma}
The proof of \Cref{lem: mollify} is postponed to  \Cref{sec: technical appendix}. The next lemma will let us estimate the symmetric difference between 
the regions enclosed by curves $\gamma_t$ at different times along the flow.
\begin{lemma}\label{lem: L1 displacement}
Fix $\bar{L}>0$ and let $\ga_t \in \B$, $t \in [t_1, t_2]$, be a smooth family of embedded curves of class $C^{2}$ with $L(\g_t) \leq \bar{L}$ such that $\g_t$ meets $\Si$ orthogonally at its endpoints. Let $E_t$ be the open bounded set bounded by $\ga_t$ and the  sub-arc $\si_{\gamma_t}$ of $\Si$ described in Definition \ref{def: sigma_gamma} and let $V_t = \partial_t \ga_t \cdot \nu_{\g_t}$ denote the normal velocity. Then
\[
|E_{t_2} \Delta E_{t_1}| \leq \bar{L}^\half \int_{t_1}^{t_2} \|V_t\|_{L^2(ds_{\gamma_t})} \, dt.
\]
\end{lemma}

\begin{proof}
Fix $t_0 \in (t_1, t_2)$. Since the curves $\g_{t_0 + \e}$ and $\si_{\gamma_{t_0+\eps}}$ vary smoothly in $\e$, so does the area of the region $R_\e = E_{t_0 + \e} \Delta E_{t_0}.$ We claim that 
\begin{equation}\label{eqn: deriv bound}
    \frac{d}{d\e}\bigg|_{\e=0} |R_\e| \leq \int_{\ga_{t_0}} |V_{t_0}| \,ds_{\g_{t_0}} .
\end{equation}
Indeed, let $L= L(\ga_{t_0})$ and reparametrize each $\ga_t$ by constant speed on $[0,L]$. Since 
\[
\gamma_{t_0 + \varepsilon}(p) = \gamma_{t_0}(p) + \varepsilon \, \partial_t \gamma_t(p)\big|_{t = t_0} + o(\varepsilon)
\]
for $p\in [0, L]$, and since tangential motion does not affect the region enclosed by the curve, we have $|R_\e \Delta S_\e| = o(\e)$ 
where 
\[
S_\varepsilon := \left\{ \gamma_{t_0}(p) + q \nu_{t_0}(p) : p \in [0,L],\ q \in ( -\e V^-(p),\e  V^+(p))\right\}. 
\]
Here $V^\pm(p) := \max(\pm V_{t_0}(p), 0)$, so the interval takes the form $
[0,\e V_{t_0}(p) ] $ or $[\e V_{t_0}(p),0 ]$ .
Since $\gamma_{t_0}$ is smooth and embedded, there exists $\e_0 > 0$ such that the map $\Psi(p,q) = \gamma_{t_0}(p) + q \nu_{t_0}(p)$ is a diffeomorphism from $[0,L] \times (-\e_0, \e_0)$ onto its image.
 By the area formula, 
\[
|S_\varepsilon| = \int_0^L \int_{-\e V^-(p)}^{\e V^+(p)} \det D\Psi (p,q) \, dq \, dp  = \varepsilon \int_0^L |V_{t_0}(p)| \, dp + o(\varepsilon).
\]
In the second identity we use the fact that the Jacobian $\det D\Psi (p,q) = 1 - q \kappa_{\ga_{t_0}}(p)$ satisfies $\det D\Psi (p,q) = 1 + O(q)$.
Dividing by $\e$ and passing $\e\to 0$ establishes \eqref{eqn: deriv bound}.

 Since \eqref{eqn: deriv bound} holds for every $t_0 \in [t_1, t_2]$, integrating with respect to $t_0$ shows that 
\[
|E_{t_2} \Delta E_{t_1}| \leq \int_{t_1}^{t_2} \| V_t\|_{L^1(ds_{\gamma_t})} \, dt.
\]
The lemma then follows from H\"{o}lder's inequality and the assumed bound $L(\g_t) \leq \bar{L}.$
\end{proof}

The final preparatory lemma lets us upgrade from $L^2$ control to $H^1$ control of the displacement along the flow. The lemma holds in greater generality than stated, but we only state and prove it in the setting where it is applied.
\begin{lemma}\label{lem: cacciopoli} 
Fix $\eta\in (0, \areaC \kappamaxSi^{-2}]$ and let $\delta_0 =\delta_0(\eta, \kappamaxSi)$ be chosen according to \Cref{lem: sat flow hp}.
There exists a constant $C = C(\eta, \Si)$ such that the following holds.
Let $\gamma \in \mathcal{B}_\eta$ be an embedded curve with $L(\gamma)\leq \IP(\eta) + \delta_0$  meeting $\Si$ orthogonally and bounding an open, relatively convex set in $\R^2 \setminus \Omega$. Let $\gamma_* \in\Crit$ be the relative boundary $\partial E_*\setminus \Om$ for some $E_* \in\setmins$, and suppose both curves are parametrized by constant speed on $[0,1]$
and oriented so that the normal to the curve coincides with the inward unit normal of the bounded set. Then 
\begin{equation}
    \label{eqn: upgrade}
\int_0^1 |\gamma'(p) - \gamma_*'(p)|^2\,dp \leq C \left( L(\gamma) - L(\gamma_*) \right) + C \int_0^1 |\gamma(p) - \gamma_*(p)|^2\,dp.
\end{equation}
\end{lemma}

\begin{proof}
Let  $w(p) := \gamma(p) - \gamma_*(p)$. First, let us see the information we get from the fact that $\ga$ and $\ga_*$ both enclose area $\eta$.
Recall from \eqref{eqn: def area} that $A_\Sigma(\gamma) =- \frac{1}{2} \int_0^1 \gamma\cdot  J \gamma'\,dp - \frac{1}{2} \int_{\si_\g} \si_\g \cdot J\si_\g' \, ds_{\si_\g}$ where $\si_\g$ is as in Definition~\ref{def: sigma_gamma}.  
Since $E$ is relatively convex by assumption, $\si_\g$ coincides with the projection of $-\g$ onto $\Si$. The analogue holds for the subarc $\si_{*}= \si_{\g_*}$.

If the traces of $\si_\g$ and $\si_*$ do not intersect, then $\g$ and $\g_*$ lie in two separate half spaces. Denote by $l_0$ a line dividing them and $\Pi_{l_0}=\Pi_0$ the projection onto the line $l_0$. Then $\int_0^1 |\g-\g_*|^2 \geq \int_0^1 |\g_*-\Pi_0(\ga_*)|^2 = \frac{1}{\IP(\eta)}\int |\g_*-\Pi_0(\ga_*)|^2 \,ds_{\g_*}$, 
which is bounded below by the corresponding quantity for a semi-circle of the same radius and its diameter. Keeping in mind  \eqref{eqn: bounds on profile} and  $\pi r^2 \ge \eta$, compare \eqref{est:rad-circ-lower},  we hence get 
 $ \int_0^1 |\g-\g_*|^2 \geq \frac{1}{\IP(\eta)} r^2\int_0^\pi \sin^2\geq \frac{1}{2\pi^{1/2}\eta^{1/2}}\frac{\eta}{\pi} \frac{\pi}{2}=\eta^{1/2}/(4\pi^{1/2})$.
Since the left-hand side of \eqref{eqn: upgrade} is bounded above by $4(\IP(\eta)+\delta_0)^2$, the estimate \eqref{eqn: upgrade} holds trivially in this case.

We can thus assume the traces of $\si_\g$ and $\si_*$ intersect nontrivially. 
{To prove the claim in this main case, it is now convenient to fix the coordinate system so that the center $z_{\g_*}$ of the circular arc $\g_*$ is at the origin. This ensures that $\gamma_*(i)$ is normal to $\gamma_*'(i)$, $i=0,1$. As $\gamma_*$ intersects $\Si$ orthogonally, this hence allows us to use that $\gamma_*(i)$ is normal to $\nu_\Si(\gamma_*(i))$ in the following proof.
As a first step we show that 
\beq \label{est:I-claim-new}
I:=\Big|    \int_{\si_\g} \si_g \cdot J \si_g'\,d s_{\si_{\g}}
    - \int_{\si_*} \si_* \cdot J \si_*'\, ds_{\si_*}\Big|
\leq C (\abs{w(0)}^2+\abs{w(1)}^2)
\eeq
}
To see this we first note that the integrands $ \si_\g \cdot J \si_\g'$ and $ \si_* \cdot J \si_*'$ coincide on the intersection of their traces. The symmetric difference between their traces is parametrized by two sub-arcs $\si_0$ and $\si_1$ of $\Si$ such that the endpoints of $\si_0$ are $\g(0)$ and $\g_*(0)$, and the endpoints of $\si_1$ are $\g(1)$ and $\g_*(1)$. Moreover, since $\si_\g\cap \si_*\neq \emptyset$, we can bound 
since $L(\si_0) \leq \IP(\eta) + \delta_0< d_\Si/2$ 
by Remark~\ref{rmk: d sigma bound}. Hence 
there is an explicit constant $C$ depending on $\Si$ such that $L(\si_0) \leq C |\g(0) - \g_*(0)|= C|w(0)|$. Similarly, $L(\si_1) \leq C |w(1)|.$ 
The above quantity $I$ can hence be bounded by 
\beq \label{est:I-new}
I\leq 
\Big|    \int_{\si_0} \si_0 \cdot J \si_0'\,d s_{\si_{0}}\Big| 
    +\Big|    \int_{\si_1} \si_1 \cdot J \si_1'\,d s_{\si_{1}}\Big|  \leq L(\si_0)\sup_{\si_0}\abs{\si_0 \cdot \nu_{\Si}(\si_0)}  +L(\si_1)) \sup_{\si_1}\abs{\si_1 \cdot \nu_{\Si}(\si_1)}
\eeq

As remarked above, our choice of coordinate system ensures that $\gamma_*(0)$ is orthogonal to $\nu_{\Si}(\gamma_0^*(0))$. The inner product $\si_0\cdot \nu_\Si(\si_0)$ hence vanishes at one of the endpoints of $\si_0$,  namely at $\gamma_*(0)$.
Since $\Si$ is a fixed $C^2$ curve, this ensures that $\sup_{\si_0}\abs{\si_0 \cdot \nu_{\Si}(\si_0)}\leq \text{osc}_{\si_0}\abs{\si_0 \cdot \nu_{\Si}(\si_0)}\leq C L(\si_0)$ for a constant $C=C(\Si)$ that only depends on an upper bound on the $C^2$ norm of (the arclength parametrisation) of $\Si$. Inserting this, and the analogue bound on $\si_1$, into \eqref{est:I-new}, immediately yields the claimed estimate \eqref{est:I-claim-new}.

To address the other term coming from the difference of areas, we add and subtract terms, the integrate by parts, and use the identity $Ja \cdot b = -a \cdot Jb$ to find
\beqa \label{est:gamma J gamma'}
\int_0^1 \g \cdot J\g' - \int_0^1 \g_* \cdot J\g_*' & = \int_0^1  w \cdot J w' + \int_0^1  w\cdot  J \gamma_*'  +\int_0^1 \gamma_* \cdot J w'  \\
 & =  \int_0^1  w \cdot J w' + \int_0^1  w\cdot  J \gamma_*'  - \int_0^1 \gamma_*' \cdot J w  + (\gamma_*\cdot Jw)\Big|_{0}^1\\
 & = \int_0^1  w \cdot J w' +  2  \int_0^1  w\cdot  J \gamma_*'  + (\gamma_*\cdot Jw)\Big|_{0}^1.
\eeqa

{ While we could of course estimate the boundary terms 
 $ \gamma_*(i)\cdot Jw(i)=\gamma_*(i)\cdot J (\gamma_*(i)-\gamma(i))$ by a multiple of $\abs{w(i)}$, $i=0,1$, this would not suffice to prove our result. Instead, we can exploit that 
 \beq 
 \label{eq:writing-w} 
 w(i)=\gamma(i)-\gamma_*(i)=\pm L(\si_i)\tau_{\Si}(\gamma_*(i))+err_i\eeq 
 (with $\pm$ chosen according to the orientation of the subarc $\si_i$ from $\gamma_*(i)$ to $\gamma(i)$), for an error term that is bounded by $\abs{err_i}\leq L(\si_i)\text{osc}_{\si_i)} \tau_{\Si}\leq C L(\si_i)^2\leq C\abs{w(i)}^2$. 
As our choice of coordinate system ensures that $\gamma_*(i)$ is orthogonal to $\nu_\Si(\gamma_*(i))=J \tau_\Si(\gamma_*(i))$, we can hence indeed bound 
 $ \abs{\gamma_*(i)\cdot Jw(i)}\leq C\abs{w(i)}^2$ 
 }

Inserting this into \eqref{est:gamma J gamma'}, rearranging this identity and applying the bounds above and  the fact that $2 (A_\Sigma(\gamma) - {A}_\Sigma(\gamma_*))= 0$ by assumption, we hence find
\begin{equation}
    \label{eqn: area constraint info-new}
  2  \Big|  \int_0^1  w\cdot  J \gamma_*'  \Big| \leq  \| w\|_{L^2}\| w'\|_{L^2}  +C(  \,|w(0)|^2 + |w(1)|^2).
\end{equation}
Now we turn to the main estimate. Letting $\ell = L(\gamma)$ and $\ell_* = L(\gamma_*)$ and using the fact that both curves are parametrized with constant speed, we have $
\ell^2 
=  \ell_*^2 + 2 \, \gamma_*'(p)\cdot  w'(p) + |w'(p)|^2.$
Rearranging and integrating over $[0,1]$ gives
\begin{equation}
    \label{eqn: rev poincare 1-new}
 \int_0^1 |w'|^2 dp =
(\ell^2 - \ell_*^2) - 2 \int_0^1 \gamma_*' \cdot w'\, dp.
\end{equation}
We integrate the second term by parts, using that $\ga_*' = \ell_* \tau_{\g_*}$ and $\gamma_*''(p) = \ell_*\kappa_* J\ga_*'(p)$ for the (constant) curvature $\kappa_*$  of $\gamma_*$:
\begin{equation}\label{eqn: rev poincare 3}
-2\int_0^1 \gamma_*'\cdot  w' \, dp =  2\int_0^1 \gamma_*'' \cdot w\, dp -2  \gamma_*'\cdot w \Big|_{0}^{1}  =    2 \ell_*\kappa_* \int_0^1 J \gamma_*' \cdot w \, dp - 2\ell_* \tau_{\gamma_*}\cdot w \Big|_{0}^{1}\,.
\end{equation}
{As above, the boundary terms can be controlled by $C(\abs{w(0)}^2+\abs{w(1)}^2)$ since $\tau_{\gamma_*}$ is normal to $\tau_\Si$ at the endpoints and since $w$ can be written as in \eqref{eq:writing-w}}
So, applying \eqref{eqn: area constraint info-new} gives us 
 \begin{equation}
   2  \Big|\int_0^1 \gamma_*'\cdot  w' \,dp \Big| \leq \ell_*{\kappa_*}\| w\|_{L^2} \| w'\|_{L^2} + C(|w^2(0)|+|w^2(1)|)
 \end{equation}
 where $C$  depends only on $\Sigma$ and $\eta$ (from \eqref{eqn: bounds on profile} we can bound $\kappa_*$ above in terms of $\eta$). We estimate the boundary terms using the Sobolev embedding (which simply follows from the fundamental theorem of calculus and H\"{o}lder's inequality):
\[
\| w\|_{C^0}^2 \leq 2 \|w\|_{L^2} \|w'\|_{L^2} + \| w\|_{L^2}^2.
\]
Substituting  these estimates in \eqref{eqn: rev poincare 1-new} and using $\ell^2 - \ell_*^2 \leq (2\IP(\eta) +\delta_0) (\ell -\ell_*)$,  we find 
\begin{align}
    \int_0^1 |w'|^2 dp \leq (2\IP(\eta) +\delta_0) (\ell -\ell_*) + C'( \|w\|_{L^2} \|w'\|_{L^2} + \| w\|_{L^2}^2)
\end{align}
for $C'=C'(\Sigma,\eta)$. Applying Young's inequality and absorbing the resulting term $\frac12\int |w'|^2$ completes the proof.
\end{proof}

We are now ready to prove \Cref{lem:distances}.
\begin{proof}[Proof of \Cref{lem:distances}]
Thanks to  Lemma~\ref{lem: mollify}, it suffices to prove the proposition when  $\partial F \setminus \Sigma$ is parametrized by an embedded convex  $C^{2,\alpha}$ curve  $\gamma\in \B_\eta$ meeting $\Sigma$ orthogonally at the endpoints.

Taking $\delta_0 = \delta_0(\eta, \kappamaxSi)>0$ as in \Cref{lem: sat flow hp}, the assumption $\delta_\eta(F) \leq \delta_0$  and \Cref{lem: sat flow hp} together ensure that $\gamma$ satisfies the hypotheses of \Cref{thm: convex data flow}. \Cref{thm: convex data flow} guarantees the existence of a global-in-time solution $\{\ga_t\}$ to the free boundary area-preserving curve shortening flow with initial data $\ga=\ga_0$ such that  $|\int\kappa ds_{\gamma_t}|\leq 2\pi$ and $\g_t \in \B_\eta$ for all $t\geq0$.

By \Cref{mainthm-flow}, there is a unique arc $c^* \in \Crit$ such that ${\ga}_t$ converges (smoothly, exponentially) to $c^*.$ 
Since $L(c^*)\leq L(\ga)$, $c^*$ is a minimizer of $L$ in $\B_\eta$ provided we choose $\delta_0 < \ell_1(\eta,\Si) - \IP(\eta)$, where $\ell_1(\eta,\Si)>\IP(\eta)$ is the lowest energy level of a non-minimizing critical point.
Let $F_* \in \setmins$ denote the set bounded by $c^*$ and $\Si$. By 
\eqref{est:dist-total} and \Cref{lem: L1 displacement} (passing $t_1\to0$ and $t_2 \to \infty)$, we find 
 \[
 |F\Delta F_*  | \leq \int_0^\infty \|\partial_t\gamma_t\|_{L^2(ds_{\gamma_t})} \, dt  \leq C \delta_\eta(F)^{1/2}\,.
 \]
Next, to bound $d_H(\partial F, \partial F_*),$ let $\tilde{\g}, \tilde{\g}_t$, and $\tilde{c}^*$ be the constant speed parametrizations of $\g$, $\g_t$, and $c^*$ on $[0,1]$. We apply the Sobolev inequality and \Cref{lem: cacciopoli} to find 
\begin{equation}\label{est:dist-L2}
\begin{split}
      \| \tilde{c}^*- \tilde{\g}\|_{C^0([0,1])}^2  & \leq  C\| (\tilde{c}^*)'- \tilde{\g}'\|_{L^2([0,1])}^2 +  
     C  \| \tilde{c}^*- \tilde{\g}\|_{L^2([0,1])}^2  \leq C\delta_\eta(F)  + C  \| \tilde{c}^*- \tilde{\gamma}\|^2_{L^2([0,1])}\,.
\end{split}
        \end{equation}
Then, by \eqref{est:dist-total} we find 
\begin{align*}
     \| \tilde{c}^*- \tilde{\gamma}\|_{L^2([0,1])} & = \left\| \int_0^\infty \partial_t\tilde{\gamma}_t \,dt\right\|_{L^2([0,1])} \leq  \int_{0}^{\infty} \|\partial_t \tilde\gamma_t\|_{L^2([0,1])} dt \leq C\delta_\eta(F)^{1/2}.
\end{align*}
This completes the proof of the proposition.
\end{proof}

\subsection{Reduction to a set bounded by a rectifiable curve}\label{ssec: rect curve}
The remainder of the paper is dedicated to proving \Cref{prop: summary reduction}. The first step is to replace $E$ with a simply connected set. To do so, we need the following quantitative sub-additivity estimate for the isoperimetric profile. 
\begin{lemma}\label{lem: quant sublinearity}
Fix $\eta>0$.    There is a positive constant $c_0 = c_0(\kappamaxSi, \eta)$ such that the following holds. Let $\{\eta_i\}_{i \in I}$ be a non-increasing finite or countable sequence of positive numbers with $\sum_{i \in I }\eta_i =\eta.$ Then
    \begin{equation}\label{eqn: quant sublinearity}
        \sum_{i \geq 1} \IP(\eta_i) - \IP(\eta) \geq c_0 \sum_{i \geq 2} \eta_i^{1/2}.
    \end{equation}
\end{lemma}
\begin{proof} If the index set $I$ has cardinality $1$, there is nothing to show, so we assume it is at least $2$.
For each $i \in I$, let $r_i$ denote the radius of a minimizer of $\IP(\eta_i)$.\footnote{While not needed for the proof, we recall from Section~\ref{sec: preliminaries} that for almost every $\eta,$ there is a unique such $r$.} Up to reindexing, we may assume that $r_1 \geq r_2 \geq \dots$.  It suffices to show \eqref{eqn: quant sublinearity} for this reindexed sequence, since for the two sequences, the left-hand sides of \eqref{eqn: quant sublinearity} are equal and the right-hand sides are ordered.

Set  $\bar \e = \pi^{-1}\arctan(1/(r_1 \kappamaxSi))$. The bound \eqref{eqn: theta bound turning angle} guarantees that $\IP(\eta_i) \leq 2\pi (1-\bar\e) r_i$
for each $i \in I$. Combining this with  the lower bound \eqref{eqn: bounds on profile} on the isoperimetric profile shows that 
\begin{equation}\label{eqn: sublinear intermed bound}
\eta_i \leq \mfrac{ \IP(\eta_i)^2}{2\pi} \leq (1-\bar \e)  \IP(\eta_i) r_i \le  (1-\bar \e)  \IP(\eta_i) r_1.
\end{equation}

We construct a competitor for the area-$\eta$ isoperimetric problem as follows. 
 Up to a rotation and translation, we may assume that a minimizer $E_1$ of $\IP(\eta_1)$ is bounded by $\Si$ and a circular arc of radius $r_1$ centered at the origin with one endpoint at $(r_1,0)$ and the other endpoint in the third quadrant. 
 For each $i\geq 2,$ let $\ell_i = \eta_i/r_1$, so that $\ell_i \leq (1-\bar \e)\,\IP(\eta_i)$.
 Define the rectangle $R = (-r_1, r_1)\times (0, \sum_{i \geq 2 } \ell_i/2)$. 
Then, letting $H^{\pm}=\{(x,y) : \pm y>0\}$ and $E_1^{\pm} = E_1\cap H^{\pm}$,  
 consider the set 
$$F = E_1^- \cup R \cup (E_1^+ + (0, \textstyle\sum_{i \geq 2 } \ell_i/2)).
$$
The convexity of $\Omega$ guarantees that $F \subset \R^2 \setminus \Omega$, and by construction $|F| = \sum_{i \in I} \eta_i = \eta$ and 
\[
P(F; \R^2\setminus \overline{\Omega})= \IP(\eta_1) + \sum_{i \geq 2} \ell_i \leq \IP(\eta_1) + (1-\bar{\e}) \sum_{i \geq 2} \IP(\eta_i).
\]
Using $ P(F; \R^2\setminus \overline{\Omega}) \geq \IP(\eta)$ and applying the lower bound $\IP(\eta_i) \geq (2\pi \eta_i)^{1/2}$ from \eqref{eqn: bounds on profile} once again, the desired estimate \eqref{eqn: quant sublinearity} follows with the constant 
\[(2\pi)^\half \bar \e \geq \Big(\mfrac{2}{\pi}\Big)^{\half} \arctan\Big(\Big(\mfrac{\pi}{2\eta}\Big)^{\half} \kappamaxSi^{-1} \Big)=:c_0.
\]
This completes the proof.
\end{proof}

  \begin{lemma}\label{lem: red 1 simply connected}
  Fix $\eta>0$. 
  There exist positive constants $C_2= C_2(\kappamaxSi,\eta)$ and $\delta_2 =\delta_2(\kappamaxSi, \eta)$ such that the following holds.
Let $E \subset \R^2 \setminus \Omega$ be a set of finite perimeter with $|E|=\eta$ and $\delta_\eta(E) \leq \delta_2$.
Then there is a connected open set $F$ in $\R^2 \setminus {\Omega}$ with $|F| \geq \eta$
whose boundary is a rectifiable Jordan curve, coinciding with $\Si$ on a connected, positive $\mathcal{H}^1$-measure set, 
such that
       \begin{align}
   |E\Delta F|^{\half}   \leq C_2 \,\delta_\eta(E) \quad \text{ and } \quad        P(F; \R^2 \setminus \Omega) \leq P(E ; \R^2 \setminus \Omega) + C_2\,\delta_\eta(E)\,.
    \end{align} 
   \end{lemma}
\begin{proof}
{\it Step 1:}
Recall that a set of finite perimeter $G$ is said to be indecomposable if $|G_1|\,|G_2| = 0$ for  any disjoint sets $G_1,G_2$ such that $G= G_1 \cup G_2$  and $P(G) = P(G_1) + P(G_2)$. By \cite[Theorem 1]{AmbrosioJEMS}, $E$ admits a unique decomposition as the union of at most countably many pairwise disjoint indecomposable sets $\{E_i\}_{i \in I}$ such that $|E_i| >0$ and $P(E) = \sum_{i\in I } P(E_i)$. It follows from \cite[Proposition 3]{AmbrosioJEMS} and  Federer's and De Giorgi's theorems \cite[Theorem 16.2, Theorem 15.9]{MaggiBOOK} that $P(E;\R^2 \setminus{\Omega}) = \sum_{i \in I } P(E_i; \R^2 \setminus {\Omega})$ as well. Therefore, letting $\eta_i = |E_i|$ and reindexing so the $\eta_i$ are non-increasing,  \Cref{lem: quant sublinearity} implies that 
\begin{align}\label{eqn: tiny components}
    \delta_\eta(E) & \geq \sum_{i \in I} \IP(\eta_i) - \IP(\eta) \geq c_0 \sum_{i \geq 2} \eta_i^\half \geq c_0 \Big(\sum_{i \geq 2} \eta_i \Big)^\half =c_0 (\eta -\eta_1)^\half.
\end{align}
Here  $c_0 = c_0(\kappamaxSi, \eta)$ is the constant from \Cref{lem: quant sublinearity} and the final inequality follows from concavity. Therefore, the  indecomposable set $E_1$ satisfies 
\begin{equation}\label{eqn: one component bound}
    \delta_\eta(E) \geq c_0(\eta -\eta_1)^\half= c_0 |E \Delta E_1|^{\half} \quad \text{ and } \quad P(E_1;\R^2 \setminus \overline{\Omega}) \leq P(E;\R^2 \setminus \overline{\Omega})\,.
\end{equation}
We choose $\delta_2 \leq c_0\eta^{1/2}/2$, so that the first estimate in \eqref{eqn: one component bound} guarantees that $\eta_1 \geq \eta/2.$ 

We claim that $P(E_1 ; \R^2 \setminus \Omega) < P(E_1)$, which in turn guarantees that $\partial E_1$ intersects $\Si$ on a set of positive $\mathcal{H}^1$ measure and that $E_1 \cup \Omega$ is indecomposable. 
To see this, 
take a ball of area $\eta - |E_1|= |E \Delta E_1|$ at positive distance from $E_1$ and $\Omega$. By the first bound in \eqref{eqn: one component bound}, this ball has perimeter at most $C\delta_\eta(E)$ for  $C = C(\kappamaxSi, \eta)$. So, using the second bound in \eqref{eqn: one component bound}, the union $\tilde{E}$ of $E_1$ and this ball has perimeter 
\[
P(\tilde{E}; \R^2 \setminus \Omega) \leq \IP(\eta)+ C\delta_2 <  2\pi^\half \eta^\half.
\]
Here, the second inequality holds thanks to the upper bound \eqref{eqn: improved upper bound} for the isoperimetric profile, provided we choose  $\delta_2 \leq \frac{1}{C}\, 2\pi^\half\eta^\half\{1-  (1  -\pi^{-1} \arctan( \big( \tfrac{\pi}{2\eta}\big)^\half {\kappamaxSi}^{-1})^{\half}\}$. On the other hand, $P(\tilde{E}) \geq 2\pi^{1/2} \eta^{1/2}$ by the isoperimetric inequality. Thus $P(\tilde{E}; \R^2 \setminus \Omega) < P(\tilde{E})$, which by definition of $\tilde{E}$ proves the claim.

{\it Step 2: }
Next we ''fill in the holes'' of $E_1 \cup \Omega$. More specifically, since $E_1 \cup \Omega$ is indecomposable,
 \cite[Corollary 1]{AmbrosioJEMS} says that the essential boundary $\partial^M (E_1 \cup \Omega)$ admits a unique decomposition into at most countably many rectifiable Jordan curves $C^+$ and $\{C^-_j\}_{j \in J}$ with $\text{int}(C^-_j) \subset \text{int}(C^+)$ such that 
\[
E_1 \cup \Omega= \text{int}(C^+)\setminus \bigcup_{j \in J} \text{int}(C^-_j), \qquad P(E_1 \cup \Omega)  = \mathcal{H}^1(C^+) + \sum_{j \in J} \mathcal{H}^1(C^-_j).
\]
Let $G =\text{int}(C^+)$ and 
$F_1 = G\setminus \Omega$. 
Notice that by construction, $F_1 \supset E_1$, $\partial F_1 \setminus \Si=C^+\setminus \Sigma$ is a single rectifiable curve in $\R^2 \setminus {\Omega}$  with endpoints on $\Sigma$, and $\partial F_1 \cap \Si = \Si \setminus C^+$ is a connected set with positive $\mathcal{H}^1$ measure.
Letting $\hat{C}^+:= C^+\setminus \Sigma$ and $\hat{C}^-_j = C^-_j \setminus \Si$ for each $j \in J$, it follows from the decomposition above  together with Federer's and De Giorgi's theorems that $P(E_1; \R^2\setminus \Omega) = \mathcal{H}^1(\hat{C}^+) + \sum_{j \in J}  \mathcal{H}^1(\hat{C}^-_j)$ and hence
\begin{equation}\label{eqn: holes energy bound}
    P(F_1; \R^2 \setminus{\Omega}) \leq P(E_1; \R^2 \setminus{\Omega})\,.
\end{equation}
To bound $|F_1\Delta E_1 |$, let  $a_j = |\text{int}( C^-_j)|$, so that $|F_1\Delta E_1|^{1/2} = (\sum_{j \in J} a_j)^{1/2}\leq \sum_{j \in J} a_j^{1/2}$.
From the lower bound \eqref{eqn: bounds on profile} on the isoperimetric profile, $(2\pi a_j)^{1/2} \leq \mathcal{H}^1(\hat{C}_j^-)$. So, as $|F_1|\geq |E_1|=\eta_1$ and hence $\IP(|F_1|)\geq \IP(\eta_1)$, we have
\begin{equation}
  \begin{split}
  \label{eqn: holes volume bound}
(2\pi)^\half|E_1 \Delta F_1|^\half &\leq P(E_1; \R^2 \setminus {\Omega}) - P(F_1; \R^2 \setminus {\Omega})\\
&\leq P(E_1; \R^2 \setminus {\Omega}) - \IP(\eta_1) \leq \delta_\eta(E) + (\IP(\eta) - \IP(\eta_1)).
  \end{split}  
\end{equation}
So, recalling from above that  $\eta_1 \geq \eta/2$, the local Lipschitz estimate \eqref{eqn: ftc profile} and \eqref{eqn: tiny components} show that  $\IP(\eta) - \IP(\eta_1) \leq (\pi/\eta_1)^{1/2} (\eta - \eta_1) \leq  C\, \delta_\eta(E)$ where $C=\pi^{1/2}/c_0$. 
Combining this with \eqref{eqn: holes volume bound}, \eqref{eqn: holes energy bound}
and \eqref{eqn: one component bound} shows the existence of $C=C(\kappa_{max}(\Sigma), \eta)$ such that 
 \begin{equation}\label{eqn: almost done}
     |E\Delta F_1|^\half \leq C\delta_\eta(E) \qquad \text{ and } \qquad P(F_1; \R^2 \setminus \Omega) \leq P(E; \R^2 \setminus \Omega).
 \end{equation}

 {\it Step 3:}
If $|F_1| \geq \eta,$ we complete the proof by
taking $F=F_1$.
Otherwise, as in step 1, take a ball of area $\eta - |F_1|$, and thus of perimeter at most $C\delta_\eta(E)$. Since $F_1$ is bounded, we may translate this ball from infinity along some ray so that it is disjoint from ${\Omega} \cup F_1$ and its boundary intersects $\partial F_1 \setminus \Si$. Then, by a slight deformation of $F_1 = F \cup B$ gives a set satisfying the conclusions of the lemma.
\end{proof}

\subsection{Reduction to a set bounded by a convex curve}\label{ssec: by hand reduction}
The next step toward \Cref{prop: summary reduction} is to replace a set of the type obtained in \Cref{lem: red 1 simply connected} by a relatively convex set with acute contact angle. First we prove the following elementary geometric lemma.
\begin{lemma} \label{lem: basic hausdorff dist est}
Fix $\bar{L}>0.$ There exists $C=C(\bar{L})>0$ such that the following holds.
    Let $\ga:[0,1] \to \R^2$ be a rectifiable curve with $L(\ga) \leq \bar{L}$ and let $\tilde{\ga}: [0,1] \to\R^2$ be a parametrization of the linear segment joining $\ga(0)$ and $\ga(1)$.
    \begin{equation}
        \label{eqn: hausdorff estimate}
    d_H(\ga, \tilde{\ga})^2 \leq C( L(\ga) - L(\tilde{\ga})). 
        \end{equation}
\end{lemma}
\begin{proof}
It suffices to bound  the distance between any point $z\in \gamma$ to its projection $\hat z$ onto the line through the endpoints $x_0$ and $x_1$ of $\gamma$.   Pythagoras, applied to the triangles $\Delta(x_0,z,\hat z)$ and $\Delta(x_1,z,\hat z)$, immediately gives the required bound of
\beqa 2 \abs{z-\hat z}^2&=\abs{z-x_0}^2-\abs{\hat z-x_0}^2+\abs{z-x_1}^2-\abs{\hat z-x_1}^2\\
&\leq 2\bar L \cdot (\abs{z-x_0}+\abs{z-x_1}-[\abs{\hat z-x_0}+\abs{\hat z-x_1}])
\leq 2\bar L\cdot (L(\gamma)-L(\tilde \gamma)).\eeqa
\end{proof}

\begin{lemma}\label{lem:handRed1}
Fix $\eta \in (0 , \areaC \kappamaxSi^{-2}]$. 
There are positive constants $\delta_3 = \delta_3 (\eta, \kappamaxSi)$ and $C_3= C_3(\eta, \kappamaxSi)$ such that the following holds. 
Let $E \subset \R^2 \setminus \Omega$ be a connected open set with $|E| \geq \eta$
 whose boundary is a rectifiable Jordan curve coinciding with $\Si$ on a connected, positive $\mathcal{H}^1$-measure set
 such that
       \begin{align}\label{eqn: hp convex}
P(E; \R^2 \setminus \Omega) \leq \IP(\eta) + \delta_3.
    \end{align} 
Then there is an open, relatively convex set $F\supset E$ such that $\partial F \setminus \Si$ meets $\Sigma$ with interior angles at most $\pi/2$ and
    \begin{align}\label{eqn: convex perim}
       P(F; \R^2 \setminus \Omega)  &\leq P(E; \R^2 \setminus \Omega),\\
    \label{eqn: convex L1 est} 
    d_H (\partial E, \partial F)^2 + |E\Delta F|  & \leq C_3 (  P(E; \R^2 \setminus \Omega)  - \IP(\eta)).
    \end{align}
\end{lemma}

\begin{proof}
Let   $\gamma:[0,1] \to \R^2$  be a  parametrization of $\partial E \setminus \Si$. Note that  $\gamma(0), \gamma(1) \in \Sigma$ and $\ga(p) \not \in \Omega$ for $p \in (0,1)$ and that by assumption $L(\gamma)\leq \bar L:= \IP(\eta)+\de_3$. 

 {\it Step 1:} Let $\tilde \Sigma \subset \Sigma$ be the set of points $x \in \Sigma$ for which the normal ray 
\begin{equation}
    \label{eqn: normal ray}
n_x : =\{y\in \mathbb{R}^2: y= x - t \nu_\Sigma(x) :  t \geq 0\}.
\end{equation}
  has nontrivial intersection with the trace of $\gamma$. Recall we orient $\Sigma$ positively so that $\nu_\Sigma$ is the inner normal of $\Sigma$.
  The set $\tilde\Sigma$ is connected thanks to the continuity of $\gamma$. Choose $\delta_3 \leq \delta_0$, where $\delta_0$ is from \Cref{lem: sat flow hp}, so that
 $L(\gamma)\leq { d_\Sigma}/{2}$ by \eqref{eqn: hp convex}, \Cref{lem: sat flow hp}, and Remark~\ref{rmk: d sigma bound}. Using this, a basic geometric argument shows that   $\tilde{\Sigma}$ is a proper subset of $\Sigma$
and the image of $\tilde{\Si}$ under the Gauss map of $\Si$ is a connected proper subset of a half circle of $\mathbb{S}^1.$

Let $j_0 < j_1$ be the endpoints of the interval $J$ for which $\tilde{\Sigma}$ is the trace of $\sigma$ restricted to the interval $J$, and let 
 $x_0= \si(j_0)$ and $x_1=\si(j_1)$.
Assume $\ga$ is oriented such that $j_0 \leq j_0' < j_1'\leq j_1$ where $\si(j_0') = \ga(0)$ and $\si(j_1')=\ga(1).$ With this orientation, we have $\underline{a}< \overline{a}$ where
\begin{align*}
 \underline{a} & = \sup \{ p \in [0, 1] : \gamma(p) \in n_{x_0}\} \\
 \overline{a} &  =\inf \{ p \in [0,1] : \ga(p) \in n_{x_1}\}.
\end{align*}
Note that $\gamma(\underline{a}) \in n_{x_0}$ and $\ga(\overline{a}) \in n_{x_1}$.  
Define the curve $\hat{\ga}:[0,1] \to \R^2 \setminus \Omega$ by letting $\hat{\g}=\g$ on $[\underline{a}, \overline{a}]$, and on the (possibly trivial) intervals $[0,\underline{a}]$ and $[\overline{a},1]$, letting $\hat{\gamma}$ parametrize the segments joining  $x_0$ to $\gamma(\underline{a})$ and joining $\gamma(\overline{a})$ to $x_1$ respectively.
The convexity of $\Sigma$ and a simple trigonometric argument show 
\begin{equation}\label{eqn: segment replacement}
L(\hat{\ga}) \leq L(\ga).
\end{equation}
Moreover, we claim that 
\begin{equation}\label{eqn: hausdorff est 1}
    d_H(\hat{\ga}, \ga)^2  \leq 3\bar{L} \, (L(\ga) - L(\hat{\ga}))\,.
\end{equation}
To see \eqref{eqn: hausdorff est 1}, first let $\tilde{\ga}:[0,1] \to \R^2 \setminus\Omega$ be the curve that is equal  to $\ga$ on $[0,1]$, joins $\ga(0)$ to $\ga(\underline{a})$ linearly on $[0,\underline{a}]$, and joins $\ga(\overline{a})$ to $\ga(1)$ linearly on $[\overline{a},1]$. It is simple to see that $d_H(\hat{\ga}, \tilde{\ga
})^2 \leq L(\tilde{\ga})^2 - L(\hat{\ga})^2 \leq 2\bar{L}(L(\tilde{\ga}) - L(\hat{\ga}))$. Next, by \Cref{lem: basic hausdorff dist est},
$d_H(\tilde{\g}, \g)^2 \leq \bar L(  L(\g) - L(\tilde{\g}) )$. Combining these two bounds yields \eqref{eqn: hausdorff est 1}.

Let $\hat{E}$ be the set bounded by $\hat{\g}$ and the segment joining $x_0$ to $x_1$. As $\abs{\hat{E}\setminus \Om}\geq \abs{E\setminus \Om}\geq \eta$, we have $P(\hat{E} ; \R^2 \setminus{\Omega}) = L(\hat\gamma)\geq \IP(\eta),$ so \eqref{eqn: segment replacement} and \eqref{eqn: hausdorff est 1} imply
\begin{equation}
    \label{eqn: summary 1}
P(\hat{E} ; \R^2 \setminus{\Omega}) \leq P(E;\R^2 \setminus{\Omega}), \qquad d_H(\partial E , \partial (\hat{E} \setminus \Omega))^2
\leq C (  P(E; \R^2 \setminus \Omega)  - \IP(\eta)) .
\end{equation}

Note the $E\subset \hat{E}$ and that $\hat{E}$ is contained in the convex region $\mathcal{K}$ bounded by $n_{x_0}, n_{x_1}$, and the segment joining $x_0$ to $x_1$.

{\it Step 2:} 
Next, let $\hat{F}$ be the convex hull of $\hat{E}$ in $\R^2$. Then $\hat{F}\supset \hat{E}$ and $P(\hat{F}) \leq P(\hat{E})$ (this classical fact is shown in the context of indecomposable sets of finite perimeter in \cite[Theorems 1 and 6]{FerrFusco}). So, since  $\hat{F}\cap {\Omega} = \hat{E}\cap {\Omega}$ by construction, the relatively convex set 
 $F = \hat{F} \setminus {\Omega}$ satisfies
\begin{equation}\label{eqn: sum 2}
    P({F}; \R^2\setminus {\Omega}) \leq P(\hat{E}; \R^2\setminus{\Omega})\,.
\end{equation}
Moreover, since $\partial F \setminus \Si$ is locally linear where it is not contained in $\partial \hat{E}$, an application of \Cref{lem: basic hausdorff dist est} shows that 
\begin{equation}\label{eqn: sum 3}
    d_H(\partial F , \partial (\hat{E} \setminus\Omega))^2  \leq C(P(\hat{E}; \R^2\setminus {\Omega}) -  P({F}; \R^2\setminus {\Omega}))  \leq C (  P(E; \R^2 \setminus \Omega)  - \IP(\eta)) \leq C  \delta_3
   .
\end{equation}
The final inequality comes from \eqref{eqn: hp convex} while the penultimate inequality uses the fact that $|F| \geq \eta$ and thus $P({F}; \R^2\setminus{\Omega})  \geq \IP(\eta)$.
Finally, $\hat{F}$ is also contained in the convex region $\mathcal{K}$. So, the rectifiable curve parametrizing $\partial F\setminus \Si$ meets $\Si$ at the points $x_0$ and $x_1$ with interior angle at most $\pi/2$. Combining \eqref{eqn: summary 1}, \eqref{eqn: sum 2}, and \eqref{eqn: sum 3}, we obtain \eqref{eqn: convex perim} and the Hausdorff distance estimate of \eqref{eqn: convex L1 est}.

{\it Step 3:} It remains to show the bound on the symmetric difference in \eqref{eqn: convex L1 est} above.
 Let $\e = |F\Delta E| =  |F|-|E|$.
 Since the isoperimetric profile is a nondecreasing function of $\eta$ and $|F|\geq \eta +\e $, we have
  $ P(E, \R^2 \setminus \Omega)  \geq  P(F, \R^2 \setminus \Omega) \geq \IP(\eta+\e) \geq \IP(\eta)$. Combining this with the lower bound from  \eqref{eqn: ftc profile} yields
 \[
   \Big(\mfrac{\pi}{2({\eta}+\e)}\Big)^\half \e \leq   \IP({\eta} + \e) -\IP({\eta} ) \leq  P(E, \R^2 \setminus \Omega)  - \IP(\eta).
 \]
 So, the desired estimate holds provided we bound $ (\frac{\pi}{2({\eta}+\e)})^\half$ below by a constant depending only on $\eta$ and $\kappamaxSi$. To this end, let   $G = F\setminus E$, so $|G| =\e$. By e.g. \cite[Theorem 16.3]{MaggiBOOK}, we have $ P(G;\R^2 \setminus {\Omega}) \leq P(E ; \R^2 \setminus {\Omega}) + P(F;  \R^2 \setminus {\Omega})$. 
 So, applying the lower and upper bounds of  \eqref{eqn: bounds on profile} and recalling \eqref{eqn: hp convex} and \eqref{eqn: convex perim}, we obtain
\begin{align*}
 (2\pi \e)^\half \leq   \IP(\e)& \leq P(G;\R^2 \setminus {\Omega}) \leq 2(\IP(\eta) + \delta_3)  \leq 4(\pi \eta)^\half + 2\delta_3.
\end{align*}
This completes the proof.
\end{proof}

\subsection{Proof of \Cref{prop: summary reduction}} \label{ssec: proof of summary reduction}
We now  combine the results of the previous two subsections with a final area-correction step to show \Cref{prop: summary reduction}.
\begin{proof}[Proof of \Cref{prop: summary reduction}]
Let
$\delta_2$ and $C_2$ be as in  \Cref{lem: red 1 simply connected} and let $\delta_3$ and $C_3$ be as in \Cref{lem:handRed1}. Let 
$\delta_1 = \min\{ \delta_2,\delta_3/C_2, 1\}$. Applying \Cref{lem: red 1 simply connected}, we obtain a set $F_1$ satisfying the assumptions of \Cref{lem:handRed1} with
       \begin{align}\label{eqn: hand red final 1}
   |E\Delta F_1|^{\half}   \leq C_2 \delta_\eta(E) \quad \text{ and } \quad        P(F_1; \R^2 \setminus \Omega) \leq P(E ; \R^2 \setminus \Omega) + C_2\delta_\eta(E)\,.
    \end{align} 
Next, applying \Cref{lem:handRed1} to $F_1$, we obtain an open, relatively convex set $F_2\supset F_1$ such that $\partial F_2 \setminus \Si$ meets $\Sigma$ with interior angles at most $\pi/2$ and
    \begin{align}
       P(F_2; \R^2 \setminus \Omega)  &\leq P(F_1; \R^2 \setminus \Omega),\\ 
    d_H (\partial F_1, \partial F_2)^2 + |F_1\Delta F_2|  & \leq C_3 (  P(F_1; \R^2 \setminus \Omega)  - \IP(\eta)) \leq C \delta_\eta(E)
    \end{align}
    where $C= C_3(C_2+1)$.
Combining this with \eqref{eqn: hand red final 1}, we see that 
\begin{equation}
    \label{eqn: hand red final 2}
|E\Delta F_2| \leq C \delta_\eta(E), \qquad P(F_2 ; \R^2 \setminus \Omega )\leq P(E; \R^2 \setminus\Omega ) + C\delta_\eta(E).
\end{equation}
The set $F_2$ has $|F_2| \geq \eta$ by construction. If $|F_2|=\eta,$ we let $F=F_2$ and see that \eqref{eqn: reduction bound} holds. Otherwise, let $j_0 < j_1$ be chosen such that $ \si([j_0, j_1]) =\partial F_2 \cap \Si.$
  For $j \in [j_0, j_1],$ let $F_j$ be the intersection of $F_2$ with the convex region $\mathcal{R}_j$ bounded by the normal rays $n_{\sigma(j)}, n_{\si(j_1)}$, and the segment joining $\si(j)$ to $\si(j_1)$. The area of $F_j$ varies continuously in $j$ with $|F_{j_0}|>\eta$ and $|F_{j_1}|= 0$, so we may find $j\in [j_0, j_1]$ such that $F:=F_{j}$ has area $\eta.$ Thanks to the convexity of $\Si$, we immediately have 
  \begin{equation}
      P(F; \R^2 \setminus\Omega) \leq P(F_2; \R^2 \setminus \Omega),
       \end{equation}
       and by construction and \eqref{eqn: hand red final 2}, we have $|F\Delta F_2| = | F_2 | -\eta \leq  C\delta_\eta(E)$. Combining these estimates with \eqref{eqn: hand red final 2} yields \eqref{eqn: reduction bound}.

Finally, assume that $\partial E\setminus \Sigma$ is a rectifiable curve with endpoints on $\Si$. The same argument used in the proof of \Cref{lem: red 1 simply connected} shows that $\partial E$ intersects $\Si$ on a positive $\mathcal{H}^1$-measure set. Thus, there was no need to apply \Cref{lem: red 1 simply connected} because the set $E$ already met the hypotheses of \Cref{lem:handRed1}. Hence we may take $F_1:=E$ in the above argument, still have \eqref{eqn: hand red final 2} and now additionally obtain from 
\Cref{lem:handRed1}
that $  d_H(\partial E, \partial F_2 )^2 \leq C_3\de_\eta(E)$.
  Next, the same argument used in step 1 of \Cref{lem:handRed1} using \Cref{lem: basic hausdorff dist est} shows that 
  \[
  d_H(\partial F, \partial F_2 )^2 \leq C(P(F_2 ; \R^2 \setminus\Omega) - P(F; \R^2 \setminus \Omega))\leq C\delta_\eta(E) .
  \]
  Together with \eqref{eqn: hand red final 2}, this shows \eqref{eqn: reduction bound 2}. This completes the proof. 
\end{proof}

\appendix
\section{Proof of \eqref{est:length lower bound}}\label{app: proof of remark}

\begin{proof}[Proof of \eqref{est:length lower bound}] 
{Fix $\ga\in \B_\eta$. It suffices to consider the case when $\g$ is oriented so that $A_\Si(\eta)>0$. We further assume without loss of generality that $\g$ is parametrized by arclength.} 
Fix any small  $\epsilon>0$. Since $\gamma$ stays outside of $\Omega$ away from the two endpoints and is defined on a compact set, we may obtain an approximation $\gamma_\epsilon \in \mathcal B$ of $\gamma$ such that
\begin{enumerate}
\item \label{item:approx req 1} $|L(\gamma) - L(\gamma_\epsilon)|<\epsilon$ and $|A_\Sigma(\gamma) - A_\Sigma(\gamma_\epsilon)|<\epsilon$, and moreover
\item \label{item:approx req 2} the image of $\gamma_\epsilon$ is the union of finitely many piecewise $\mathcal{C}^2$ curves $\gamma_0, \gamma_1, \cdots \gamma_k$ where $\gamma_0 \in \mathcal B$ is embedded and $\gamma_i$ are closed and embedded for $i=1, \cdots, k$, where for $i,j \in \{0, \cdots,k\}$, $\gamma_i$ and $\gamma_j$ do not intersect except possibly meeting transversally at their endpoints.
\end{enumerate}
As $\gamma$ is of class $H^2$, and hence $C^1$,
such a $\gamma_\epsilon$ can be easily constructed to be in fact even piecewise linear;
 taking a fine enough subdivision $t_0=0<t_1< \cdots < t_i = i/k < \cdots<t_k=L(\gamma)$ of $[0,L(\gamma)]$ for a large integer $k$ and replacing $\tilde \gamma|_{[t_i, t_{i+1}]}$ with the line segment connecting $\tilde \gamma(t_i)$ and $\tilde \gamma(t_{i+1})$, we can get a piecewise linear curve which satisfies \eqref{item:approx req 1}. Then, up to slightly perturbing the vertices of the piecewise linear $\gamma_\epsilon$ to avoid any overlapping segments, $\gamma_\epsilon$ can be taken to satisfy \eqref{item:approx req 2}. Denote by $E_0$ the region bounded by $\gamma_0$ and $\Sigma$, and $E_k$ the region bounded by $\gamma_i$ for $i=1, \cdots, k$. By the usual isoperimetric inequality for closed curves we have $L(\gamma_i)^2 \ge 4\pi |E_i|$, and $L(\gamma_0) \ge I_\Omega(|E_0|)$. Therefore,
\[L(\gamma_\epsilon) = \sum_{i=0}^{k}L(\gamma_i) \ge I_\Omega(|E_0|) + \sum_{i=1}^{k} \sqrt{4\pi |E_i|}.\]
On the other hand, note that $|E_0| + |E_1| + \cdots + |E_k| \ge \eta-\epsilon$ since the $E_i$'s are counted with a sign in the algebraic area of $\gamma_\epsilon$. By \eqref{eqn: bounds on profile} we have that $\sqrt{4\pi a} \ge I_\Omega(a)$ for any $a>0$.  We can now use that the isoperimetric profile is sub-additive in the sense that $I_\Omega(a) + I_\Omega(b) \ge I_\Omega(a+b)$ for any $a,b>0$; see Lemma \ref{lem: quant sublinearity} for  a more general and quantitative version of this statement. Because $I_\Omega(a)$ is nondecreasing in $a$, combining all of the above we have that
\[L(\gamma_\epsilon) \ge I_\Omega(|E_0|) + \sum_{i=1}^{k} \sqrt{4\pi|E_i|} \ge \sum_{i=1}^{k} I_\Omega(|E_i|) \ge I_\Omega(|E_0| + \cdots + |E_k|) \geq I_\Omega(\eta-\epsilon).\]
Taking $\epsilon \to 0$ finishes the proof.
\end{proof}
\section{Proof of \Cref{lem: mollify}}\label{sec: technical appendix}
\begin{proof}[Proof of \Cref{lem: mollify}] {\it Step 1:} 
    Let $\ga:[0,1] \to \R^2 \setminus \Omega^\circ$ be a constant speed parametrization of $\partial F \setminus \Si$, so that $\ga(0), \ga(1) \in \Si$ and $\ga(t) \not \in \Si$ for $t \in (0,1)$. 
    For $N \in \mathbb{N}$ large to be fixed later, let $p_j =j/N$ for $j=0,\dots, N$ and let $\ga_1:[0,1]\to \R^2 $ be the polygonal curve defined as follows.
    For $j =1,\dots, N-2$, define $\ga_1|_{[p_j, p_{j+1}]}$ to be the  constant speed linear interpolation from $\ga(p_j)$ to $\ga(p_{j+1})$.
    Let $\ga_1(0)$ and $\ga_1(1)$  be the nearest point projections of $\ga(p_1)$ and $\ga(p_{N-1})$ on $\Si$ respectively. Define $\ga_1|_{[p_0,p_1]}$ as the constant speed linear interpolation from $\ga_1(0)$ to $\ga(p_1)$ and  $\ga_1|_{[p_{N-1},p_N]}$ as the constant speed linear interpolation from $\ga(p_{N-1})$ to $\ga(1)$.  By construction, $\ga_1$ is a piecewise linear convex curve whose endpoints meet $\Sigma$ orthogonally. Provided $N$ is chosen sufficiently large, $\ga_1(p)$ lies outside ${\Omega}$ for all $p\in(0,1)$. Together with $\Sigma$, $\gamma_1$ bounds an open and relatively convex set $F_1 \subset F.$ The errors
    \begin{equation}
            \label{eqn: est 1}
                d_H(\partial F, \partial F_1), \quad |F\setminus F_1|, \quad |L(\ga) - L(\ga_1)|
    \end{equation}
    can be made arbitrarily small by choosing $N$ sufficiently large.

    {\it Step 2:} The curve $\ga_1$ is smooth away from the corners at $p_1, \dots, p_{N-1}.$ We can smooth each of these corners in a $C^{2,1}$ fashion as follows. Choosing $\sigma \ll 1/N$, let $\ga_2 : [0,1] \to \R^2$ be the curve that is equal to $\ga_1$ outside of $\cup_{j-1}^{N-1} [p_j-\sigma, p_j + \si]$ and such that for each $j=1,\dots N-1$, $\ga_2|_{[p_j-\sigma, p_j + \si]}$ is defined as the cubic B\'{e}zier curve with parameters chosen so that, at the endpoints $p_j-\si$ and $p_j+\si$, the tangents match those of $\ga_1$ and the curvature vanishes. By construction, $\ga_2$ is a convex curve, and provided $\si$ is chosen sufficiently small, $\ga_2(p)$ lies outside of ${\Omega}$ for each $p\in(0,1)$. Since $\ga_1$ and $\ga_2$ agree in neighborhoods of their endpoints, $\ga_2$ meets $\Si$ orthogonally at its endpoints. Together with $\Sigma$, $\gamma_2$ bounds an open, relatively convex set $F_2 \subset F_1$, and the errors
    \begin{equation} \label{eqn: est 2}
    d_H(\partial F_2, \partial F_1), \quad |F_1\setminus F_2|, \quad |L(\ga_2) - L(\ga_1)|
       \end{equation}
    can be made arbitrarily small by choosing $\si$ sufficiently small.

    {\it Step 3:}
  Reparametrize $\ga_2$ on $[0,1]$ with orientation such that the normal $\nu_{\ga_2}$ to $\ga_2$ coincides with the outward unit normal to $F_2.$  
    For $\rho \geq 0$ to be chosen later, define $\ga_{3,\rho}: [0,1] \to \R^2$ as follows. For $p \in [\sigma, 1-\sigma]$, let 
   \begin{equation}
   \ga_{3,\rho} (p) = \ga_2(p) + \rho \nu_{\ga_2}(p)\,, 
     \end{equation}
     which is a $C^{2,1}$, embedded convex curve.
    Let $\ga_{3,\rho}(0)$ be the nearest point projection of $\ga_{3,\rho}(\rho)$ on $\Si$
    and define $\ga_{3,\rho}|_{[0,\sigma]}$ to be the constant speed linear interpolation from $\ga_{3,\rho}(0)$ to $\ga_{3,\rho}(\sigma)$.
    Define $\ga_{3,\rho}(1)$ and $\ga_{3,\rho}|_{[1-\sigma, 1]}$ analogously.
    By construction, $\ga_{3,\rho}(p) \in \R^2 \setminus{\Omega}$ for all $p \in (0,1)$ and $\g_{3,\rho}$ meets $\Si$ orthogonally. 

 Let $F_3=F_{3,\rho}$ be the open, connected region bounded by $\g_{3,\rho}$ and $\Si$.   
  We claim that  $F_3$ is relatively convex, provided $\delta_0$, $\rho$, $\sigma$, and $1/N$ are sufficiently small. To this end, we will show the set $G_{3,\rho}$ bounded by $\g_3$ and the segment joining $\g_3(0)$ and $\g_3(1)$ is convex. 
  First notice 
 that, by the convexity of $\Omega$ and orthogonal contact angle between $\Si$ and the segment $\gamma_2\vert_{[0,1/N-\sigma]}$,
 the  linear extension of the shifted segment $\ga_{3,\rho}|_{[{\sigma},1/N-{\sigma}]}$, which intersects $\Si$ if $\rho$ is small, 
 has contact angle at most $\pi/2$, and likewise for the other side.
 Consequently, the interior angles of $\ga_{3,\rho}$  at $p=\sigma$ and $p=1-\sigma$ are at most $\pi$.
 Up to replacing $\g_{3,\rho}$ by the curve obtained by running the corner-smoothing procedure (with a smaller $\sigma'$) of Step 2 at $p=\sigma$ and $p=1-\sigma$, we may also assume $\g_{3,\rho}$ is $C^{2,1}$.
 
 Consequently, $G_{3,\rho}$ is convex provided the turning angle of $\g_{3,\rho}$ is at most $2\pi$, or equivalently (given the orthogonal contact angle of $\g_{3,\rho}$), if the set of normals $A_\rho=\{ \nu_\Si(x) : x \in \Si \cap \partial F_{3,\rho}\}$ to $\Si$ lies in a half-circle of $\mathbb{S}^1.$
Choose $\delta_0>0$ according to \Cref{lem: sat flow hp}. Since $\delta_\eta(F) \leq \delta_0$, \Cref{lem: sat flow hp} and \Cref{rmk: d sigma bound} guarantee that the endpoints $\ga(0)$ and $\ga(1)$ of $\g$ cannot be antipodal points of $\Si$, and that the set of normals $\{ \nu_\Si(x) : x \in \Si \cap \partial F\}$ lies in a {\it strict} subset of a half-circle of $\mathbb{S}^1.$ 
By the continuity of the construction, 
 there exist $\bar{N}$ and $\bar{\rho}$ such that the same holds for $A_\rho$ provided  $\rho \leq \bar{\rho}$, $N\geq \bar{N}$, and $\sigma$ is small enough depending on $N$ as in Step 2.
 This yields the desired convexity.

Now, the area $|F_{3,\rho}|$ varies continuously and is monotonically increasing with respect to $\rho$. Moreover, there exists $\hat{N}$ such that for any choice of parameters $N>\hat{N}$ and $\sigma\ll 1/N$, there exists $\rho \leq \bar{\rho}$ such that  $|F_{3,\rho}| >\eta$. Since $F_{3,0} = F_2 \subset F$ has area at most $\eta,$ for each  $N$ (and $\sigma$ depending on $N$ as in step 2) we may choose $\rho_0$ such that $|F_{3,\rho_0}| = \eta.$ Let $\ga_3 = \ga_{3, \rho_0}$. From the construction we see that $\rho_0\to 0$ as $1/N \to 0$ and that 
    \begin{equation}
        \label{eqn: est 3a}
         d_H(\gamma_2, \ga_3), \quad |F_2\Delta  F_3|, \quad |L(\ga_2) - L(\ga_3)|
    \end{equation}
    can be made arbitrarily small by choosing $1/N$ (and thus $\rho_0$) sufficiently small. Choosing $1/N$ small enough depending on $\e$ and taking $F_\e = F_3$, the proof follows by combining \eqref{eqn: est 1}, \eqref{eqn: est 2}, and \eqref{eqn: est 3a}.
\end{proof} 

\bibliography{literature}
\bibliographystyle{plain}

\end{document}